\documentclass[12pt]{article}
\usepackage[dvips]{epsfig}
\usepackage{graphics}
\usepackage{amsmath}
\usepackage{amsfonts}
\usepackage{amssymb}
\usepackage{amsbsy}
\usepackage{amsthm}
\usepackage{dsfont}
\usepackage{textcomp}
\usepackage{bm}
\usepackage{enumerate}
\usepackage[mathscr]{eucal}
\usepackage[T1,OT1,TS1,T2A]{fontenc}
\usepackage[cp1251]{inputenc}
\usepackage[russian,english]{babel}
\usepackage[sf,bf]{titlesec}
\titleformat{\section}[hang]{\large\upshape\bfseries}{\large\bfseries\thesection.}{0.7em}{}
\theoremstyle{plain}
\newtheorem{theorem}{Theorem}[section]
\newtheorem{lemma}{Lemma}[section]
\newtheorem{definition}{Definition}[section]
\newtheorem{remark}{Remark}[section]

\newtheorem{corollary}{Corollary}[section]
\newtheorem{proposition}{Proposition}[section]
\newtheorem{notation}{Notation}[section]
\newtheoremstyle{clum}%
{5pt}{5pt}%
{\itshape}%
{0.0ex}{\bfseries}%
{}%
{1.5ex}%
{\thmname{#1}\thmnumber{#2}\thmnote{#3}}%
\theoremstyle{clum}
\newtheorem*{nonumtheorem}{Theorem}
\theoremstyle{plain}%
 \numberwithin{theorem}{section}
 \numberwithin{equation}{section}

\begin{document}
\title{
\textbf{\LARGE
The spectral theory of the Fourier operator
truncated on the positive half-axis.}}
\author{\Large\textbf{Victor Katsnelson}.}
\date{\ }%
\maketitle
\begin{quote}

\end{quote}

\abstract{\noindent The spectral analysis of the operator Fourier
truncated on the positive half-axis is done.
 }

\vspace{2.0ex}%
\noindent%
\textbf{\normalsize Notation and terminology.}\\
\(\mathbb{R}\) stands for the real axis.\\
\(\mathbb{R}^{+}\) stands for the positive real half-axis:\ %
\(\mathbb{R}^{+}=\{t\in\mathbb{R}:\,t>{}0\}.\)\\
\(\mathbb{C}\) stands for the complex plane.\\
\(\mathbb{N}\) stands for the set of natural numbers, \(\mathbb{N}=\{0,\,1,\,2,\,\ldots\,\,\,\}\)\\
\(\mathfrak{M}_{p,q},\,p\in\mathbb{N},\,q\in\mathbb{N}\), is the set of all \(p\times{}q\) matrices
(\(p\) rows,\,\(q\) columns) with entries from \(\mathbb{C}\).

Assume that \((\mathscr{X},\mathscr{B},m)\) is a measurable space
(\(\mathscr{B}\) is a sigma algebra of subsets of \(\mathscr{X}\),
\(m\) is a non-negative measure defined on \(\sigma\)). If
\(f(t):\,\mathscr{X}\to\mathbb{R}\) is a
\(\mathscr{B}-measurable\) real valued function on
\(\mathscr{X}\), then
\begin{equation*}
\underset{\mathscr{X}}{\textup{ess\,sup}}\,f(t)=\inf_{\substack{E\in\mathscr{B}:\\
m(E)=0}}\sup_{t\in\mathscr{X}\setminus{}E}f(t),\quad
\underset{\mathscr{X}}{\textup{ess\,inf}}\,f(t)=\sup_{\substack{E\in\mathscr{B}:\\
m(E)=0}}\inf_{t\in\mathscr{X}\setminus{}E}f(t)\,.
\end{equation*}

Assume that the set \(\mathscr{X}\) carries two structures:
\(\mathscr{X}\) is a metric space provided by the metric \(d\),
and \(\mathscr{X}\) is a measurable space, where \(m\) is a
non-negative measure defined on the sigma algebra of all Borelian
subsets of \(\mathscr{X}\). If \(a\) is a point of \(\mathscr{X}\)
and \(B\) is a subset of \(\mathscr{X}\), \(B\in\mathscr{B}\),
then
\begin{equation*}
\textup{ess\,dist}\,(a,B)\stackrel{\textup{\tiny
def}}{=}\underset{x\in{}B}{\textup{ess\,inf}}\,\,d(a,x)\,.
\end{equation*}
For a set \(B\), \(B\in\mathscr{B}\), its essential closure
\(\textup{clos}_e(B)\) is defined as
\begin{equation*}
\textup{clos}_e(B)=\{a\in\mathbb{X}:\,\textup{ess\,dist}\,(a,B)=0\}\,.
\end{equation*}
If the set \(E\), \(E\in\mathscr{B}\), is of zero measure:
\(m(E)=0\), we write \(E=\emptyset_e\).

In what follows, the measurable space \(\mathscr{X}\) is an
interval of a straight line in the complex plane, \(\mathscr{B}\)
is the sigma-algebra of Borelian subsets of this straight line,
and \(m\) is the linear (one-dimensional) Lebesgue measure on this
\(\mathscr{B}\).

 \setcounter{section}{0}
\section{The Fourier operator truncated on positive half-axis.}

In this paper we study the truncated Fourier operator
\(\mathscr{F}_{\mathbb{R^{+}}}\),
\begin{equation}%
\label{DTFTr}
(\mathscr{F}_{\mathbb{R^{+}}}x)(t)=\frac{1}{\sqrt{2\pi}}%
\int\limits_{\mathbb{R^{+}}}x(\xi)e^{it\xi}\,d\xi\,,\ \ \
t\in{}{\mathbb{R^{+}}}.
\end{equation}%
The operator \(\mathscr{F}_{\mathbb{R^{+}}}\) is considered as an
operator acting in the space \(L^2(\mathbb{R^{+}})\) of all square
measurable complex valued functions on \(\mathbb{R^{+}}\) provided
with the scalar product
\begin{equation*}
\langle{}x,y\rangle=\int\limits_{\mathbb{R^{+}}}x(t)\overline{y(t)}\,dt\,.
\end{equation*}
The operator \(\mathscr{F}^{ \ast}_{\mathbb{R^{+}}}\) adjoint to
the operator \(\mathscr{F}_{\mathbb{R^{+}}}\) with respect
 to this scalar product is
\begin{equation}%
\label{DTFTrA}
(\mathscr{F}_{\mathbb{R^{+}}}^{\ast}x)(t)=\frac{1}{\sqrt{2\pi}}%
\int\limits_{\mathbb{R^{+}}}x(\xi)e^{-it\xi}\,d\xi\,,\ \ \
t\in{}{\mathbb{R^{+}}}\,.
\end{equation}%
The operator \(\mathscr{F}_{\mathbb{R^{+}}}\) is the operator of
the form
\begin{equation}%
\label{UnDil}%
 \mathscr{F}_{\mathbb{R^{+}}}=
P_{\mathbb{R^{+}}}\,\mathscr{F}\,P_{\mathbb{R^{+}}|_{{\scriptstyle
L}^2{\scriptstyle({\mathbb{R}}^{+})}}}\,,
\end{equation}%
where \(\mathscr{F}\) is the Fourier operator on the whole real
axis:
\begin{equation}
\label{FWRA}
(\mathscr{F}x(t)=\frac{1}{\sqrt{2\pi}}%
\int\limits_{\mathbb{R}}x(\xi)e^{it\xi}\,d\xi\,,\ \ \
t\in{}{\mathbb{R}},
\end{equation}
\begin{equation*}
\mathscr{F}:\,L^2(\mathbb{R})\to{}L^2(\mathbb{R})\,,
\end{equation*}
and \(P_{_{\scriptstyle\mathbb{R^{+}}}}\) is the natural
orthogonal projector from \(L^2(\mathbb{R})\) onto
\(L^2(\mathbb{R}^{+})\):
\begin{equation}
\label{NaPr}%
 (P_{\mathbb{R^{+}}}x)(t)=\mathds{1}_{_{\scriptstyle\mathbb{R}^{+}}}\!(t)\,x(t),\ \
 x\in{}L^2(\mathbb{R}),
\end{equation}
 \(\mathds{1}_{_{\scriptstyle\mathbb{R_{+}}}}\!(t)\) is the indicator
function of the set \(\mathbb{R}^{+}\). For any set \(E\), its
indicator function \(\mathds{1}_{_{\scriptstyle E}}\) is
\begin{equation}
\label{IndF}
\mathds{1}_E(t)= \begin{cases}1,&\ \textup{if}\ t\in{}E,\\
0,&\ \textup{if}\ t\not\in{}E.
\end{cases}
\end{equation}
It should be mention that the Fourier operator \(\mathscr{F}\) is
an unitary operator in \(L^2(\mathbb{R})\):
\begin{equation}
\label{Unit}
\mathscr{F}^{\ast}\mathscr{F}=\mathscr{F}\mathscr{F}^{\ast}=
\mathscr{I}_{L^2(\mathbb{R})},
\end{equation}
\(\mathscr{I}_{L^2(\mathbb{R})}\) is the identity operator in
\(L^2(\mathbb{R})\).%

From \eqref{UnDil} and  \eqref{Unit} it follows that the operators
\(\mathscr{F}_{\mathbb{R^{+}}}\) and
\(\mathscr{F}_{\mathbb{R^{+}}}^{\ast}\) are contractive: %
\(\|\mathscr{F}_{\mathbb{R^{+}}}\|\leq1\),
\(\|\mathscr{F}_{\mathbb{R^{+}}}^{\ast}\|\leq1\). We show later
that actually
\begin{equation}
\label{NorTrF}%
 \|\mathscr{F}_{\mathbb{R^{+}}}\|=1,\quad
\|\mathscr{F}_{\mathbb{R^{+}}}^{\ast}\|=1\,.
\end{equation}
Nevertheless, these operators are strictly contractive:
\begin{equation}
\label{StrCon}%
 \|\mathscr{F}_{\!_{\scriptstyle\mathbb{R_{+}}}}x\|<\|x\|,\quad
 \|\mathscr{F}^{\ast}_{\!_{\scriptstyle\mathbb{R^{+}}}}x\|<\|x\|\,, \quad
 \forall\,x\in{}L^2(\mathbb{R}^{+}),\,x\not=0\,,
\end{equation}
and their spectral radii
\(r(\mathscr{F}_{\!_{\scriptstyle\mathbb{R^{+}}}})\) and
\(r(\mathscr{F}^{\ast}_{\!_{\scriptstyle\mathbb{R^{+}}}})\) are
less that one:
\begin{equation}
\label{SpRad}%
r(\mathscr{F}_{\!_{\scriptstyle\mathbb{R^{+}}}})=
r(\mathscr{F}^{\ast}_{\!_{\scriptstyle\mathbb{R^{+}}}})=1/\sqrt{2}.
\end{equation}
In particular, the operators
\(\mathscr{F}_{\!_{\scriptstyle\mathbb{R^{+}}}}\) and
\(\mathscr{F}^{\ast}_{\!_{\scriptstyle\mathbb{R^{+}}}}\) are
contractions of the class \(C_{00}\) in the sense of \cite{SzNFo}.

In \cite{SzNFo}, a spectral theory of contractions in a Hilbert
space is developed. The starting point of this theory is the
representation of the given contractive operator \(A\) acting is
the Hilbert space \(\mathscr{H}\) in the form
\begin{equation}
\label{ComCon}%
 A=PUP,
\end{equation}
 where
\(U\) is an unitary operator acting is some \emph{ambient} Hilbert
space \(\mathfrak{H}\), \(\mathscr{H}\subset\mathfrak{H}\), and
\(P\) is the orthogonal projector from the whole space
\(\mathfrak{H}\) onto its subspace \(\mathscr{H}\). In the
construction of \cite{SzNFo} there is required that not only the
equality \eqref{ComCon} but also the whole series of the
equalities
\begin{equation}
\label{Dilat}%
 A^n=PU^nP,\quad n\in\mathbb{N},
\end{equation}
hold. The unitary operator \(U\) acting in the ambient Hilbert
space \(\mathfrak{H}\), \(\mathcal{H}\subset\mathfrak{H}\), is
said to be \emph{the unitary dilation of the operator \(A\),}
\(A:\mathcal{H}\to\mathcal{H}\), if the equalities \eqref{Dilat}
hold. In  \cite{SzNFo} it was shown that every contractive
operator \(A\) admits an unitary dilation. Using the unitary
dilation, a functional model of the operator \(A\) is constructed.
This functional model is an operator acting in some Hilbert space
of analytic functions. The functional model of the operator \(A\)
is an operator which is unitary equivalent to \(A\). The spectral
theory of the original operator \(A\) is developed by analyzing
its functional model.

The relation \eqref{UnDil} is of the form \eqref{ComCon}, where
\(\mathscr{H}=L^2(\mathbb{R}^{+})\),
\(\mathfrak{H}=L^2(\mathbb{R})\),
\(U=\mathscr{F},\,A=\mathscr{F}_{_{\scriptstyle\mathbb{R}_{+}}}\),
\(P=P_{\mathbb{R^{+}}}\) is an orthoprojector from \(L^2(\mathbb{R})\) onto
\(L^2(\mathbb{R}^{+})\), \eqref{NaPr}.
However, for these objects the equalities \eqref{Dilat} do not
hold for all \(n\in\mathbb{N}\), but only for \(n=0,\,1\). So, the
operator \(\mathscr{F}\) is not an unitary delation of its
truncation \(\mathscr{F}_{_{\scriptstyle\mathbb{R}_{+}}}\).
Nevertheless, we succeeded in constructing such a functional model
of the operator \(\mathscr{F}_{_{\scriptstyle\mathbb{R}_{+}}}\)
which is easily analyzable. Analyzing this model, we develop the
complete spectral theory of the operator
\(\mathscr{F}_{_{\scriptstyle\mathbb{R}^{+}}}\).

\section{Operator calculus for the truncated Fourier operator \mathversion{bold}%
\({\bm{\mathscr{F}}}_{{\!\scriptstyle{\bm{\mathbb{R}}}^{+}}}\).}
\mathversion{normal}

In this section, we present the formulations of some of our main results.
Proofs will be presented later, in the next sections.

First we set the terminology and notation related to the notion of spectrum of
a linear operator.

Let \(A\) be a linear operator acting in the Hilbert space
\(\mathscr{H}\), with the domain of definition \(\mathcal{D}_A\).
(\(\mathcal{D}_A\) is assumed to be a linear subspace of
\(\mathscr{H}\), not necessary closed.) The \emph{resolvent set}
\(\textup{\large\(\rho\)}(A)\) of the operator \(A\) is the set of
all complex numbers \(\lambda\) for which the operator
\(\lambda\,\mathscr{I}~-~A\) maps one-to-one \(\mathcal{D}_A\) onto
the whole \(\mathscr{H}\) and the inverse operator
\((\lambda\,\mathscr{I}-A)^{-1}:\,\mathscr{H}\to\mathcal{D}_A\) is
a bounded operator in \(\mathscr{H}\). The complement
\(\textup{\large\(\sigma\)}(A)~=~\mathbb{C}~\setminus~\textup{\large\(\rho\)}(A)\)
of the resolvent set \(\textup{\large\(\rho\)}(A)\) is said to be
the \emph{spectrum of the operator} \(A\).

The resolvent set of every closed operator is an open set (may be
empty). If the operator \(A\) is a bounded everywhere defined
operator, that is \(\mathcal{D}_A=\mathscr{H}\), then the
resolvent set \(\textup{\large\(\rho\)}(A)\) is not empty:
\(\textup{\large\(\rho\)}(A)\supseteq\{\lambda:\,|\lambda|>\|A\|\}\).
The spectrum of a bounded everywhere defined linear operator is a
not empty set.

The operator valued function \((z\,\mathscr{I}-A)^{-1}\) of the
variable \(z\) defined on the resolvent set
\(\textup{\large\(\rho\)}(A)\) of the operator \(A\) is said to be
the \emph{resolvent of the operator \(A\)} and denoted by
\(R_A(z)\):
\begin{equation}
\label{Res} R_A(z)=(z\,\mathscr{I}-A)^{-1}\,.
\end{equation}

We adhere to the following classification of  points of the
spectrum of a linear operator. (See \cite[VII.5.1]{DuSch}.)
\begin{definition}
\label{SpClas}%
Let \(A\) be a linear operator in a Hilbert space \(\mathscr{H}\),
with the domain of definition \(\mathcal{D}_A\).
\begin{enumerate}%
\item[\textup{1}.] The set of all \(\lambda\in\textup{\large\(\sigma\)}(A)\)
such that the mapping \(\lambda\,\mathscr{I}-A: \,\mathcal{D}_A\to\mathscr{H}\), is not one-to-one is
called the \emph{point spectrum of \(A\)}, and is denoted by
\(\textup{\large\(\sigma\)}_{\!p}(A)\). Thus,
\(\lambda\in\textup{\large\(\sigma\)}_{\!p}(A)\) if and only if
\(Ax=\lambda{}x\) for some non-zero \(x\in\mathcal{D}_A\).%
\item[\textup{2}.] The set of all
\(\lambda\in\textup{\large\(\sigma\)}(A)\) for which  the subspace
 \((\lambda\mathscr{I}-A)\mathscr{D}_A\) is not dense in \(\mathscr{H}\)
 is called the \emph{residual spectrum of \(A\)}, and  is denoted by
\(\textup{\large\(\sigma\)}_{\!r}(A)\).  Thus,
\(\lambda\in\textup{\large\(\sigma\)}_{\!r}(A)\) if and only if
\(A^{\ast}x=\overline{\lambda}x\) for some non-zero \(x\in\mathcal{D}_{A^{\ast}}\),
where \(A^{\ast}\) is the operator adjoint to the operator \(A\).
\item[\textup{3}.] The set of all
\(\lambda\in\textup{\large\(\sigma\)}(A)\) for which
the mapping \(\lambda\,\mathscr{I}-A: \,\mathcal{D}_A\to\mathscr{H}\), is one-to-one,
 the subspace
 \((\lambda\mathscr{I}-A)\mathscr{D}_A\) is dense in
 \(\mathscr{H}\), but the inverse operator
 \((\lambda\mathscr{I}-A)^{-1}\!:\,(\lambda\mathscr{I}-A)\mathscr{D}_A\to\mathscr{D}_A\)
  is unbounded
 is called  the \emph{continuous spectrum of \(A\)}, and  is denoted by
\(\textup{\large\(\sigma\)}_{\!c}(A)\). Thus,
\(\lambda\in\textup{\large\(\sigma\)}_{\!c}(A)\) if and only if
\(\lambda\not\in\textup{\large\(\sigma\)}_{\!p}(A)\), \(\lambda\not\in\textup{\large\(\sigma\)}_{\!r}(A)\),
and there exists a sequence \(\{x_n\}_{n}\) such that
\[x_n\in\mathscr{D}_A, \ \ \|x_n\|=1,\quad \textup{but} \ \
\|(A-\lambda\mathscr{I}                                 )x_n\|\to0\ \ \textup{as} \ \ n\to\infty.\]
\end{enumerate}%
\end{definition}
If the operator \(A\) is closed, in particular if \(\mathscr{D}_A=\mathscr{H}\) and \(A\) is bounded, then
\[\textup{\large\(\sigma\)}(A)=
\textup{\large\(\sigma\)}_{\!p}(A)\cup%
\textup{\large\(\sigma\)}_{\!r}(A)\cup\textup{\large\(\sigma\)}_{\!c}(A).\]
This classification is rough, sometimes one introduces the more
fine classification, like \emph{essential spectrum,\,approximate
point spectrum,} etc. However we keep in the classification of
Definition \ref{SpClas}.

Let \(a\) and \(b\) be points of \(\mathbb{C}\). By definition,
the interval \([a,\,b]\) is the set:
\([a,\,b]=\{(1-\tau)a+\tau{}b:\,\tau\,\,\textup{runs
over}\,\,[0,\,1]\}\). The open interval \((a,b)\), as well as
half-open intervals are defined analogously.

\paragraph{1.\,The spectrum of
\mathversion{bold}\(\bf{\mathscr{F}_{\!_{\scriptstyle\bm{{\mathbb{R}}^{+}}}}}\).
\mathversion{normal} \\[-3.5ex]}
\begin{theorem} {\ } \\[-3.0ex] %
\label{SpTFO}%
\begin{enumerate}
\item[\textup{1}.] The spectrum
\(\textup{\large\(\sigma\)}(
\mathscr{F}_{\!_{\scriptstyle\mathbb{R}^{+}}})\) of the truncated
Fourier operator \(\mathscr{F}_{\!_{\scriptstyle\mathbb{R}^{+}}}\)
is: \\[-1.0ex]
\begin{equation}%
\label{SpecTFO}
\textup{\large\(\sigma\)}(\mathscr{F}_{\!_{\scriptstyle\mathbb{R}^{+}}})=
[-2^{-1/2}\,e^{i\pi/4},\,2^{-1/2}\,e^{i\pi/4}]
\end{equation}%
\item[\textup{2}.] The truncated
Fourier operator \(\mathscr{F}_{\!_{\scriptstyle\mathbb{R}^{+}}}\)
has no point spectrum and no residual spectrum:
\begin{equation}%
\label{prSpecTFO}
\textup{\large\(\sigma\)}_{\!p}(\mathscr{F}_{\!_{\scriptstyle\mathbb{R}^{+}}})=\emptyset,\quad
\textup{\large\(\sigma\)}_{\!r}(\mathscr{F}_{\!_{\scriptstyle\mathbb{R}^{+}}})=\emptyset\,.
\end{equation}%
Thus, its spectrum is continuous:
\begin{equation*}%
\textup{\large\(\sigma\)}(\mathscr{F}_{\!_{\scriptstyle\mathbb{R}^{+}}})=
\textup{\large\(\sigma\)}_{\!c}(\mathscr{F}_{\!_{\scriptstyle\mathbb{R}^{+}}})\,.
\end{equation*}%
\end{enumerate}
\end{theorem}
\begin{nonumtheorem}[\ \(\textup{\ref{SpTFO}}^{\!\boldsymbol\ast}\)\!.]
{\ } \\[-3.0ex] %
\begin{enumerate}
\item[\textup{1}.] %
The spectrum \(\textup{\large\(\sigma\)}(
\mathscr{F}^{\,\ast}_{\!_{\scriptstyle\mathbb{R}^{+}}})\) of the
operator
\(\mathscr{F}^{\,\ast}_{\!_{\scriptstyle\mathbb{R}^{+}}}\), which
is adjoint to the operator
\(\mathscr{F}_{\!_{\scriptstyle\mathbb{R}^{+}}}\), is:
\begin{equation}%
\label{SpecTFOa}
\textup{\large\(\sigma\)}(\mathscr{F}^{\,\ast}_{\!_{\scriptstyle\mathbb{R}^{+}}})=
[-2^{-1/2}\,e^{-i\pi/4},\,2^{-1/2}\,e^{-i\pi/4}]
\end{equation}%
\item[\textup{2}.] The operator \(\mathscr{F}^{\,\ast}_{\!_{\scriptstyle\mathbb{R}^{+}}}\)
has no point spectrum and no residual spectrum:
\begin{equation}%
\label{prSpecTFOa}
\textup{\large\(\sigma\)}_{\!p}(\mathscr{F}^{\,\ast}_{\!_{\scriptstyle\mathbb{R}^{+}}})=\emptyset,\quad
\textup{\large\(\sigma\)}_{\!r}(\mathscr{F}^{\,\ast}_{\!_{\scriptstyle\mathbb{R}^{+}}})=\emptyset\,.
\end{equation}%
Thus, its spectrum is continuous:
\begin{equation*}%
\textup{\large\(\sigma\)}(\mathscr{F}^{\,\ast}_{\!_{\scriptstyle\mathbb{R}_{+}}})=
\textup{\large\(\sigma\)}_{\!c}(\mathscr{F}^{\,\ast}_{\!_{\scriptstyle\mathbb{R}^{+}}})\,.
\end{equation*}%
\end{enumerate}
\end{nonumtheorem}
 The point \(\zeta=0\) is in some sense a singular
point for operator
\(\mathscr{F}_{\!_{\scriptstyle\mathbb{R}_{+}}}\). (See in
particular Theorem \ref{TEstRes} and Remark \ref{spSing}). This
point splits the spectrum
\(\textup{\large\(\sigma\)}(\mathscr{F}_{\!_{\scriptstyle\mathbb{R}^{+}}})\)
on two parts:
\(\textup{\large\(\sigma\)}^{+}(\mathscr{F}_{\!_{\scriptstyle\mathbb{R}^{+}}})\)
and
\(\textup{\large\(\sigma\)}^{-}(\mathscr{F}_{\!_{\scriptstyle\mathbb{R}^{+}}})\).
\begin{definition}
\label{SplSpD}
\begin{equation}%
\label{SpFSpl} \textup{\large\(\sigma\)}^{+}(\mathscr{F}_{\!_{\scriptstyle\mathbb{R}^{+}}})%
\stackrel{\textup{\tiny def}}{=}\Big(0,\,
\frac{1}{\sqrt{2}}\,e^{i\pi/4}\Big],\quad
\textup{\large\(\sigma\)}^{-}(\mathscr{F}_{\!_{\scriptstyle\mathbb{R}^{+}}})\stackrel{\textup{\tiny
def}}{=} \Big[-\frac{1}{\sqrt{2}}\,e^{i\pi/4},\,0 \Big)\,.
\end{equation}%
\end{definition}
Thus, the spectrum
\(\textup{\large\(\sigma\)}(\mathscr{F}_{\mathbb{R}^{+}})\) of the
operator \(\mathscr{F}_{\!_{{\scriptstyle\mathbb{R}^{+}}}}\)
splits into the union
\begin{equation}%
\label{SplSp}
\textup{\large\(\sigma\)}(\mathscr{F}_{\mathbb{R}^{+}})=
\textup{\large\(\sigma\)}^{+}(\mathscr{F}_{\mathbb{R}^{+}})
\cup\textup{\large\(\sigma\)}^{-}(\mathscr{F}_{\mathbb{R}^{+}})\cup\{0\}\,.
\end{equation}

\paragraph{2.\,Growth of the resolvent of the operator \(\mathscr{F}_{\!_{\scriptstyle\mathbb{R}_{+}}}\) near the spectrum.\\}
Let us discuss how the resolvent
\((z\mathscr{I}-\mathscr{F}_{\!_{\scriptstyle\mathbb{R}^{+}}})^{-1}\)
growth when \(z\) approaches the spectrum
\(\textup{\large\(\sigma\)}(\mathscr{F}_{\!_{\scriptstyle\mathbb{R}_{+}}})\).
The growth of the resolvent depends on the point
\(\zeta\in\textup{\large\(\sigma\)}(\mathscr{F}_{\!_{\scriptstyle\mathbb{R}_{+}}})\)
which \(z\) approaches. Roughly speaking, for every fixed
\(\zeta\in\textup{\large\(\sigma\)}%
(\mathscr{F}_{\!_{\scriptstyle\mathbb{R}_{+}}})\),
\(\zeta\not=0\),  the norm
\(\|(z\mathscr{I}-\mathscr{F}_{\!_{\scriptstyle\mathbb{R}_{+}}})^{-1}\|\)
growths as \(C(\zeta)/|z-\zeta|\)  as \(z\to\zeta\).  Here
\(C(\zeta),\, 0<C(\zeta)<\infty\), is a constant with respect to
\(z\). However, \(C(\zeta)\to\infty\) as \(\zeta\to{}0\). If
\(\zeta=0\), that is if \(z\to{}0\), then
\(\|(z\,\mathscr{I}-\mathscr{F}_{\!_{\scriptstyle\mathbb{R}_{+}}})^{-1}\|\)
growths  as \(|z|^{-2}\). Let us formulate the precise result.
\begin{theorem}%
\label{TEstResT}%
Let \(\zeta\) be a point of the spectrum
\(\textup{\large\(\sigma\)}(\mathscr{F}_{\!_{\scriptstyle\mathbb{R}^{+}}})\)
of the operator \(\mathscr{F}_{\!_{\scriptstyle\mathbb{R}^{+}}}\),
and let the point \(z\) lie on the normal to the interval
\(\textup{\large\(\sigma\)}(\mathscr{F}_{\!_{\scriptstyle\mathbb{R}^{+}}})\)
at the point \(\zeta\):
\begin{equation}
\label{znorsp}%
 z=\zeta\pm{}|z-\zeta|e^{i3\pi/4}\,.
\end{equation}
 Then
\begin{enumerate}
\item[\textup{1.}]
The resolvent \((z\mathscr{I}-
\mathscr{F}_{\!_{\scriptstyle\mathbb{R}^{+}}})^{-1}\) admits the
estimate from above:
\begin{equation}
\label{UpEsResS} \big\|(z\mathscr{I}-
(\mathscr{F}_{\!_{\scriptstyle\mathbb{R}^{+}}})^{-1}\big\|\leq
A(z)\frac{1}{|\zeta|}\cdot\frac{1}{|z-\zeta|}\,,
\end{equation}
where \(A(z)=\dfrac{(2|z|^2+1)^{1/2}}{2}\).
\item[\textup{2.}]
If moreover the condition \(|z-\zeta|\leq|\zeta|\) is satisfied,
then the resolvent \((z\mathscr{I}-
\mathscr{F}_{\!_{\scriptstyle\mathbb{R}_{+}}})^{-1}\) also
admits the estimate from below:
\begin{equation}
\label{LoEsResS} A(z)\frac{1}{|\zeta|}\cdot\frac{1}{|z-\zeta|}
-B(z)|\zeta||z-\zeta|
\leq\big\|(z\mathscr{I}-\mathscr{F}_E)^{-1}\big\| \,,
\end{equation}
 where \(A(z)\) is the same that in \eqref{UpEsResS} and
\(B(z)=\dfrac{4}{(2|z|^2+1)^{3/2}}\)\,.
\item[\textup{3.}] For \(\zeta=0\), the resolvent
\((z\mathscr{I}-
\mathscr{F}_{\!_{\scriptstyle\mathbb{R}^{+}}})^{-1}\) admits the
estimates
\begin{equation}
\label{EZEZ} 2A(z)\frac{1}{|z|^2}-B(z) \leq\big\|(z\mathscr{I}-
\mathscr{F}_{\!_{\scriptstyle\mathbb{R}^{+}}})^{-1}\big\|\leq{}2A(z)\frac{1}{|z|^2},
\end{equation}
where \(A(z)\) and \(B(z)\) are the same that in \eqref{UpEsResS},
 \eqref{LoEsResS}, and \(z\) is an arbitrary point of the normal.
\end{enumerate}
In particular, if \(\zeta\not=0\), and \(z\) tends to \(\zeta\)
along the normal to the interval
\(\textup{\large\(\sigma\)}(\mathscr{F}_{\!_{\scriptstyle\mathbb{R}^{+}}})\),
then
\begin{equation}
\label{AsNzp}
\big\|(z\mathscr{I}-
\mathscr{F}_{\!_{\scriptstyle\mathbb{R}_{+}}})^{-1}\big\|=C(\zeta)\,\frac{1}{|z-\zeta|}
+O(1)\,,
\end{equation}
where
\begin{equation}
\label{ConAs}%
 C(\zeta)=\frac{\sqrt{1+2|\zeta|^2}}{2|\zeta|}\,.
\end{equation}
If \(\zeta=0\) and \(z\) tends to \(\zeta\) along the normal to
the interval
\(\textup{\large\(\sigma\)}(\mathscr{F}_{\!_{\scriptstyle\mathbb{R}^{+}}})\),
then
\begin{equation}
\label{Aszp}
\big\|(z\mathscr{I}-
\mathscr{F}_{\!_{\scriptstyle\mathbb{R}_{+}}})^{-1}\big\|=|z|^{-2}+O(1)\,,
\end{equation}
where \(O(1)\) is a value which remains bounded as \(z\) tends to
\(\zeta\).
\end{theorem}%
\begin{remark}
\label{nonnorm} From Theorem \ref{TEstRes} it follows that the
operator \(\mathscr{F}_{\!_{\scriptstyle\mathbb{R}^{+}}}\) is not
similar to a normal operator. The similarity of the operator
\(\mathscr{F}_{\!_{\scriptstyle\mathbb{R}^{+}}}\) to a normal
operator is non-compatible with the growth \eqref{Aszp} of its
resolvent.
\end{remark}
\begin{remark}
\label{spSing}%
 The point \(\zeta=0\) is a distinguished point of
the spectrum
\(\textup{\large\(\sigma\)}(\mathscr{F}_{\!_{\scriptstyle\mathbb{R}^{+}}})\).
Near this point the resolvent of the operator
\(\mathscr{F}_{\!_{\scriptstyle\mathbb{R}^{+}}}\) growths faster
than near any other point
\(\zeta\in\textup{\large\(\sigma\)}(\mathscr{F}_{\!_{\scriptstyle\mathbb{R}_{+}}})\).
 The point \(\zeta=0\) is a \emph{spectral singularity}. In what follows we still will face a special role of the point \(\zeta=0\) in the spectral theory of the operator
\(\mathscr{F}_{\!_{\scriptstyle\mathbb{R}^{+}}}\).
\end{remark}
\paragraph{3.\,The multiplicity of the spectrum of \(\bm{\mathscr{F}_{\!_{\scriptstyle\mathbb{R}^{+}}}}\).\\}
For operators which are not normal there is no general full-blooded theory of spectral multiplicity.
For self-adjoint operators, the property of its spectrum to be of multiplicity one, or in other word
the property of the spectrum to be simple, is equivalent
to the property of the operator to possess a cyclic vector. Therefore, for an arbitrary operator,
the property of
the operator to possess a cyclic vector may be accepted as a definition of the simplicity
of its spectrum.

We recall that the \emph{vector \(x,\,x\in\mathscr{H}\), is said to be
cyclic for the operator \(A,\,A:\,\mathscr{H}\to\mathscr{H}\),} if the linear hall of the set of vectors %
\(\{A^nx\}_{n\in\mathbb{N}}\) is dense in \(\mathscr{H}\).

\begin{theorem}
\label{SimSpe} The spectrum of the operator
\(\mathscr{F}_{\!_{\scriptstyle\mathbb{R}^{+}}}\) is simple:
 there exist vectors which are cyclic for \(\mathscr{F}_{\!_{\scriptstyle\mathbb{R}^{+}}}\).
\end{theorem}
\paragraph{4.\,Operator calculus for the operator
\(\mathscr{F}_{\!_{\scriptstyle\mathbb{R}^{+}}}\).} {\ }\\[0.8ex]
\noindent%
\textbf{Holomorphic operator calculus.} We say that a
\emph{function \(h\) is holomorphic on a closed set
\(\textup{\large\(\sigma\)}\),}
\(\textup{\large\(\sigma\)}\in\mathbb{C}\), if the function \(f\)
is defined and holomorphic in an open neighborhood of the set
\(\textup{\large\(\sigma\)}\). (The neighborhood may depend on the
function \(f\).) The set of functions holomorphic on the set
\(\textup{\large\(\sigma\)}\) forms an algebra over the field of
complex numbers. This algebra is denoted by
\(\textup{hol}(\textup{\large\(\sigma\)})\).

According to the general theory of linear operators, for arbitrary
function \(h\) which is holomorphic on the spectrum
\(\textup{\large\(\sigma\)}(\mathscr{F}_{\!_{\scriptstyle\mathbb{R}^{+}}})\)
of the operator \(\mathscr{F}_{\!_{\scriptstyle\mathbb{R}^{+}}}\)
 one can define the operator \(h(\mathscr{F}_{\!_{\scriptstyle\mathbb{R}^{+}}})\)
 by means of the integral
 \begin{equation}
 \label{DunInt}
 h_{\textup{hol}}(\mathscr{F}_{\!_{\scriptstyle\mathbb{R}_{+}}})\stackrel{\textup{\tiny def}}{=}
 \frac{1}{2\pi{}i}\int\limits_{\Gamma}h(z)\,%
 R_{_{\,\scriptstyle\mathscr{F}_{\!_{\mathbb{R}^{+}}}}}(z)\,dz\,,
 \end{equation}
where \(
R_{_{\,\scriptstyle\mathscr{F}_{\!_{\mathbb{R}^{+}}}}}(z)=
 (z\mathscr{I}-\mathscr{F}_{\!_{\scriptstyle\mathbb{R}^{+}}})^{-1}\) is the resolvent
 of the operator
 \(\mathscr{F}_{\!_{\scriptstyle\mathbb{R}^{+}}}\),
\(\Gamma\) is an arbitrary simple contour which encloses the
spectrum \(\textup{\large\(\sigma\)}(
\mathscr{F}_{\!_{\scriptstyle\mathbb{R}^{+}}})\) and is contained
in the domain of holomorphy of the function \(h\). The integral is
taken counterclockwise. The value of this integral does not depend on the contour
\(\Gamma\).

 The operator \(h_{\textup{hol}}(\mathscr{F}_{\!_{\scriptstyle\mathbb{R}^{+}}})\)
 is called \emph{the function \(h\) of the operator
 \(\mathscr{F}_{\!_{\scriptstyle\mathbb{R}^{+}}}\).}

The correspondence
\[h(z)\to{}h_{\textup{hol}}(\mathscr{F}_{\!_{\scriptstyle\mathbb{R}^{+}}}),\quad \textup{where }
h_{\textup{hol}}(\mathscr{F}_{\!_{\scriptstyle\mathbb{R}^{+}}}) \ \textup{is
defined by }\eqref{DunInt}\,,
\]
is said to be the \emph{holomorphic functional calculus for the
operator \(\mathscr{F}_{\!_{\scriptstyle\mathbb{R}^{+}}}\).}

 The holomorphic functional calculus
is a homomorphism of the algebra
\(\textup{hol}(\textup{\large\(\sigma\)}(\mathscr{F}_{\!_{\scriptstyle\mathbb{R}^{+}}}))\)
into the algebra of bounded operators in \(\mathscr{H}=L^2(\mathbb{R}^{+})\): \\

\noindent%
 \textbf{Algebraic properties of the holomorphic
functional calculus:}
\begin{enumerate}
\item[\textup{1}.] If \(h(\zeta)\equiv{}1\), then \(h_{\textup{hol}}(\mathscr{F}_{\!_{\scriptstyle\mathbb{R}^{+}}})=
\mathscr{I}\).
\item[\textup{2}.] If \(h(\zeta)\equiv{}\zeta\), then \(h_{\textup{hol}}(\mathscr{F}_{\!_{\scriptstyle\mathbb{R}^{+}}})=
\mathscr{F}_{\!_{\scriptstyle\mathbb{R}^{+}}}\).
\item[\textup{3}.] If \(h(\zeta)=\alpha_1h_1(\zeta)+\alpha_2h_2(\zeta)\), where \(h_1,\,h_2%
\in\textup{hol}(\textup{\large\(\sigma\)}(\mathscr{F}_{\!_{\scriptstyle\mathbb{R}^{+}}}))\),\\
\(\alpha_1,\,\alpha_2\in\mathbb{C}\), \ then \ %
\(h_{\textup{hol}}(\mathscr{F}_{\!_{\scriptstyle\mathbb{R}^{+}}})%
=\alpha_1(h_1)_{\textup{hol}}(\mathscr{F}_{\!_{\scriptstyle\mathbb{R}^{+}}})+
\alpha_2(h_2)_{\textup{hol}}(\mathscr{F}_{\!_{\scriptstyle\mathbb{R}^{+}}})\).
\item[\textup{4}.] If \(h(\zeta)=h_1(\zeta)\cdot{}h_2(\zeta)\), where \(h_1,\,h_2%
\in\textup{hol}(\textup{\large\(\sigma\)}(\mathscr{F}_{\!_{\scriptstyle\mathbb{R}^{+}}}))\),\ %
 then \\ %
\(h_{\textup{hol}}(\mathscr{F}_{\!_{\scriptstyle\mathbb{R}^{+}}})%
=(h_1)_{\textup{hol}}(\mathscr{F}_{\!_{\scriptstyle\mathbb{R}^{+}}})\cdot{}
(h_2)_{\textup{hol}}(\mathscr{F}_{\!_{\scriptstyle\mathbb{R}^{+}}})\).
\end{enumerate}
From 1\,-\,4 it follows
\begin{enumerate}
\item[\textup{5}.] If \(h(\zeta)\in\textup{hol}%
(\textup{\large\(\sigma\)}(\mathscr{F}_{\!_{\scriptstyle\mathbb{R}^{+}}}))\), %
and \(h(\zeta)\not=0\) for
\(\zeta\in\textup{\large\(\sigma\)}(\mathscr{F}_{\!_{\scriptstyle\mathbb{R}^{+}}})\),
\ 
\(h^{-1}(\zeta)\in\textup{hol}%
(\textup{\large\(\sigma\)}(\mathscr{F}_{\!_{\scriptstyle\mathbb{R}^{+}}})\),
then the operator
\(h_{\textup{hol}}(\mathscr{F}_{\!_{\scriptstyle\mathbb{R}^{+}}})\) is
invertible, and\\
\(\big(h_{\textup{hol}}(\mathscr{F}_{\!_{\scriptstyle\mathbb{R}^{+}}})\big)^{-1}=%
\big(h^{-1}\big)_{\textup{hol}}(\mathscr{F}_{\!_{\scriptstyle\mathbb{R}^{+}}})\)\,.
\end{enumerate}
In particular,
\begin{enumerate}
\item[\textup{6}.] If \(h(\zeta)\equiv{}(z-\zeta)^{-1}\), where
\(z\in\mathbb{C}\setminus\textup{\large\(\sigma\)}(\mathscr{F}_{\!_{\scriptstyle\mathbb{R}^{+}}})\),
then  %
\(h_{\textup{hol}}(\mathscr{F}_{\!_{\scriptstyle\mathbb{R}^{+}}})=
\big(z\mathscr{I}-\mathscr{F}_{\!_{\scriptstyle\mathbb{R}^{+}}}\big)^{-1}\).\\
\end{enumerate}

The holomorphic functional calculus is applicable to an arbitrary
bounded operator \(A\) in a Hilbert space. However the operator
\(\mathscr{F}_{\!_{\scriptstyle\mathbb{R}^{+}}}\) is the very
specific operator. For this operator, it is possible to define
functions \(h(\mathscr{F}_{\!_{\scriptstyle\mathbb{R}^{+}}})\) for
functions \(h\) from much more wider class that the class
\(\textup{hol}(\textup{\large\(\sigma\)}(\mathscr{F}_{\!_{\scriptstyle\mathbb{R}^{+}}}))\).\\[1.2ex]
\textbf{\(\mathbf{\bm{\mathscr{
F}}_{\!_{\scriptstyle\bm{\mathbb{R}}^{\bm{+}}}}}\)-admissible
functions.} The next notion is one of the main notions of this
work.
\begin{definition}%
\label{DFAdF}%
 A function \(h(\zeta)\) is said to be
\({\mathscr{F}}_{\!_{\scriptstyle\mathbb{R}^{+}}}\)--\,\emph{admissible}
if
\begin{enumerate}
\item[\textup{1}.]  \(h(\zeta)\) is a Borel-measurable function which is
 defined almost everywhere with respect to the
Lebesgue measure on the spectrum
\(\textup{\large\(\sigma\)}(\mathscr{F}_{\!_{\scriptstyle\mathbb{R}^{+}}})%
=\Big[-\frac{1}{\sqrt{2}}\,e^{i\pi/4},\,
\frac{1}{\sqrt{2}}\,e^{i\pi/4}\Big]\) of the operator
\(\mathscr{F}_{\!_{\scriptstyle\mathbb{R}^{+}}}\).
\item[\textup{2}.] The norm \(\|h\|_{{}_{\scriptstyle\mathscr{F}_{_{\mathbb{R}^{+}}}}}\) is  finite, where
\begin{equation}
\label{NFAF}%
\|h\|_{{}_{\scriptstyle\mathscr{F}_{_{\mathbb{R}_{+}}}}}\stackrel{\textup{\tiny{}def}}{=}
\underset{\zeta\in%
{\textstyle\sigma}(\mathscr{F}_{\!_{\scriptstyle\mathbb{R}^{+}}})}{\textup{ess
sup}}\,\left(\frac{|h(\zeta)+h(-\zeta)|}{2}+ \frac{|h(\zeta)-h(-\zeta)|}{2|\zeta|}\right)\,.
\end{equation}%
\end{enumerate}%
The set of all
\(\mathscr{F}_{_{{\scriptstyle\mathbb{R}}^{+}}}\)\,-\,admissible
functions provided by natural `pointwise' algebraic operation and
the norm \eqref{NFAF} is denoted by
\(\mathfrak{B}_{_{{\scriptstyle\mathscr{F}}_{_{\!{\mathbb{R}}^{+}}}}}\)\!\!.
\end{definition}%
An analogous definition related to the adjoint operator
\(\mathscr{F}^{\,\ast}_{_{\!{\mathbb{R}}^{+}}}\) is:
\begin{definition}%
\label{DFAdFAd}%
 A function \(h(\zeta)\) is said to be
\(\mathscr{F}^{\,\ast}_{\!_{\scriptstyle\mathbb{R}^{+}}}\)-\emph{admissible}
if
\begin{enumerate}
\item[\textup{1}.]  \(h(\zeta)\) is a Borel-measurable function which is
 defined almost everywhere with respect to the
Lebesgue measure on the spectrum
\(\textup{\large\(\sigma\)}(\mathscr{F}^{\,\ast}_{\!_{\scriptstyle\mathbb{R}_{+}}})%
=\Big[-\frac{1}{\sqrt{2}}\,e^{-i\pi/4},\,
\frac{1}{\sqrt{2}}\,e^{-i\pi/4}\Big]\) of the operator
\(\mathscr{F}^{\,\ast}_{\!_{\scriptstyle\mathbb{R}^{+}}}\).
\item[\textup{2}.] The norm
\(\|h\|_{{}_{\scriptstyle\mathscr{F}^{\,\ast}_{_{\mathbb{R}^{+}}}}}\)
is  finite, where
\begin{equation}
\label{NFAFAd}%
\|h\|_{{}_{\scriptstyle\mathscr{F}^{\,\ast}_{_{\mathbb{R}^{+}}}}}\stackrel{\textup{\tiny{}def}}{=}
\underset{\zeta\in%
{\textstyle\sigma}(\mathscr{F}^{\ast}_{\!_{\scriptstyle\mathbb{R}^{+}}})}{\textup{ess
sup}}\,\left(\frac{|h(\zeta)+h(-\zeta)|}{2}+ \frac{|h(\zeta)-h(-\zeta)|}{2|\zeta|}\right)\,.
\end{equation}%
\end{enumerate}%
The set of all
\(\mathscr{F}^{\,\ast}_{_{{\scriptstyle\mathbb{R}}^{+}}}\)\,-\,admissible
functions provided by natural "pointwise" algebraic operation and
the norm \eqref{NFAF} is denoted by
\(\mathfrak{B}_{_{{\scriptstyle\mathscr{F}}^{\ast}_{_{{\mathbb{R}}^{+}}}}}\)\!\!.
\end{definition}%
\begin{lemma}
\label{BanAl}
 Each of two sets
\(\mathfrak{B}_{_{{\scriptstyle\mathscr{F}}_{_{{\mathbb{R}}^{+}}}}}\),
 \(\mathfrak{B}_{_{{\scriptstyle\mathscr{F}}^{\ast}_{_{{\mathbb{R}}^{+}}}}}\)
 is a Banach algebra, and
\begin{equation}
\|h_1h_2\|_{_{{\scriptstyle\mathscr{F}}_{_{{\mathbb{R}}^{+}}}}}%
\leq\|h_1\|_{_{{\scriptstyle\mathscr{F}}_{_{{\mathbb{R}}^{+}}}}}%
\|h_2\|_{_{{\scriptstyle\mathscr{F}}_{_{{\mathbb{R}}_{+}}}}} \ \
\textup{or} \ \ \|h_1h_2\|_{_{{\scriptstyle\mathscr{F}}^{\ast}_{\mathbb{R}^{+}}}}%
\leq\|h_1\|_{_{{\scriptstyle\mathscr{F}}^{\ast}_{\mathbb{R}_{+}}}}\,%
\|h_2\|_{_{{\scriptstyle\mathscr{F}}^{\ast}_{\mathbb{R}_{+}}}}
\end{equation}
for every
\(h_1,\,h_2\in\mathfrak{B}_{_{{\scriptstyle\mathscr{F}}_{_{{\mathbb{R}}^{+}}}}}\)
or
\(h_1,\,h_2\in\mathfrak{B}_{_{{\scriptstyle\mathscr{F}}^{\ast}_{_{{\mathbb{R}}^{+}}}}}\)
respectively.
\end{lemma}
\begin{lemma}
\label{InvCo}%
 Let \(h(\zeta)\) be an
\({\mathscr{F}}_{\!_{\scriptstyle\mathbb{R}^{+}}}\)\!--\,admissible
function.

The function \(h^{-1}(\zeta)\) is  an
\({\mathscr{F}}_{\!_{\scriptstyle\mathbb{R}^{+}}}\)\!--\,admissible
function if and only if the set of values of the function \(h\) is
separated from zero, that is the following condition
\begin{equation}
\label{sepCo}
\underset{\zeta\in{\textstyle{\sigma}}(\mathscr{F}_{{\mathbb{R}^{+}}})}{\textup{ess
inf}}\,|h(\zeta)|>0\,.
\end{equation}
 holds.

If the condition \eqref{sepCo} holds, then
 \begin{equation}
 \label{InvEst}
 \| h^{-1}(\zeta)\|_{\mathscr{F}_{{\mathbb{R}^{+}}}}\leq
  \| h(\zeta)\|_{\mathscr{F}_{{\mathbb{R}^{+}}}}\cdot
 \Big(\,
 \underset{\zeta\in{\textstyle\sigma}({\mathscr{F}_{{\mathbb{R}^{+}}}})}{\textup{ess inf}}\,|h(\zeta)|
 \Big)^{-2}
\,.
 \end{equation}
\end{lemma}
\begin{definition}
\label{DeCoF} If \(h(\zeta)\) is a complex-valued function defined
on a subset \(S\) of the complex plane, than the \emph{conjugated}
function \(\overline{h}(\zeta)\) is a function defined on the
conjugated set \(\overline{S}\) by the equality
\begin{equation}
\label{defcf} \overline{h}(\zeta)\stackrel{\textup{\tiny
def}}{=}\overline{h(\overline{\zeta})}\,.
\end{equation}
\end{definition}
\begin{lemma}
\label{PrCoFu} A function \(h\) belongs to the algebra
\(\mathfrak{B}_{_{{\scriptstyle\mathscr{F}}_{_{{\mathbb{R}}^{+}}}}}\)
if and only if the conjugate function \(\overline{h}\) belongs to
the algebra
 \(\mathfrak{B}_{_{{\scriptstyle\mathscr{F}}^{\ast}_{_{{\mathbb{R}}^{+}}}}}\). Moreover,
 \begin{equation*}
 \|h\|_{_{\scriptstyle\mathfrak{B}_{_{{\mathscr{F}}_{_{{\mathbb{R}}^{+}}}}}}}
 =
 \|\overline{h}\|_{_{\scriptstyle\mathfrak{B}_{_{{\mathscr{F}}_{_{{\mathbb{R}}^{+}}}}}}}\,.
 \end{equation*}
\end{lemma}
\noindent \textbf{The resolvent-based functional
 calculus for the operator \textbf{\({\bm{\mathscr{
F}}_{\!_{\scriptstyle\bm{\mathbb{R}}_{\bm{+}}}}}\)}.}\\
 The following theorem is one of the main results of
this work:
\begin{theorem}
\label{ExFCal}%
 Let \(h(\zeta)\) be an \(\mathscr{F}_{_{{\scriptstyle\mathbb{R}}^{+}}}\)\,-\,admissible
 function:
 \(h(\zeta)\in\mathfrak{B}_{_{{\scriptstyle\mathscr{F}}_{_{{\mathbb{R}}^{+}}}}}\).\\
Then there exists the strong limit
 \begin{align}
 \label{FAdFOfOp}
 h({\mathscr{F}}_{_{{\!\scriptstyle\mathbb{R}}_{+}}})&\stackrel{\textup{\tiny def}}{=} \\
 =\lim_{\varepsilon\to+0}\,\,&\frac{1}{2\pi{}i}
 \int\limits_{{\textstyle\sigma}(\mathscr{F}_{\!_{\scriptstyle\mathbb{R}^{+}}})}\hspace{-2.0ex}
 h(\zeta)\,\Big(
 R_{_{\scriptstyle\mathscr{F}_{\!_{\mathbb{R}^{+}}}}}\!\big(\zeta-\varepsilon{}ie^{i\pi/4}\big)-
 R_{_{\scriptstyle\mathscr{F}_{\!_{\mathbb{R}^{+}}}}}\!\big(\zeta+\varepsilon{}ie^{i\pi/4}\big)\Big)\,d\zeta\,,
 \notag
 \end{align}
 where \( R_{_{\,\scriptstyle\mathscr{F}_{\!_{\mathbb{R}^{+}}}}}(z)=
 (z\mathscr{I}-\mathscr{F}_{\!_{\scriptstyle\mathbb{R}^{+}}})^{-1}\) is the resolvent
 of the operator
 \(\mathscr{F}_{\!_{\scriptstyle\mathbb{R}^{+}}}\),
 and the integral is taken along the interval
 \(\textup{\large\(\sigma\)}(\mathscr{F}_{\!_{\scriptstyle\mathbb{R}^{+}}})%
=\Big[-\frac{1}{\sqrt{2}}\,e^{i\pi/4},\,
\frac{1}{\sqrt{2}}\,e^{i\pi/4}\Big]\) from the point
\(-\frac{1}{\sqrt{2}}\,e^{i\pi/4}\) to the point
\(\frac{1}{\sqrt{2}}\,e^{i\pi/4}\).
\end{theorem}
 The following result \textsf{s}upplements Theorem \ref{ExFCal}\,.
 \begin{nonumtheorem}[\ \(\textup{\ref{ExFCal}}^{\textup{\,\textmd{s}}}\)\!.] Let the function \(h\), which appears
 in \textup{Theorem \ref{ExFCal}}, satisfy the extra conditions:
 \begin{enumerate}
 \item[\textup{1.}]\,\(h(\zeta)\) is continuous for
 \(\zeta\in\Big[-\frac{1}{\sqrt{2}}\,e^{i\pi/4},\,
\frac{1}{\sqrt{2}}\,e^{i\pi/4}\Big]\)\textup{;}
\item[\textup{2.}]\,\(\dfrac{h(\zeta)-h(-\zeta)}{\zeta}\) \ is continuous for
\(\Big[-\frac{1}{\sqrt{2}}\,e^{i\pi/4},\,
\frac{1}{\sqrt{2}}\,e^{i\pi/4}\Big]\)\textup{;}
\item[\textup{3.}]
\(h\Big(-\frac{1}{\sqrt{2}}\,e^{i\pi/4}\Big)=0\,,\, h\Big(\frac{1}{\sqrt{2}}\,e^{i\pi/4}\Big)=0\)\,.
 \end{enumerate}
 Then the limit in \eqref{FAdFOfOp} exists in the uniform operator topology.
 \end{nonumtheorem}
\begin{definition}
\label{DefFadFp}%
 Let \(h(\zeta)\) be an
\(\mathscr{F}_{_{{\scriptstyle\mathbb{R}}^{+}}}\)\,-\,admissible
 function:
 \(h(\zeta)\in\mathfrak{B}_{_{{\scriptstyle\mathscr{F}}_{_{{\mathbb{R}}^{+}}}}}\).
The operator \(h(\mathscr{F}_{\!_{\scriptstyle\mathbb{R}^{+}}})\)
 defined by \eqref{FAdFOfOp}
 is called \emph{the function \(h\) of the operator
 \(\mathscr{F}_{\!_{\scriptstyle\mathbb{R}^{+}}}\).}

The correspondence
\[h(z)\to{}h(\mathscr{F}_{\!_{\scriptstyle\mathbb{R}^{+}}}),\quad \textup{where }
h(\mathscr{F}_{\!_{\scriptstyle\mathbb{R}^{+}}}) \ \textup{is
defined by } \eqref{FAdFOfOp}\,,
\]
is said to be the
\emph{resolvent-based functional calculus for the operator~%
\(\mathscr{F}_{\!_{\scriptstyle\mathbb{R}^{+}}}\).}\\
\end{definition}

\noindent%
 \textbf{Properties of the resolvent-based functional
 calculus.}\\[-4.0ex]
\begin{lemma}
\label{compat}%
The resolvent-based
 based calculus extends the holomorphic functional
 calculus:
 \begin{enumerate}%
 \item[\textup{1}.]
The algebra
\(\textup{hol}(\textup{\large\(\sigma\)}(\mathscr{F}_{\!_{\scriptstyle\mathbb{R}^{+}}}))\)
is contained in the algebra
\(\mathfrak{B}_{_{{\scriptstyle\mathscr{F}}_{_{{\mathbb{R}}^{+}}}}}\).
\item[\textup{2}.]
For
\(h\in\textup{hol}(\textup{\large\(\sigma\)}(\mathscr{F}_{\!_{\scriptstyle\mathbb{R}^{+}}})\),
both definitions of the function \(h\) of the operator  of the
operator \(\mathscr{F}_{\!_{\scriptstyle\mathbb{R}^{+}}}\), the
definition \eqref{DunInt} and the definition \eqref{FAdFOfOp},
yield the same result, i.e. the integral in the right hand side of
\eqref{DunInt} coincides with the limit of the integrals in the
right hand side of \eqref{FAdFOfOp}\,:
\begin{equation}
\label{Coinci}
h_{\textup{hol}}(\mathscr{F}_{\!_{\scriptstyle\mathbb{R}^{+}}})=
h(\mathscr{F}_{\!_{\scriptstyle\mathbb{R}^{+}}})\,.
\end{equation}
\end{enumerate}
\end{lemma}
\begin{theorem}
\label{HomomFr}
 The resolvent-based
functional calculus
is a homomorphism of the algebra
\(\mathfrak{B}_{_{{\scriptstyle\mathscr{F}}_{_{{\mathbb{R}}^{+}}}}}\)\!\!
of \ %
\(\mathscr{F}_{\!_{\scriptstyle\mathbb{R}^{+}}}\)-\,admissible
functions into the algebra of bounded operators in
\(\mathscr{H}=L^2(\mathbb{R}^{+})\):
\begin{enumerate}
\item[\textup{1}.] If \(h(\zeta)\equiv{}1\), then \(h(\mathscr{F}_{\!_{\scriptstyle\mathbb{R}^{+}}})=
\mathscr{I}\).
\item[\textup{2}.] If \(h(\zeta)\equiv{}\zeta\), then \(h(\mathscr{F}_{\!_{\scriptstyle\mathbb{R}^{+}}})=
\mathscr{F}_{\!_{\scriptstyle\mathbb{R}^{+}}}\).
\item[\textup{3}.]If \(h(\zeta)=\alpha_1h_1(\zeta)+\alpha_2h_2(\zeta)\), where \(h_1,\,h_2%
\in\mathfrak{B}_{_{{\scriptstyle\mathscr{F}}_{_{{\mathbb{R}}_{+}}}}}\),
\(\alpha_1,\,\alpha_2\in\mathbb{C}\), then %
\(h(\mathscr{F}_{\!_{\scriptstyle\mathbb{R}^{+}}})%
=\alpha_1h_1(\mathscr{F}_{\!_{\scriptstyle\mathbb{R}^{+}}})+
\alpha_2h_2(\mathscr{F}_{\!_{\scriptstyle\mathbb{R}^{+}}})\).
\item[\textup{4}.] If \(h(\zeta)=h_1(\zeta)\cdot{}h_2(\zeta)\), where \(h_1,\,h_2%
\in\mathfrak{B}_{_{{\scriptstyle\mathscr{F}}_{_{{\mathbb{R}}^{+}}}}}\),
 then %
\(h(\mathscr{F}_{\!_{\scriptstyle\mathbb{R}^{+}}})%
=h_1(\mathscr{F}_{\!_{\scriptstyle\mathbb{R}^{+}}})\cdot{}
h_2(\mathscr{F}_{\!_{\scriptstyle\mathbb{R}^{+}}})\).
\item[\textup{5}.] If \(h%
\in\mathfrak{B}_{_{{\scriptstyle\mathscr{F}}_{_{{\mathbb{R}}_{+}}}}}\)\!\!, and
\(h^{-1}\in\mathfrak{B}_{_{{\scriptstyle\mathscr{F}}_{_{{\mathbb{R}}_{+}}}}}\)
\textup{\small(see Lemma \ref{InvCo})},
then the operator
\(h(\mathscr{F}_{\!_{\scriptstyle\mathbb{R}^{+}}})\) is
invertible, and\\
\[\big(h(\mathscr{F}_{\!_{\scriptstyle\mathbb{R}^{+}}})\big)^{-1}=%
\big(h^{-1}\big)(\mathscr{F}_{\!_{\scriptstyle\mathbb{R}^{+}}})\,.\]
\end{enumerate}
\end{theorem}
\begin{theorem}
\label{Bicont}%
  The two-sides estimate
\begin{equation}
\label{TSE}
\tfrac{1}{2}\,\|h\|_{_{\scriptstyle{\mathscr{F}}_{\mathbb{R}^{+}}}}
\leq\|(h(\mathscr{F}_{_{{\scriptstyle\mathbb{R}}^{+}}})\|%
\leq\|h\|_{_{\scriptstyle\mathscr{F}_{\mathbb{R}^{+}}}}\,,
\end{equation}
holds  for every function
\(h\in\mathfrak{B}_{_{{\scriptstyle\mathscr{F}}_{_{{\mathbb{R}}^{+}}}}}\),
where \(\|h(\mathscr{F}_{\mathbb{R}^{+}})\|\) is the norm of the
operator \(h(\mathscr{F}_{\mathbb{R}^{+}})\) considered as an
operator from \(L^2(\mathbb{R}^{+})\) into
\(L^2(\mathbb{R}^{+})\), and the norm \(\|h\|_{_{\scriptstyle\mathscr{F}_{\mathbb{R}^{+}}}}\)
of the function \(h\) is defined in \textup{Definition \ref{DFAdF}}.
\end{theorem}
\begin{theorem}%
\label{strCont}%
Let \(\{h_n\}_{n\in\mathbb{N}}\) be a sequence of functions from
\(\mathfrak{B}_{_{{\scriptstyle\mathscr{F}}_{_{{\mathbb{R}}^{+}}}}}\)
which satisfies the conditions:
\begin{enumerate}
\item[\textup{1}.] The norms \(\|h_n\|_{_{{}_{\scriptstyle\mathscr{F}_{_{\mathbb{R}^{+}}}}}}\)
are uniformly bounded:
\begin{equation}
\label{unBound}%
 \sup_{n\in\mathbb{N}}\|h_n\|_{_{{}_{\scriptstyle\mathscr{F}_{_{\mathbb{R}^{+}}}}}}<\infty\,.
\end{equation}
\item[\textup{2}.] For \(m\)\,-\,almost every \(\zeta\in\Big[-\frac{1}{\sqrt{2}}\,e^{i\pi/4},\,
\frac{1}{\sqrt{2}}\,e^{i\pi/4}\Big]\), there exist the limit
\begin{equation}
\label{limfu} h(\zeta)=\lim_{n\to\infty}h_n(\zeta)\,.
\end{equation}
\end{enumerate}
Then
\(h\in\mathfrak{B}_{_{{\scriptstyle\mathscr{F}}_{_{{\mathbb{R}}^{+}}}}}\),
and
\begin{equation}
h(\mathscr{F}_{_{{\mathbb{R}}_{+}}})=\lim_{n\to\infty}h_n(\mathscr{F}_{_{{\mathbb{R}}^{+}}}),
\end{equation}
where the limit stands for the strong convergence of a sequence of
operators.
\end{theorem}%
The next result is a spectral mapping theorem for the
\(\mathscr{F}_{\!_{{\scriptstyle\mathbb{R}}^{+}}}\)\,-\,admissible
functional calculus.
 Given a Borelian-measurable complex-valued  function \(h\),
\(h:\,\textup{\large\(\sigma\)}(\mathscr{F}_{\!_{\scriptstyle\mathbb{R}^{+}}})\to\mathbb{C}\),
 the \emph{essential \(h\)-image}
\(\big(h(\textup{\large\(\sigma\)}(\mathscr{F}_{\!_{\scriptstyle\mathbb{R}^{+}}})\big)_e\)
of the spectrum \(\textup{\large\(\sigma\)}(\mathscr{F}_{\!_{\scriptstyle\mathbb{R}^{+}}})\)
is defined as
\begin{equation}
\label{EDIS}
\big(h(\textup{\large\(\sigma\)}(\mathscr{F}_{\!_{\scriptstyle\mathbb{R}^{+}}})\big)_e%
\stackrel{\textup{\tiny def}}{=}\{z\in\mathbb{C}:\,
\underset{\zeta\in
{\textstyle\sigma}(\mathscr{F}_{\!_{\scriptstyle\mathbb{R}^{+}}})}{\textup{ess\,inf}}\,|z-h(\zeta)|=0\}\,,
\end{equation}
where the interval
\(\textup{\large\(\sigma\)}(\mathscr{F}_{\!_{\scriptstyle\mathbb{R}^{+}}})\)
is provided by the one-dimensional Lebesgue measure \(m\).
(The essential \(h\)-image \(\big(h(\textup{\large\(\sigma\)}(\mathscr{F}_{\!_{\scriptstyle\mathbb{R}^{+}}})\big)_e)\)
is determined by the \emph{mapping} \(h\) rather by the set
\(h(\textup{\large\(\sigma\)}(\mathscr{F}_{\!_{\scriptstyle\mathbb{R}^{+}}}))\).
\begin{theorem}
Let \(h\) be an
\(\mathscr{F}_{_{\scriptstyle\mathbb{R}^{+}}}\)\,-\,admissible
function. Then the spectrum
\(\textup{\large\(\sigma\)}\big(h(\mathscr{F}_{\!_{\scriptstyle\mathbb{R}^{+}}})\big)\)
of the operator
\(h(\mathscr{F}_{\!_{\scriptstyle\mathbb{R}^{+}}})\) coincides with the essential \(h\)-image
of the spectrum \(\textup{\large\(\sigma\)}(\mathscr{F}_{\!_{\scriptstyle\mathbb{R}^{+}}})\):
\begin{equation}
\label{Sfecl}%
\textup{\large\(\sigma\)}\big(h(\mathscr{F}_{\!_{\scriptstyle\mathbb{R}^{+}}})\big)=
\big(h(\textup{\large\(\sigma\)}(\mathscr{F}_{\!_{\scriptstyle\mathbb{R}^{+}}})\big)_e\,.
\end{equation}
If \(z\notin\textup{\large\(\sigma\)}\big(h(\mathscr{F}_{\!_{\scriptstyle\mathbb{R}^{+}}})\big)\),
then the resolvent
\(R_{_{\,\scriptstyle{}h(\mathscr{F}_{\!_{\mathbb{R}^{+}}})}}(z)=
 \big(z\mathscr{I}-h(\mathscr{F}_{\!_{\scriptstyle\mathbb{R}^{+}}})\big)^{-1}\) of the operator
 \(h(\mathscr{F}_{\!_{\scriptstyle\mathbb{R}^{+}}})\) is:
 \begin{subequations}
 \label{RFOOp}
 \begin{equation}
 \label{RFOOpa}
 R_{_{\,\scriptstyle{}h(\mathscr{F}_{\!_{\mathbb{R}^{+}}})}}(z)=
 r(\mathscr{F}_{\!_{\scriptstyle\mathbb{R}^{+}}})\,,
 \end{equation}
 where
 \begin{equation}
 \label{RFOOpb}
 r(\zeta)=(z-h(\zeta))^{-1}\,.
 \end{equation}
 \end{subequations}
\end{theorem}
Proofs of the results formulated in this section will be presented
in next sections.

\section{Spectral projectors which correspond to the operator
\mathversion{bold}%
\(\bm{\mathscr{F}}_{\!_{\scriptstyle\mathbb{R}^{+}}}\).}
\mathversion{normal}
 Though the operators \(\mathscr{F}_E\)
and \(\mathscr{F}^{\ast}_E\) are non-normal, the
\({\mathscr{F}}_{\!_{\scriptstyle\mathbb{R}^{+}}}\)\,-\,ad\-mis\-sible
functional calculus allows to some extent to work with these
operators as if they are self-adjoint. In particular we construct
objects which may be considered
 as resolutions of identity related to the operators \(\mathscr{F}_E\)
and \(\mathscr{F}^{\ast}_E\). The resolution of identity related
to the operator \(\mathscr{F}_E\) is a family of its spectral
projectors. We construct this family of spectral projectors as the
family of functions of the operator \(\mathscr{F}_E\) which
functions  are the indicator functions of subsets of the spectrum
\(\textup{\large\(\sigma\)}_{\mathscr{F}_E}\). Though this family
of subsets is not so rich as in the case of self-adjoint operator
and does not contain \emph{all} Borelian subsets of
\(\textup{\large\(\sigma\)}_{\mathscr{F}_E}\), it is rich enough
for our goal.
\begin{definition}
\label{DeSySe}%
 For a subset \(\Delta\) of the complex plane, we
define its \emph{symmetric part} \(\Delta_s\) and \emph{asymmetric
part} \(\Delta_a\):
\begin{equation}
\Delta_s=\Delta\cap(-\Delta), \quad
\Delta_a=\Delta\setminus(-\Delta)\,.
\end{equation}
\end{definition}

\vspace{2.0ex}
\noindent%
Here, as usual,
\(-\Delta=\lbrace\,z\in\mathbb{C}:\,-z\in\Delta\,\rbrace\). So,
\begin{equation}
\label{PrSDec} %
\Delta=\Delta_s\cup\Delta_a,\quad \Delta_s\cap\Delta_a=\emptyset,
\quad \Delta_s=-\Delta_s,\quad
\Delta_a\cap(-\Delta_a)=\emptyset\,.
\end{equation}

With every subset of \(\Delta\in\mathbb{C}\), we associate its
indicator function \(\mathds{1}_{{}_\Delta}(z)\):
\begin{equation*}
\mathds{1}_{{}_\Delta}(z)=1\ \ \textup{if} \ \ z\in\Delta,\,\,
\mathds{1}_{{}_\Delta}(z)=0\ \ \textup{if} \ \ z\not\in\Delta\,.
\end{equation*}

\begin{definition}
\label{DeEsSep}
 The set
\(S,\,S\subseteq\textup{\large\(\sigma\)}_{\mathscr{F}_E}\) is
\emph{essentially separated from zero} if
\begin{equation*}
\textup{ess\,dist}\,(S,0)>0\,.
\end{equation*}
\end{definition}

\begin{lemma}
\label{CryAdm}%
 Let \(\Delta\) be a subset of the spectrum
\(\textup{\large\(\sigma\)}_{\mathscr{F}_E}\) of the operator
\(\mathscr{F}_E\). The indicator function \(\mathds{1}_{{}_\Delta}\)
is \(\mathscr{F}_E\)-admissible if and only if the
asymmetric part \(\Delta_a\) of the set \(\Delta\) is essentially
separated from zero.
\end{lemma}
\begin{proof}
Since the function \(\mathds{1}_{{}_\Delta}(\zeta)\) is bounded (either
\(|\mathds{1}_{{}_\Delta}(\zeta)|=1\) or
\(|\mathds{1}_{{}_\Delta}(\zeta)|=0\)), the function
\(\mathds{1}_{{}_\Delta}(\zeta)\) is \(\mathscr{F}_E\)-admissible if and
only if the function
\(\dfrac{\mathds{1}_{{}_\Delta}(\zeta)-\mathds{1}_{{}_\Delta}(-\zeta)}{\zeta}\)
is essentially bounded. In view of \eqref{PrSDec},
\(\mathds{1}_{{}_\Delta}(\zeta)-\mathds{1}_{{}_\Delta}(-\zeta)=
\mathds{1}_{{}_{\Delta_a}}(\zeta)-\mathds{1}_{{}_{\Delta_{-a}}}(-\zeta)\).
From the other hand, %
\[|\mathds{1}_{\Delta_a}(\zeta)-\mathds{1}_{\Delta_{-a}}(-\zeta)|=
\mathds{1}_{{}_{\Delta_a}\cup(-\Delta_a)}(\zeta)\,,\] so the function
\(\dfrac{\mathds{1}_{{}_{\Delta}}(\zeta)-\mathds{1}_{{}_\Delta}(-\zeta)}{\zeta}\)
is essentially bounded if and only if the function
\(\dfrac{\mathds{1}_{{}_{\Delta_a\cup(-\Delta_a)}}(\zeta)}{\zeta}\) is
essentially bounded. The last function is essentially bounded if
and only if the set \(\Delta_a\cup(-\Delta_a)\) is essentially
separated from zero. From the structure of the set
\(\Delta_a\cup(-\Delta_a)\) it is clear that the set
\(\Delta_a\cup(-\Delta_a)\) is essentially separated from zero if
and only if the set \(\Delta_a\) is essentially separated from
zero.
\end{proof}

\begin{definition}
The Borelian subset \(\Delta\) of the spectrum
\(\textup{\large\(\sigma\)}(\mathscr{F}_{\!_{\scriptstyle\mathbb{R}^{+}}})\)
is said to be
\emph{\(\mathscr{F}_{\!_{\scriptstyle\mathbb{R}^{+}}}\)-admissible
set} if the indicator function \(\mathds{1}_{{}_{\Delta}}(\zeta)\) of \(\Delta\) is a
\(\mathscr{F}_{\!_{\scriptstyle\mathbb{R}^{+}}}\)-admissible
function.
\end{definition}

\begin{lemma}
Let \(\Delta_1\) and \(\Delta_2\) be
\(\mathscr{F}_{\!_{\scriptstyle\mathbb{R}^{+}}}\)-admissible
subsets of the spectrum
\(\text{\large\(\sigma\)}(\mathscr{F}_{\!_{\scriptstyle\mathbb{R}^{+}}})\).
Then the sets \(\Delta_1\cup{}\Delta_2\),
\(\Delta_1\cap{}\Delta_2\) and \(\Delta_1\setminus{}\Delta_2\) are
\mbox{\(\mathscr{F}_{\!_{\scriptstyle\mathbb{R}^{+}}}\)-admissible} sets
as well. In particular if the set \(\Delta\) is \(\mathscr{F}_{\!_{\scriptstyle\mathbb{R}^{+}}}\)-admissible, then its
complement, the set
\(\textup{\large\(\sigma\)}({\mathscr{F}_{\!_{\scriptstyle\mathbb{R}^{+}}}})\setminus\Delta\),
is \(\mathscr{F}_{\!_{\scriptstyle\mathbb{R}^{+}}}\)-admissible as well.
\end{lemma}

\begin{definition}
\label{DeSpPro}%
 Let \(\Delta\) be a \(\mathscr{F}_{\!_{\scriptstyle\mathbb{R}^{+}}}\)-admissible
subset of the spectrum
\(\text{\large\(\sigma\)}(\mathscr{F}_{\!_{\scriptstyle\mathbb{R}^{+}}})\).
The operator
\(\mathscr{P}_{\!_{\scriptstyle\mathscr{F}_{{\mathbb{R}^{+}}}}}(\Delta)\)
is defined as
\begin{equation}
\label{DSpPro}
\mathscr{P}_{\!_{\scriptstyle\mathscr{F}_{{\mathbb{R}^{+}}}}}(\Delta)%
\stackrel{\textup{\tiny
def}}{=}\mathds{1}_{{}_\Delta}(\mathscr{F}_{\!_{\scriptstyle\mathbb{R}^{+}}}),
\end{equation}
where \(\mathds{1}_{{}_\Delta}(\zeta)\) is the indicator function of the
set \(\Delta\) and the function
\(\text{\large\(\mathds{1}\)}_{{}_\Delta}(\mathscr{F}_{\!_{\scriptstyle\mathbb{R}^{+}}})\)
of the operator
\(\mathscr{F}_{\!_{\scriptstyle\mathbb{R}^{+}}}\) is understood in the sense of %
\textup{Definition~\ref{DefFadFp}}.
\end{definition}
\begin{theorem}
\label{PrSpPr}%
 The family of the operators
 \(\lbrace{}\mathscr{P}_{\!_{\scriptstyle\mathscr{F}_{{\mathbb{R}^{+}}}}}(\Delta)\rbrace_{\Delta}\),
 where \(\Delta\) runs over the family of all
 \(\mathscr{F}_{\!_{\scriptstyle\mathbb{R}^{+}}}\)-admissible sets, possesses the following
 properties:
 \begin{enumerate}
 \item[\textup{1.}]
 If the sets \(\Delta_1\) and \(\Delta_2\) are
 \(\mathscr{F}_{\!_{\scriptstyle\mathbb{R}^{+}}}\)-admissible, then
\begin{equation}
\label{PrSpPr1}%
 \mathscr{P}_{\!_{{\scriptstyle\mathscr{F}}_{\mathbb{R}^{+}}}}(\Delta_1\cap\Delta_2)=
\mathscr{P}_{\!_{{\scriptstyle\mathscr{F}}_{\mathbb{R}^{+}}}}(\Delta_1)
\cdot%
\mathscr{P}_{\!_{{\scriptstyle\mathscr{F}}_{\mathbb{R}^{+}}}}(\Delta_2);
\end{equation}
In particular, for every \(\mathscr{F}_{\!_{\scriptstyle\mathbb{R}^{+}}}\)\!-\,admissible set
\(\Delta\), the operator \\
\(\mathscr{P}_{\!_{{\scriptstyle\mathscr{F}}_{\mathbb{R}^{+}}}}(\Delta)\)
is a projector\footnote{The projector \(\mathscr{P}_{\!_{{\scriptstyle\mathscr{F}}_{\mathbb{R}^{+}}}}(\Delta)\)
may be not orthogonal. See Theorem \ref{SAdPrTh} below.}
\begin{equation}
\label{PrSpPrpr}%
\mathscr{P}^{\,2}_{\!_{{\scriptstyle\mathscr{F}}_{\mathbb{R}^{+}}}}(\Delta)=
\mathscr{P}^{\phantom{\,2}}_{\!_{{\scriptstyle\mathscr{F}}_{\mathbb{R}^{+}}}}(\Delta)\,;
\end{equation}
\item[\textup{2.}]
The projectors corresponding to the \(\mathscr{F}_{_{\scriptstyle\mathbb{R}^{+}}}\)-admissible
sets \(\emptyset\) and
\(\textup{\large\(\sigma\)}({\mathscr{F}_{{_{\scriptstyle\mathbb{R}^{+}}}}})\) are:
\begin{equation}
\label{PrSpPr2}%
\mathscr{P}_{\!_{{\scriptstyle\mathscr{F}}_{\mathbb{R}^{+}}}}(\emptyset)=0\,;\quad
\mathscr{P}_{\!_{{\scriptstyle\mathscr{F}}_{\mathbb{R}^{+}}}}%
(\textup{\large\(\sigma\)}({\mathscr{F}_{{_{\scriptstyle\mathbb{R}^{+}}}}})
=\mathscr{I},
\end{equation}
where \(\mathscr{I}\) is the identity operator in the space
\(L^2(\mathbb{R}^{+})\).
\item[\textup{3.}] The correspondence \(\Delta\to\mathscr{P}_{\!_{{\scriptstyle\mathscr{F}}_{\mathbb{R}^{+}}}}(\Delta)
L^2(\mathbb{R}^{+})\)
 between subsets of the spectrum \(\textup{\large\(\sigma\)}({\mathscr{F}_{{_{\scriptstyle\mathbb{R}^{+}}}}})\)
  and subspaces of the space \(L^2(\mathbb{R}^{+})\)
 preserves the order:
\begin{equation}
\label{PrSpPr3}%
\textup{If} \ \ \Delta_1\subset\Delta_2, \ \ \textup{then} \ \
\mathscr{P}_{\!_{{\scriptstyle\mathscr{F}}_{\mathbb{R}^{+}}}}(\Delta_1)L^2(\mathbb{R}^{+})\subseteq%
\mathscr{P}_{\!_{{\scriptstyle\mathscr{F}}_{\mathbb{R}^{+}}}}(\Delta_2)L^2(\mathbb{R}^{+})\,.
\end{equation}
 \item[\textup{4.}]
 If the sets \(\Delta_1\) and \(\Delta_2\) are
 \(\mathscr{F}_{\!_{\scriptstyle\mathbb{R}^{+}}}\)-admissible, and
 \(\Delta_1\cap\Delta_2=\emptyset\), then
\begin{subequations}
\label{AdPrSpPr}%
\begin{gather}
\label{AdPrSpPr1}%
\mathscr{P}_{\!_{{\scriptstyle\mathscr{F}}_{\mathbb{R}^{+}}}}(\Delta_1\cup\Delta_2)
=
\mathscr{P}_{\!_{{\scriptstyle\mathscr{F}}_{\mathbb{R}^{+}}}}(\Delta_1)
+
\mathscr{P}_{\!_{{\scriptstyle\mathscr{F}}_{\mathbb{R}^{+}}}}(\Delta_2)
\,;\\
\label{AdPrSpPr2}%
\mathscr{P}_{\!_{{\scriptstyle\mathscr{F}}_{\mathbb{R}^{+}}}}(\Delta_1)
\cdot
\mathscr{P}_{\!_{{\scriptstyle\mathscr{F}}_{\mathbb{R}^{+}}}}(\Delta_2)=
\mathscr{P}_{\!_{{\scriptstyle\mathscr{F}}_{\mathbb{R}^{+}}}}(\Delta_2)
\cdot
\mathscr{P}_{\!_{{\scriptstyle\mathscr{F}}_{\mathbb{R}^{+}}}}(\Delta_1)=0\,.
\end{gather}
\end{subequations}
 In particular, for every \(\mathscr{F}_{\!_{\scriptstyle\mathbb{R}^{+}}}\)\!-\,admissible set
 \(\Delta\),
 the equalities
\begin{subequations}
\label{ColCom}
\begin{gather}
\label{ColCom1}
\mathscr{P}_{\!_{{\scriptstyle\mathscr{F}}_{\mathbb{R}^{+}}}}(\Delta)+
\mathscr{P}_{\!_{{\scriptstyle\mathscr{F}}_{\mathbb{R}^{+}}}}(
\textup{\large\(\sigma\)}(\mathscr{F}_{\!_{\mathbb{R}^{+}}})\setminus\Delta)=\mathscr{I}\,;\\
\label{ColCom2}
\mathscr{P}_{\!_{{\scriptstyle\mathscr{F}}_{\mathbb{R}^{+}}}}(\Delta)\cdot
\mathscr{P}_{\!_{{\scriptstyle\mathscr{F}}_{\mathbb{R}^{+}}}}(
\textup{\large\(\sigma\)}(\mathscr{F}_{\!_{\mathbb{R}^{+}}})\setminus\Delta)=
\mathscr{P}_{\!_{{\scriptstyle\mathscr{F}}_{\mathbb{R}^{+}}}}(
\textup{\large\(\sigma\)}(\mathscr{F}_{\!_{\mathbb{R}^{+}}})\setminus\Delta)
\cdot
\mathscr{P}_{\!_{{\scriptstyle\mathscr{F}}_{\mathbb{R}^{+}}}}(\Delta)
=0
\end{gather}
\end{subequations}
hold.
\end{enumerate}
\end{theorem}
\begin{proof} The mapping \(\Delta\to\textup{\normalsize\(\mathds{1}\)}_{\Delta}(\zeta)\)
possesses the properties:\\[-3.0ex]
\begin{center}
\(\mathds{1}_{\Delta_1\cap\Delta_2}(\zeta)=\mathds{1}_{\Delta_1}(\zeta)\cdot\mathds{1}_{\Delta_2}(\zeta)\)
for every \(\Delta_1,\,\Delta_2\),\\[0.5ex]
\(\mathds{1}_{\Delta_1\cup\Delta_2}(\zeta)=\mathds{1}_{\Delta_1}(\zeta)+\mathds{1}_{\Delta_2}(\zeta)\)
if \(\Delta_1\cap\Delta_2=\emptyset\),\\[0.5ex]
\(\mathds{1}_{{}_\emptyset}(\zeta)\equiv0\),  and
\(\mathds{1}_{{}_{{\textstyle\sigma}\!({\scriptstyle\mathscr{F}}_{\!_{\mathbb{R}^{+}}})}}(\zeta)\equiv{}1\) for
\(\zeta\!\in{\textup{\large\(\sigma\)}}({\mathscr{F}_{\!_{\mathbb{R}^{+}}})}\).
\end{center}
Statements 1\,,3,\,4 of the present Theorem are consequences of
these properties of the mapping \(\Delta\to\mathds{1}_{\Delta}(\zeta)\)
and the properties of the mapping
\(\mathds{1}_{\Delta}(\zeta)\to\mathds{1}_{\Delta}(\mathscr{F}_E)\), which are
particular cases of the properties formulated as Statements
1\,-\,2 of the Theorem \ref{Homom}. The property
\(\mathscr{P}_{\mathscr{F}_E}(\emptyset)=0\) is evident.
\noindent
The property
\(\mathscr{P}_{\mathscr{F}_E}(\textup{\large\(\sigma\)}({\mathscr{F}_{\mathbb{R}^{+}}}))=\mathscr{I}\),
that is the equality
\(\text{\normalsize\(\mathds{1}\)}_{\textstyle{\sigma}({\scriptstyle\mathscr{F}_{\mathbb{R}^{+}}})}(\mathscr{F}_{\mathbb{R}^{+}})
=\mathscr{I}\)
is a consequence of the property 1 of the holomorphic functional
calculus.
 The function
\(\text{\normalsize\(\mathds{1}\)}_{\textstyle{\sigma}({\scriptstyle\mathscr{F}_{\mathbb{R}^{+}}})}(\zeta)\)
can be considered as the restriction of the function
\(h(\zeta)\equiv{}1\),
\(h\in\textup{hol}\,(\textup{\large\(\sigma\)}(\mathscr{F}_{\!_{\scriptstyle\mathbb{R}^{+}}}))\),
on the set
\(\textup{\large\(\sigma\)}(\mathscr{F}_{\!_{\scriptstyle\mathbb{R}^{+}}})\).
The property \(h({\mathscr{F}_{\mathbb{R}^{+}}})=\mathscr{I}\) is
the property 1 of the holomorphic operator calculus. According to
the holomorphic functional calculus,
 \(\big(\text{\normalsize\(\mathds{1}\)}_{\textstyle{\sigma}%
 ({\scriptstyle\mathscr{F}_{\mathbb{R}^{+}}})}\big)_{\textup{hol}}(\mathscr{F}_{\mathbb{R}^{+}})=\mathscr{I}\).
 According to  Lemma \ref{compat},
  \(\big(\text{\normalsize\(\mathds{1}\)}_{\textstyle{\sigma}%
 ({\scriptstyle\mathscr{F}_{\mathbb{R}^{+}}})}\big)_{\textup{hol}}(\mathscr{F}_{\mathbb{R}^{+}})=
 \text{\normalsize\(\mathds{1}\)}_{\textstyle{\sigma}%
 ({\scriptstyle\mathscr{F}_{\mathbb{R}^{+}}})}(\mathscr{F}_{\mathbb{R}^{+}})\).
 \end{proof}

\begin{theorem}
\label{SAdPrTh}{\ }\\[-3.0ex] %
\begin{enumerate}
\item[\textup{1.}]
 If the set
\(\Delta,\,\Delta\subseteq\textup{\large\(\sigma\)}_{_{\!{\scriptstyle\mathscr{F}}_{\mathbb{R}^{+}}}}\),
is symmetric, that is  \(\Delta_a=\emptyset_e\), then the
projector \(\mathscr{P}_{_{\!{\scriptstyle\mathscr{F}}_{\mathbb{R}^{+}}}}(\Delta)\) is an orthogonal
projector, i.e.
\begin{equation}
\label{OrPr}
\mathscr{P}_{_{\!{\scriptstyle\mathscr{F}}_{\mathbb{R}^{+}}}}(\Delta)
=\mathscr{P}_{_{\!{\scriptstyle\mathscr{F}}_{\mathbb{R}^{+}}}}(\Delta)\,.
\end{equation}
\item[\textup{2.}] If the set \(\Delta,\,\Delta\subseteq\textup{\large\(\sigma\)}_{_{\!{\scriptstyle\mathscr{F}}_{\mathbb{R}^{+}}}}\), is not
    symmetric, i.e. \(\Delta_a\not=\emptyset_e\), then the projector
    \(\mathscr{P}_{_{\!{\scriptstyle\mathscr{F}}_{\mathbb{R}^{+}}}}(\Delta)\) is  not orthogonal, i.e.
    \(\mathscr{P}_{_{\!{\scriptstyle\mathscr{F}}_{\mathbb{R}^{+}}}}(\Delta)~\not=%
    ~\mathscr{P}_{_{\!{\scriptstyle\mathscr{F}}_{\mathbb{R}^{+}}}}^{\,\ast}(\Delta)\).
\end{enumerate}
\end{theorem}
\begin{theorem}
\label{PrDSOrTh}%
 Let \(\Delta_1\) and \(\Delta_2\) are subsets of the spectrum
 \(\textup{\large\(\sigma\)}_{_{\!{\scriptstyle\mathscr{F}}_{\mathbb{R}^{+}}}}\) which satisfy the condition
 \begin{equation}
 \label{NonInt}
 \Big(\Delta_1\cup(-\Delta_1)\Big)\bigcap\Big(\Delta_2\cup(-\Delta_2)\Big)=\emptyset\,.
 \end{equation}
 Then the image subspaces \(\mathscr{P}_{_{\!{\scriptstyle\mathscr{F}}_{\mathbb{R}^{+}}}}(\Delta_1)L^2(E)\) and
 \(\mathscr{P}_{_{\!{\scriptstyle\mathscr{F}}_{\mathbb{R}^{+}}}}(\Delta_2)L^2(E)\) are mutually orthogonal, that is
 \begin{equation}
 \label{MutOrt}%
 \mathscr{P}_{_{\!{\scriptstyle\mathscr{F}}_{\mathbb{R}^{+}}}}^{\,\ast}(\Delta_2)\,%
 \mathscr{P}_{_{\!{\scriptstyle\mathscr{F}}_{\mathbb{R}^{+}}}}(\Delta_1)=0\,.
 \end{equation}
 In particular,  if
 \begin{subequations}
 \label{InPartDis}
 \begin{equation}%
 \label{InPartDis1}
 \textup{either}\ \Delta_1\subset\textup{\large\(\sigma\)}_{_{\!{\scriptstyle\mathscr{F}}_{\mathbb{R}^{+}}}}^{\,+},\  %
  \Delta_2\subset\textup{\large\(\sigma\)}_{_{\!{\scriptstyle\mathscr{F}}_{\mathbb{R}^{+}}}}^{\,+},\ | \textup{or}\ \ %
   \Delta_1\subset\textup{\large\(\sigma\)}_{_{\!{\scriptstyle\mathscr{F}}_{\mathbb{R}^{+}}}}^{\,-},\ %
 \Delta_2\subset\textup{\large\(\sigma\)}_{_{\!{\scriptstyle\mathscr{F}}_{\mathbb{R}^{+}}}}^{\,-},\
 \end{equation}
 and moreover
 \begin{equation}%
 \label{InPartDis2}
 \Delta_1\cap\Delta_2=\emptyset,
 \end{equation}
 \end{subequations}
 then \eqref{MutOrt} holds.
 \end{theorem}
 By induction with respect to \(n\), from \eqref{AdPrSpPr} the
following statement can be derived:

\begin{proposition}
\label{AddiPro}%
 Let
\(\Delta_k,\,\Delta_k\subset\textup{\large\(\sigma\)}_{_{\!{\scriptstyle\mathscr{F}}_{\mathbb{R}^{+}}}}\),
\,\(1\leq{}k\leq{}n\,\,\), be a finite sequence of sets possessing
the properties:\\[-3.0ex]
\begin{enumerate}
\item[\textup{a)}.] Each of the sets \(\Delta_k,\,1\leq{}k\leq{}n\), is
\(\mathscr{F}_{_{{\scriptstyle\mathbb{R}}^{+}}}\)-admissible;
\item[\textup{b)}.] The sets \(\Delta_k,\,1\leq{}k\leq{}n\), are
disjoint. This means that
\begin{equation*}
\Delta_p\cap\Delta_q=\emptyset,\quad
\forall\,p,\,q:\,\,1\leq{}p,\,q\leq{}n,\,p\not=q\,.
\end{equation*}
\end{enumerate}
\label{AddPropM}%
 Then the set
\(\Delta=\bigcup\limits_{1\leq{}k\leq{}n}\Delta_k\) is
\(\mathscr{F}_{_{\scriptstyle\mathbb{R}^{+}}}\)- admissible,
and
\begin{equation}
\label{AddMapp}%
 \mathscr{P}_{_{\!{\scriptstyle\mathscr{F}}_{\mathbb{R}^{+}}}}(\Delta)=
\sum\limits_{1\leq{}k\leq{}n}\mathscr{P}_{_{\!{\scriptstyle\mathscr{F}}_{\mathbb{R}^{+}}}}(\Delta_k)\,.
\end{equation}
\end{proposition}
 The property of the mapping \(\Delta\to\mathscr{P}_{_{\!{\scriptstyle\mathscr{F}}_{\mathbb{R}^{+}}}}(\Delta)\)
  expressed as Proposition \ref{AddiPro} can be naturally considered as \emph{the additivity
 of this mapping with respect to \(\Delta\).}

\emph{In general, the mapping
\(\Delta\to\mathscr{P}_{_{\!{\scriptstyle\mathscr{F}}_{\mathbb{R}^{+}}}}(\Delta)\) is not countably
additive.}If
\(\Delta_k,\,\Delta_k\subset\textup{\large\(\sigma\)}_{\mathscr{F}_E}\),
\,\(1\leq{}k<\infty,\,\) is a countable sequence of
\(\mathscr{F}_{_{{\scriptstyle\mathbb{R}}^{+}}}\)-admissible sets, then their union
\begin{equation}
\label{CounUn}
\Delta=\!\!\bigcup\limits_{1\leq{}k<\infty}\!\!\Delta_k
\end{equation}
may be a non-admissible set. Even if the set \(\Delta\),
\eqref{CounUn}, is admissible and the sets \(\Delta_k\)
 are pairwise disjoint, the equality
\begin{equation*}
 \mathscr{P}_{_{\!{\scriptstyle\mathscr{F}}_{\mathbb{R}^{+}}}}(\Delta)=
\sum\limits_{1\leq{}k<\infty}\mathscr{P}_{_{\!{\scriptstyle\mathscr{F}}_{\mathbb{R}^{+}}}}(\Delta_k)\,.
\end{equation*}
may be violated. And what is more, it may happen that despite all
the  sets \(\Delta_k,\,1\leq{}k<\infty,\) and their union
\(\Delta\) are \(\mathscr{F}_{\mathbb{R}^{+}}\)-admissible, the series in the
right hand side of of the last formula may diverge in any
reasonable sense and even
\(\|\mathscr{P}_{_{\!{\scriptstyle\mathscr{F}}_{\mathbb{R}^{+}}}}\!\!(\Delta_k)\|\to\infty\) as
\(k\to\infty\). The fact that the union \(\Delta\) of the
countable sequence of \(\mathscr{F}_{\mathbb{R}^{+}}\)-admissible sets
\(\Delta_k\) is a
 \(\mathscr{F}_E\)-admissible set does not forbid the following pathology: \emph{the
property of the sets \(\Delta_k\) be
 essentially separated from zero
 may be not uniform with respect to \(k\).} Each of the sets
 \(\Delta_k\) may be fully asymmetric: \(\Delta_k=(\Delta_k)_a\),
 but their union \(\Delta\) may be symmetric:
 \(\Delta=\Delta_s\),
 hence \(\mathscr{F}_{\mathbb{R}^{+}}\)-admissible, Lemma \ref{CryAdm}.
 However if the  property of the sets \(\Delta_k\) be essentially separated
 from zero is \emph{not uniform with respect to \(k\),} then the sequence
 \(\|\mathscr{P}_{_{\!{\scriptstyle\mathscr{F}}_{\mathbb{R}^{+}}}}(\Delta_k)\|,\,\,1\leq{}k<\infty,\)
  is unbounded.

  Nevertheless, some restricted property of countable additivity
  of the mapping \(\Delta\to\mathscr{P}_{_{\!{\scriptstyle\mathscr{F}}_{\mathbb{R}^{+}}}}(\Delta)\)
  takes place.

 \begin{theorem}
  \label{RCoAddPr}
  Let \(\lbrace{\Delta_k}\rbrace_{1\leq{}k<\infty}\) be a sequence of Borelian
  subsets of the spectrum
  \(\textup{\large\(\sigma\)}_{_{\!{\scriptstyle\mathscr{F}}_{\mathbb{R}^{+}}}}\) possessing the
  following properties:
  \begin{enumerate}
  \item[\textup{1.}] The sets \(\lbrace{}\Delta_k\rbrace_{1\leq{}k<\infty}\) are
   pairwise disjoint:
   \begin{equation}
   \label{PaWaDi}%
   \Delta_p\cap\Delta_q=\emptyset \ \ \forall\,\,p,\,q:
   \,1\leq{}p,\,q<\infty,\,p\not=q\,.
   \end{equation}
\item[\textup{2.}] The sequence
\(\lbrace{(\Delta_k)_a}\rbrace_{1\leq{}k<\infty}\)
 of the asymmetric parts  \((\Delta_k)_a\) of the sets \(\Delta_k\)
is uniformly essentially separated from zero\,, that is
\begin{equation}%
\label{USFZ}%
 \inf_k\,\textup{ess\,dist}\big((\Delta_k)_a,0\big)>0\,.
\end{equation}
\end{enumerate}
Then the set
\(\Delta=\!\!\bigcup\limits_{1\leq{}k<\infty}\!\!\Delta_k\) is
\(\mathscr{F}_{_{{\scriptstyle\mathbb{R}}^{+}}}\)-admissible  and the equality
\begin{equation}
\label{CoAdPr}
 \mathscr{P}_{_{\!{\scriptstyle\mathscr{F}}_{\mathbb{R}^{+}}}}(\Delta)=
\sum\limits_{1\leq{}k<\infty}\mathscr{P}_{_{\!{\scriptstyle\mathscr{F}}_{\mathbb{R}^{+}}}}(\Delta_k)\,.
\end{equation}
holds, where the series in the right hand side of \eqref{CoAdPr}
converges strongly.
  \end{theorem}
  \begin{lemma}
\label{EsNoPr} Let the set \(\Delta,\
\Delta\subseteq\textup{\large\(\sigma\)}({{{\mathscr{F}}_{\mathbb{R}^{+}}}})\) be
non-empty: \(\Delta\not=\emptyset_e\).
\begin{enumerate}
\item[\textup{1.}]
If the set \(\Delta\) is symmetric, i.e.
\(\textup{mes}\,\Delta_a=0\), then
\[\|\mathscr{P}_{_{\!{\scriptstyle\mathscr{F}}_{\mathbb{R}^{+}}}}(\Delta)\|=1\,.\]
\item[\textup{2.}]
If the set is not symmetric, i.e \(\Delta_a\not=\emptyset_e\),
 then
\[\|\mathscr{P}_{_{\!{\scriptstyle\mathscr{F}}_{\mathbb{R}^{+}}}}(\Delta)\|=\frac{1}{2d}\sqrt{1+2d^2},\]
where \(d=\textup{ess\,dist}\,(\Delta_a,0)\,.\) In particular,
\[\|\mathscr{P}_{_{\!{\scriptstyle\mathscr{F}}_{\mathbb{R}^{+}}}}(\Delta)\|>1\,.\]
\end{enumerate}
\end{lemma}
\begin{definition}
\label{AdmSub} To every \(\mathscr{F}_{_{\!{\scriptstyle\mathbb{R}}^{+}}}\)-admissible set
\(\Delta\)\,\,,\
\(\Delta\subseteq\textup{\large\(\sigma\)}(\mathscr{F}_E)\)\,, we
relate the subspace
\begin{equation}
\mathcal{H}_{_{\!{\scriptstyle\mathscr{F}}_{\mathbb{R}^{+}}}}(\Delta)\stackrel{\textup{\tiny def}}{=}%
 \mathscr{P}_{_{\!{\scriptstyle\mathscr{F}}_{\mathbb{R}^{+}}}}(\Delta)L^2(\mathbb{R}^+)\,,
\end{equation}
which is the image of the projector
\(\mathscr{P}_{_{\!{\scriptstyle\mathscr{F}}_{\mathbb{R}^{+}}}}\).
\end{definition}

\begin{remark}
For every admissible \(\Delta\), the subspace
\(\mathcal{H}_{_{\!{\scriptstyle\mathscr{F}}_{\mathbb{R}^{+}}}}(\Delta)\) is closed because it is
the null-subspace of the bounded operator
\(\mathcal{I}~-~\mathscr{P}_{_{\!{\scriptstyle\mathscr{F}}_{\mathbb{R}^{+}}}}(\Delta)\).
\end{remark}

Since \(\mathscr{P}_{_{\!{\scriptstyle\mathscr{F}}_{\mathbb{R}^{+}}}}(\emptyset)=0\) and
\(\mathscr{P}_{_{\!{\scriptstyle\mathscr{F}}_{\mathbb{R}^{+}}}}%
(\textup{\large\(\sigma\)}({\mathscr{F}_{\mathbb{R}^{+}}}))%
=\mathscr{I}\),
\begin{equation*}
\mathcal{H}_{_{\!{\scriptstyle\mathscr{F}}_{\mathbb{R}^{+}}}}(\emptyset)=0,\quad
\mathcal{H}_{_{{\scriptstyle\mathscr{F}}_{\mathbb{R}^{+}}}}\!%
(\textup{\large\(\sigma\)}({\mathscr{F}}_{\mathbb{R}^{+}}))=
L^2(\mathbb{R}^{+})\,.
\end{equation*}
In view of \eqref{ColCom},
 the subspace %
\(\mathcal{H}_{_{{\scriptstyle\mathscr{F}}_{\mathbb{R}^{+}}}}%
(\textup{\large\(\sigma\)}({\mathscr{F}}_{\mathbb{R}^{+}})\setminus~\Delta)\) is
the null-space of the projector \(\mathscr{P}_{_{\!{\scriptstyle\mathscr{F}}_{\mathbb{R}^{+}}}}\):
\begin{equation}
\label{NSpSP}
\mathcal{H}_{_{\!{\scriptstyle\mathscr{F}}_{\mathbb{R}^{+}}}}%
(\textup{\large\(\sigma\)}({\mathscr{F}}_{\mathbb{R}^{+}})\setminus~\Delta)=
\lbrace{}x\in{}L^2(\mathbb{R}_+):\,\mathscr{P}_{_{\!{\scriptstyle\mathscr{F}}_{\mathbb{R}^{+}}}}(\Delta)x=0\rbrace\,
\end{equation}
and the subspaces \(\mathcal{H}_{_{{\scriptstyle\mathscr{F}}_{\mathbb{R}^{+}}}}%
(\Delta)\) and
\(\mathcal{H}_{_{{\scriptstyle\mathscr{F}}_{\mathbb{R}^{+}}}}%
(\textup{\large\(\sigma\)}(\mathscr{F}_{\mathbb{R}^{+}})\setminus\Delta)\) are
complementary:
\begin{equation}
\label{CompSu} \mathcal{H}_{_{\!{\scriptstyle\mathscr{F}}_{\mathbb{R}^{+}}}}(\Delta)\dotplus
\mathcal{H}_{_{{\scriptstyle\mathscr{F}}_{\mathbb{R}^{+}}}}%
(\textup{\large\(\sigma\)}(\mathscr{F}_{\mathbb{R}^{+}})\setminus\Delta)=L^2(\mathbb{R}_+)\,.
\end{equation}
(The sum in \eqref{CompSu} is  direct).

Since the projection operator
\(\mathscr{P}_{_{\!{\scriptstyle\mathscr{F}}_{\mathbb{R}^{+}}}}(\Delta)\) is a function of the
operator \(\mathscr{F}_{_{{\scriptstyle\mathbb{R}}^{+}}}\), it commutes with \(\mathscr{F}_{_{{\scriptstyle\mathbb{R}}^{+}}}\):
\begin{equation}
\label{ComRel} \mathscr{P}_{_{\!{\scriptstyle\mathscr{F}}_{\mathbb{R}^{+}}}}(\Delta)\,%
\mathscr{F}_{_{\scriptstyle\mathbb{R}^{+}}}=
\mathscr{F}_{_{{\scriptstyle\mathbb{R}}^{+}}}%
\,\mathscr{P}_{_{\!{\scriptstyle\mathscr{F}}_{\mathbb{R}^{+}}}}(\Delta)\,.
\end{equation}
From \eqref{ComRel} it follows that the pair of complementary
subspaces \(\mathcal{H}_{_{{\scriptstyle\mathscr{F}}_{\mathbb{R}^{+}}}}(\Delta)\) and
\(\mathcal{H}_{_{{\scriptstyle\mathscr{F}}_{\mathbb{R}^{+}}}}%
(\textup{\large\(\sigma\)}({\mathscr{F}}_{\mathbb{R}^{+}})\setminus\Delta)\), %
\eqref{CompSu}, %
reduces the operator \(\mathscr{F}_{_{\!{\scriptstyle\mathbb{R}}^{+}}}\):
\begin{multline}
\label{Redu}%
\mathscr{F}_{_{\!{\scriptstyle\mathbb{R}}^{+}}}%
\mathcal{H}_{_{{\scriptstyle\mathscr{F}}_{\mathbb{R}^{+}}}}(\Delta)\subseteq
\mathcal{H}_{_{{\scriptstyle\mathscr{F}}_{\mathbb{R}^{+}}}}(\Delta),\\
\mathscr{F}_{_{{\scriptstyle\mathbb{R}}^{+}}}%
\mathcal{H}_{_{{\scriptstyle\mathscr{F}}_{\mathbb{R}^{+}}}}%
(\textup{\large\(\sigma\)}(\mathscr{F}_{_{{\scriptstyle\mathbb{R}}^{+}}})\setminus\Delta)\subseteq%
\mathcal{H}_{_{{\scriptstyle\mathscr{F}}_{\mathbb{R}^{+}}}}%
(\textup{\large\(\sigma\)}(\mathscr{F}_{_{\!{\scriptstyle\mathbb{R}}^{+}}})\setminus\Delta)\,.
\end{multline}
In particular,  the  subspace
\(\mathcal{H}_{_{{\scriptstyle\mathscr{F}}_{\mathbb{R}^{+}}}}(\Delta)\) is invariant with  respect
to the operator \(\mathscr{F}_{_{\!{\scriptstyle\mathbb{R}}^{+}}}\), and one can consider the
restriction \(\mathscr{F}_{_{\!{\scriptstyle\mathbb{R}}^{+}}}(\Delta)\) of the operator
\(\mathscr{F}_{_{\!{\scriptstyle\mathbb{R}}^{+}}}\) onto its invariant subspace
\(\mathcal{H}_{_{{\scriptstyle\mathscr{F}}_{\mathbb{R}^{+}}}}(\Delta)\):
\begin{equation}%
\label{Restr} \mathscr{F}_{_{\!{\scriptstyle\mathbb{R}}^{+}}}(\Delta)\stackrel{\textup{\tiny def}}{=}
\mathscr{F}_{_{{\scriptstyle\mathbb{R}^{+}}_{{\scriptstyle|}\mathcal{H}_{_{{\mathscr{F}}_{\mathbb{R}^{+}}}}(\Delta)}}}\,.
\end{equation}
\begin{theorem}
\label{SpeResTe}%
Let \(\Delta\) be an admissible subset of the spectrum
\(\textup{\large\(\sigma\)}(\mathscr{F}_{_{{\scriptstyle\mathbb{R}}^{+}}})\) of the operator
\(\mathscr{F}_{_{{\scriptstyle\mathbb{R}}^{+}}}\). The spectrum of the operator
\(\mathscr{F}_{_{{\scriptstyle\mathbb{R}}^{+}}}(\Delta)\), which acts in the space
\(\mathcal{H}_{_{{\scriptstyle\mathscr{F}}_{\mathbb{R}^{+}}}}(\Delta)\), is the essential closure
\(\textup{ess\,clos}\,(\Delta)\) of the set \(\Delta\):
\begin{equation}
\label{SpeRes}%
\textup{\large\(\sigma\)}(\mathscr{F}_{_{{\scriptstyle\mathbb{R}}^{+}}}(\Delta))=\textup{ess\,clos}\,(\Delta)\,.
\end{equation}
\end{theorem}
Theorem \ref{SpeResTe} justifies the following
\begin{definition}
\label{SpePro} Let
\(\Delta,\,\Delta\in\textup{\large\(\sigma\)}({\mathscr{F}_{_{{\scriptstyle\mathbb{R}}^{+}}}})\),
be a \(\mathscr{F}_{_{{\scriptstyle\mathbb{R}}^{+}}}\)-admissible set. The projector
\(\mathscr{P}_{_{{\scriptstyle\mathscr{F}}_{\mathbb{R}^{+}}}}(\Delta)\)
 defined by \eqref{DSpPro} is said to be the \(\mathscr{F}_{_{{\scriptstyle\mathbb{R}}^{+}}}\)
 \emph{spectral projector corresponding
to  the set \(\Delta\).}

The subspace \(\mathcal{H}_{_{{\scriptstyle\mathscr{F}}_{\mathbb{R}^{+}}}}(\Delta)\)\,--\,the
image of the operator
\(\mathscr{P}_{_{{\scriptstyle\mathscr{F}}_{\mathbb{R}^{+}}}}(\Delta)\)\,--\,is said to be the
\(\mathscr{F}_{_{{\scriptstyle\mathbb{R}}^{+}}}\) \emph{spectral subspace corresponding to  the
set \(\Delta\).}
\end{definition}

\begin{definition}
\label{DefSpMes}%
 The operator-valued function
\(\Delta\to\mathscr{P}_{_{{\scriptstyle\mathscr{F}}_{\mathbb{R}^{+}}}}(\Delta)\) which is defined
on the set  of all
\(\mathscr{F}_{\!_{{\scriptstyle\mathbb{R}^{+}}}}\) admissible
subsets \(\Delta\) of the spectrum
\(\textup{\large\(\sigma\)}({\mathscr{F}_{\!_{\scriptstyle\mathbb{R}^{+}}}})\)
and whose values are spectral projectors
\(\mathscr{P}_{\!_{{\scriptstyle\mathscr{F}}_{\scriptstyle\mathbb{R}^{+}}}}\!\!(\Delta)\)
of the operator
\(\mathscr{F}_{\!_{{\scriptstyle\mathbb{R}^{+}}}}\) is said to be
the \emph{spectral measure of the operator
\(\mathscr{F}_{\!_{{\scriptstyle\mathbb{R}^{+}}}}\).}
\end{definition}
We recall that the spectral measure of
\(\mathscr{F}_{\!_{{\scriptstyle\mathbb{R}^{+}}}}\) possesses some
property of sigma additivity. See Theorem \ref{RCoAddPr}.

 For \(0<\varepsilon\leq 1/\sqrt{2}\), let
\begin{equation}
\label{DeASiE}
\Delta_{+}(\varepsilon)=e^{i\pi/4}\big[\varepsilon,\,1/\sqrt{2}\big],\quad
\Delta_{-}(\varepsilon)=e^{i\pi/4}\big[-1/\sqrt{2},-\varepsilon\big]\,.
\end{equation}
Each of the sets \(\Delta_{+}(\varepsilon),
\Delta_{-}(\varepsilon)\) with \(\varepsilon>0\) is
\(\mathscr{F}_{_{\!{\scriptstyle\mathbb{R}}^{+}}}\)-admissible, however the norms of spectral
projectors
\(\mathscr{P}_{_{{\scriptstyle\mathscr{F}}_{\mathbb{R}^{+}}}}(\Delta_{+}(\varepsilon))\),
\(\mathscr{P}_{_{{\scriptstyle\mathscr{F}}_{\mathbb{R}^{+}}}}(\Delta_{-}(\varepsilon))\) tend to
\(\infty\) as \(\varepsilon\to+0\). Indeed, the sets
\(\Delta_{\pm}(\varepsilon))\) are fully
asymmetric:%
 \[\Delta_{+}(\varepsilon)=(\Delta_{+}(\varepsilon))_{a},\ \
\Delta_{-}(\varepsilon)=(\Delta_{+}(\varepsilon))_{a}.\] It is
clear that
\[\textup{ess\,dist}\big(\Delta_{+}(\varepsilon))_{a},0\big)=\varepsilon,\ \ \textup{ess\,dist}\big(\Delta_{-}(\varepsilon))_{a},0\big)=\varepsilon\]
According to Lemma \ref{EsNoPr},
\begin{equation}
\label{NOPE}
\|\mathscr{P}_{_{{\scriptstyle\mathscr{F}}_{\mathbb{R}^{+}}}}(\Delta_{+}(\varepsilon))\|=%
\frac{1}{2\varepsilon}\,\sqrt{1+2\varepsilon^2},\ \ %
\|\mathscr{P}_{_{{\scriptstyle\mathscr{F}}_{\mathbb{R}^{+}}}}(\Delta_{-}(\varepsilon))\|=%
\frac{1}{2\varepsilon}\,\sqrt{1+2\varepsilon^2}\,.
\end{equation}
In particular,
\begin{equation}
\label{SpSing}
\|\mathscr{P}_{_{{\scriptstyle\mathscr{F}}_{\mathbb{R}^{+}}}}(\Delta_{+}(\varepsilon))\|\to+\infty,\,\
\ \
\|\mathscr{P}_{_{{\scriptstyle\mathscr{F}}_{\mathbb{R}^{+}}}}(\Delta_{-}(\varepsilon))\|\to+\infty\
\ \textup{as}\ \ \varepsilon\to+0\,.
\end{equation}
At the same time, the set
\begin{equation}
\label{DeSiE}
\Delta(\varepsilon)=\Delta_{+}(\varepsilon)\cup\Delta_{-}(\varepsilon)
\end{equation}
is symmetric: \(\big(\Delta(\varepsilon)\big)_{a}=\emptyset\).
According to Lemma \ref{EsNoPr},
\begin{equation*}
\|\mathscr{P}_{_{{\scriptstyle\mathscr{F}}_{\mathbb{R}^{+}}}}(\Delta(\varepsilon))\|=1\quad
\textup{for every}\ \ \ \varepsilon>0,
\end{equation*}
or
\begin{equation}
\label{SpSingP}%
\|\mathscr{P}_{_{{\scriptstyle\mathscr{F}}_{\mathbb{R}^{+}}}}(\Delta_{+}(\varepsilon))+
\mathscr{P}_{_{{\scriptstyle\mathscr{F}}_{\mathbb{R}^{+}}}}(\Delta_{-}(\varepsilon))\|=1 \quad
\textup{for every}\ \ \ \varepsilon>0.
\end{equation}

\vspace{1.0ex} \noindent \emph{Sums of projectors from two
unbounded families form a bounded~family of projectors.}

\vspace{1.0ex} The family
\(\lbrace\Delta(\varepsilon)\rbrace_{\varepsilon>0}\) is
monotonic: \(\Delta(\varepsilon_1)\subseteq\Delta(\varepsilon_1)\)
if \(\varepsilon_1>\varepsilon_2\). Moreover,
\(\bigcup\limits_{\varepsilon>0}\Delta(\varepsilon)=
\textup{\large\(\sigma\)}(\mathscr{F}_{_{{\scriptstyle\mathbb{R}}^{+}}})\setminus0\). Since
\(\mathscr{P}_{_{{\scriptstyle\mathscr{F}}_{\mathbb{R}^{+}}}}%
(\textup{\large\(\sigma\)}(\mathscr{F}_{_{{\scriptstyle\mathbb{R}}^{+}}})\setminus0)=\\
\mathscr{P}_{_{{\scriptstyle\mathscr{F}}_{\mathbb{R}^{+}}}}%
(\textup{\large\(\sigma\)}(\mathscr{F}_{_{{\scriptstyle\mathbb{R}}^{+}}}))=\mathscr{I}\),
then, according to \textup{Theorem \ref{RCoAddPr}}, the following
assertion holds:
\begin{lemma}
\label{StCon}%
The estimate
\begin{equation}%
\label{CorEst}
\|\mathscr{P}_{_{{\scriptstyle\mathscr{F}}_{\mathbb{R}^{+}}}}(\Delta(\varepsilon))\|=1\quad
\textup{for every}\ \ \ \varepsilon>0.
\end{equation}%
and the limiting relation
\begin{equation}
\label{AprId}
\lim_{\varepsilon\to+0}\mathscr{P}_{_{{\scriptstyle\mathscr{F}}_{\mathbb{R}^{+}}}}(\Delta(\varepsilon))=\mathscr{I}\,,
\end{equation}
hold, where convergence is the strong convergence of operators.
\end{lemma}

\begin{corollary}
\label{SiOrFa} As we saw,
\begin{equation}%
\label{Unb}
\sup_{\Delta}\|\mathscr{P}_{_{{\scriptstyle\mathscr{F}}_{\mathbb{R}^{+}}}}(\Delta)\|=\infty\,,
\end{equation}%
where \(\Delta\) runs over the class of all
\(\mathscr{F}_{_{\!{\scriptstyle\mathbb{R}}^{+}}}\)-admissible sets. From this it follows that the
family
\(\big\lbrace\mathscr{P}_{_{{\scriptstyle\mathscr{F}}_{\mathbb{R}^{+}}}}(\Delta)\big\rbrace\) of
spectral projectors is not similar to an orthogonal family of
projectors.
\end{corollary}
\begin{lemma}
By contrast with \eqref{Unb},
\begin{equation}%
\label{EFPr}
\|\mathscr{F}_{_{\!{\scriptstyle\mathbb{R}}^{+}}}\mathscr{P}_{_{{\scriptstyle\mathscr{F}}_{\mathbb{R}^{+}}}}(\Delta)\|\leq \frac{\sqrt{2}+1}{2\sqrt{2}} \ \
\textup{for every admissible} \ \ \Delta\,.
\end{equation}%
\end{lemma}
In fact, following estimate holds.
\begin{lemma}
Let \(h(\zeta)\) be any \(\mathscr{F}_{_{\!{\scriptstyle\mathbb{R}}^{+}}}\)-admissible complex-valued function.
Then
\begin{subequations}
\label{EFPrS   }
\begin{enumerate}
\item[\textup{1}.]
\begin{equation}
\label{EFPrS1} \|\mathscr{F}_{_{\!{\scriptstyle\mathbb{R}}^{+}}}h(\mathscr{F}_{_{\!{\scriptstyle\mathbb{R}}^{+}}})\|\leq
\sqrt{\frac{3}{2}}\cdot\underset{\zeta\in\textup{\small\(\sigma\)}(\mathscr{F}_{_{\!{\scriptstyle\mathbb{R}}^{+}}})}{\textup{ess
sup}}\,|h(\zeta)|\,.
\end{equation}
\item[\textup{2}.] If moreover the function \(h\) is real-valued: \(h(\zeta)\in\mathbb{R}\) for
 \(\zeta\in\textup{\large\(\sigma\)}(\mathscr{F}_{_{\!{\scriptstyle\mathbb{R}}^{+}}})\),
then
\begin{equation}
\label{EFPrS2} \|\mathscr{F}_{_{\!{\scriptstyle\mathbb{R}}^{+}}}h(\mathscr{F}_{_{\!{\scriptstyle\mathbb{R}}^{+}}})\|\leq
\underset{\zeta\in\textup{\small\(\sigma\)}(\mathscr{F}_{_{\!{\scriptstyle\mathbb{R}}^{+}}})}{\textup{ess
sup}}\,|h(\zeta)|\,.
\end{equation}
\item[\textup{3}.] If the function \(h\) takes non-negative values: \(h(\zeta)\in[0,\infty)\) for
 \(\zeta\in\textup{\large\(\sigma\)}(\mathscr{F}_{_{\!{\scriptstyle\mathbb{R}}^{+}}})\),
then
\begin{equation}
\label{EFPrS3} \|\mathscr{F}_{_{\!{\scriptstyle\mathbb{R}}^{+}}}h(\mathscr{F}_{_{\!{\scriptstyle\mathbb{R}}^{+}}})\|\leq
\frac{\sqrt{2}+1}{2\sqrt{2}}\,\underset{\zeta\in\textup{\small\(\sigma\)}%
(\mathscr{F}_{_{\!{\scriptstyle\mathbb{R}}^{+}}})}{\textup{ess
sup}}\,|h(\zeta)|\,.
\end{equation}
\end{enumerate}
\end{subequations}
\end{lemma}
\section{Functions of the operator
\mathversion{bold}%
\(\bm{\mathscr{F}}_{\!_{\scriptstyle\mathbb{R}^{+}}}\) as
integrals over its spectral measure.} \mathversion{normal}
 After the spectral projectors
\(\mathscr{P}_{_{{\scriptstyle\mathscr{F}}_{\mathbb{R}^{+}}}}(\Delta)\) were introduced,
\eqref{DSpPro}, and investigated, (see in particular Lemma
\eqref{SpeResTe}), the question arises: how to represent the
original operator \(\mathscr{F}_{_{\!{\scriptstyle\mathbb{R}}^{+}}}\) in terms of these spectral
projectors. Our goal here is to give a meaning to the
representation
\begin{equation}
\label{SpRepr}
\mathscr{F}_{_{\!{\scriptstyle\mathbb{R}}^{+}}}=\int\limits_{\textup{\small\(\sigma\)}_{\!\mathscr{F}_{\!E}}}%
\zeta\,\mathscr{P}_{_{{\scriptstyle\mathscr{F}}_{\mathbb{R}^{+}}}}(d\zeta)\,,
\end{equation}
and more generally,
\begin{equation}
\emph{\emph{}}\label{SpReprf}
f(\mathscr{F}_{_{\!{\scriptstyle\mathbb{R}}^{+}}})=\int\limits_{\textup{\small\(\sigma\)}_{\!\mathscr{F}_{\!E}}}%
f(\zeta)\,\mathscr{P}_{_{{\scriptstyle\mathscr{F}}_{\mathbb{R}^{+}}}}(d\zeta)\,.
\end{equation}
We emphasize  that the operator \(\mathscr{F}_{_{\!{\scriptstyle\mathbb{R}}^{+}}}\) is non-normal,\
the family
\(\big\lbrace\mathscr{P}_{_{{\scriptstyle\mathscr{F}}_{\mathbb{R}^{+}}}}(\Delta)\big\rbrace\) is
not orthogonal and even unbounded: \eqref{Unb}. However it turns
out that if the interval \(\Delta\) is essentially separated from
zero:
\begin{equation}%
\label{EsDi}%
\textup{ess\,dist}(\Delta,0)>0,\,
\end{equation}
and the function \(f(\zeta)\) is bounded on \(\Delta\), then the
integral
\(\int\limits_{\Delta}f(\zeta)\,\mathscr{P}_{_{{\scriptstyle\mathscr{F}}_{\mathbb{R}^{+}}}}(d\zeta)\)
can be provided with a meaning.

Namely, let \(g(\zeta)\) be a simple function,  that is the
function of the form
\begin{equation}%
\label{SimRe}%
 g(\zeta)=\sum\limits_{k}a_k\mathds{1}_{\Delta_k}(\zeta)\,,
\end{equation}%
where \(a_k\) are come complex numbers, and the collections
\(\Delta_k\) of sets forms a partition (finite) of the original
set \(\Delta\): \(\Delta=\bigcup\limits_k\Delta_k,\ \
\Delta_p\cap\Delta_q=\emptyset,\,\,p\not=q\). We define the
integral
\(\int\limits_{\Delta}g(\zeta)\,\,\mathscr{P}_{_{{\scriptstyle\mathscr{F}}_{\mathbb{R}^{+}}}}(d\zeta)\)
as
\begin{equation}%
\label{DInSF}%
\int\limits_{\Delta}g(\zeta)\,\,\mathscr{P}_{_{{\scriptstyle\mathscr{F}}_{\mathbb{R}^{+}}}}(d\zeta)
\stackrel{\textup{\tiny
def}}{=}\sum\limits_{k}a_k\mathscr{P}_{_{{\scriptstyle\mathscr{F}}_{\mathbb{R}^{+}}}}(\Delta_k)\,.
\end{equation}%
The value of the sum in the right hand side of \eqref{DInSF} does
not depend on the representation of the function \(g\) in the form
\eqref{SimRe}. So, the value in the left hand side of
\eqref{DInSF} is well defined. From the other hand, decoding
definition of \(\mathscr{P}_{_{{\scriptstyle\mathscr{F}}_{\mathbb{R}^{+}}}}(\Delta_k)\) as
\(\mathds{1}_{\Delta_k}(\mathscr{F}_{_{\!{\scriptstyle\mathbb{R}}^{+}}})\), we have
\begin{equation*}
\sum\limits_{k}a_k\mathscr{P}_{_{{\scriptstyle\mathscr{F}}_{\mathbb{R}^{+}}}}(\Delta_k)=
\sum\limits_{k}a_k\mathds{1}_{\Delta_k}(\mathscr{F}_{_{\!{\scriptstyle\mathbb{R}}^{+}}})=
\big(\sum\limits_{k}a_k\mathds{1}_{\Delta_k}\big)(\mathscr{F}_{_{\!{\scriptstyle\mathbb{R}}^{+}}})=
g(\mathscr{F}_{_{\!{\scriptstyle\mathbb{R}}^{+}}})
 \,,
\end{equation*}
and finally,
\begin{equation}%
\label{DvTrak}
\int\limits_{\Delta}g(\zeta)\,\,\mathscr{P}_{_{{\scriptstyle\mathscr{F}}_{\mathbb{R}^{+}}}}(d\zeta)=
g(\mathscr{F}_{_{\!{\scriptstyle\mathbb{R}}^{+}}})\,.
\end{equation}
So \emph{for any simple  function \(g(\zeta)\)  vanishing outside
the set \(\Delta\)\,, where \(\Delta\) is separated from zero, the
integral
\(\int\limits_{\Delta}g(\zeta)\,\,\mathscr{P}_{_{{\scriptstyle\mathscr{F}}_{\mathbb{R}^{+}}}}(d\zeta)\)
is well defined and is interpreted as a function  of the operator
\(\mathscr{F}_{_{\!{\scriptstyle\mathbb{R}}^{+}}}\) in the sense of the above introduced functional
calculus. }

Given a function \(f\) bounded on \(\Delta\) and vanishing outside
of \(\Delta\), there exists sequence \(f_n\) of simple functions
vanishing outside of \(\Delta\) which converges to \(f\) uniformly
on \(\Delta\):
\begin{equation*}
\underset{n\to\infty}{\overline{\textup{lim}}}\, \underset{\zeta\in\Delta}{\textup{sup}}\,%
\big|f(\zeta)-f_n(\zeta)\big|=0\,.
\end{equation*}

\emph{The integral
\(\int\limits_{\Delta}f(\zeta)\,\,\mathscr{P}_{\mathscr{F}_E}(d\zeta)\)
will be defined as the limit of integrals
\(\int\limits_{\Delta}f_n(\zeta)\,\,\mathscr{P}_{_{{\scriptstyle\mathscr{F}}_{\mathbb{R}^{+}}}}(d\zeta)\)
of simple functions \(f_n\)} if we justify that such a limit
exists and does not depend on the approximating sequence
\(\lbrace{}f_n\rbrace\).

If \(h(\zeta)\) be a function essentially bounded on  \(\Delta\)
and vanishing outside of \(\Delta\), then
\[\|h(\mathscr{F})\|\leq{}\big(1+1/d\big)\sup_{\zeta\in\Delta}|h(\zeta)|,\]
where \(d=\textup{ess\,dist}\,(\Delta,0)\). (See \eqref{TSE} and
\eqref{NFAFAd}). Applying this estimate to \(h=f-f_n\), we see
that \(\|f(\mathscr{F}_{_{\!{\scriptstyle\mathbb{R}}^{+}}})-f_n(\mathscr{F}_{_{\!{\scriptstyle\mathbb{R}}^{+}}})\|\to0\) as
\(n\to\infty\). The convergence here is a convergence in the uniform operator topology.
According to \eqref{DvTrak}, this can be presented
as
\begin{equation*}
\bigg\|f(\mathscr{F}_{_{\!{\scriptstyle\mathbb{R}}^{+}}})-
\int\limits_{\Delta}f_n(\zeta)\,\,\mathscr{P}_{_{{\scriptstyle\mathscr{F}}_{\mathbb{R}^{+}}}}(d\zeta)\bigg\|\to0
\ \ \ \textup{as} \ \ \ n\to\infty\,.
\end{equation*} %
Thus, there exists the limit of integrals
\(\int\limits_{\Delta}f_n(\zeta)\,\,\mathscr{P}_{_{{\scriptstyle\mathscr{F}}_{\mathbb{R}^{+}}}}(d\zeta)\).
We declear this limit as
\(\int\limits_{\Delta}f(\zeta)\,\,\mathscr{P}_{_{{\scriptstyle\mathscr{F}}_{\mathbb{R}^{+}}}}(d\zeta)\):
\begin{equation*}
\int\limits_{\Delta}f(\zeta)\,\,\mathscr{P}_{_{{\scriptstyle\mathscr{F}}_{\mathbb{R}^{+}}}}(d\zeta)\stackrel{\textup{\tiny
def}}{=}
\lim_{n\to\infty}\int\limits_{\Delta}f_n(\zeta)\,\,\mathscr{P}_{_{{\scriptstyle\mathscr{F}}_{\mathbb{R}^{+}}}}(d\zeta),
\end{equation*}
where convergence is the convergence in the norm of operators
acting in \(L^2(\mathbb{R}_{+})\).

So \emph{the integral
\(\int\limits_{\Delta}f(\zeta)\,\,\mathscr{P}_{_{{\scriptstyle\mathscr{F}}_{\mathbb{R}^{+}}}}(d\zeta)\)
is defined if \(\Delta\) is any subset of
\({\text{\normalsize\(\sigma\)}}(\mathscr{F}_{\mathbb{R}^{+}})\)
 separated from zero and \(f\) is any
function  bounded  on \(\Delta\) and vanishing outside
\(\Delta\).}
Moreover, this integral can be interpreted as a function \(f\) of the operator %
\(\mathscr{F}_E\) in the sense of Definition \ref{DefFadFp}:
\begin{equation*}
\int\limits_{\Delta}f(\zeta)\,\,\mathscr{P}_{_{{\scriptstyle\mathscr{F}}_{\mathbb{R}^{+}}}}(d\zeta)=
f(\mathscr{F}_{_{\!{\scriptstyle\mathbb{R}}^{+}}}).
\end{equation*}

Let now \(f\) be any bounded function defined on the spectrum
\({\text{\normalsize\(\sigma\)}}(\mathscr{F}_{\mathbb{R}^{+}})\). (We emphasize
that the spectrum
\({\text{\normalsize\(\sigma\)}}(\mathscr{F}_{\mathbb{R}^{+}})\) is not
separated from zero, but  contains the zero point, which is  the
singular point in some sense: see \eqref{SpSing}.)
The integral \(\int\limits_{\textup{\small\(\sigma\)}%
_{\!\mathscr{F}_{\!E}}}f(\zeta)\,\,\mathscr{P}_{_{{\scriptstyle\mathscr{F}}_{\mathbb{R}^{+}}}}(d\zeta)\)
will be defined as an unproper integral. We remove a
\emph{symmetric} \(\varepsilon\)\,-\,neighborhood
\(V_{\varepsilon}\) of zero
\begin{equation}
\label{epsn}%
 V_{\varepsilon}=\big(-\varepsilon{}e^{i\pi/4},\,\varepsilon{}e^{i\pi/4}\big),
\end{equation}
 from the spectrum %
\({\text{\normalsize\(\sigma\)}}(\mathscr{F}_{\mathbb{R}^{+}})\) and  integrate
\(f\) over the set \(\Delta(\varepsilon)=
{\text{\normalsize\(\sigma\)}}(\mathscr{F}_{\mathbb{R}^{+}})~\setminus~V_{\epsilon}
\). (This is the same set \(\Delta(\varepsilon)\) that was already
defined in \eqref{DeASiE},\,\eqref{DeSiE}.)
 The set \(\Delta(\varepsilon)\) is  separated from zero, so
the integral
\(\int\limits_{\Delta(\varepsilon)}f(\zeta)\,\,\mathscr{P}_{_{{\scriptstyle\mathscr{F}}_{\mathbb{R}^{+}}}}(d\zeta)\)
is already defined.
 Then we pass to the limit
as \(\varepsilon\to+0\). If the limits exists in some sense, we
decare
the limiting operator as the integral \(\int\limits_%
{{\text{\small\(\sigma\)}}(\mathscr{F}_{\mathbb{R}^{+}})}f(\zeta)\,%
\,\mathscr{P}_{_{{\scriptstyle\mathscr{F}}_{\mathbb{R}^{+}}}}(d\zeta)\):
\begin{equation}
\label{ValPr}
\int\limits_{{\text{\small\(\sigma\)}}(\mathscr{F}_{\mathbb{R}^{+}})}f(\zeta)\,%
\,\mathscr{P}_{_{{\scriptstyle\mathscr{F}}_{\mathbb{R}^{+}}}}(d\zeta)x\,\stackrel{\textup{\tiny
def}}{=}
\lim_{\varepsilon\to+0}\int\limits_%
{{\text{\small\(\sigma\)}}(\mathscr{F}_{\mathbb{R}^{+}})\setminus{}V_{\varepsilon}}f(\zeta)\,\,%
\mathscr{P}_{_{{\scriptstyle\mathscr{F}}_{\mathbb{R}^{+}}}}(d\zeta)\,.
\end{equation}

\begin{lemma}
We assume that \(f\) is a \(\mathscr{F}_{_{\!{\scriptstyle\mathbb{R}}^{+}}}\)-admissible function.

Then the limit in \eqref{ValPr} exists in the sense of strong
convergence, that is for every \(x\in{}L^2(\mathbb{R}_{+})\),
 \begin{equation}
\label{ValPrS}
\bigg\|\int\limits_{{\text{\small\(\sigma\)}}(\mathscr{F}_{\mathbb{R}^{+}})}%
f(\zeta)\,\,\mathscr{P}_{_{{\scriptstyle\mathscr{F}}_{\mathbb{R}^{+}}}}(d\zeta)x\,-
\int\limits_{{\text{\small\(\sigma\)}}(\mathscr{F}_{\mathbb{R}^{+}})%
\setminus{}V_{\varepsilon}}f(\zeta)\,\,%
\mathscr{P}_{_{{\scriptstyle\mathscr{F}}_{\mathbb{R}^{+}}}}(d\zeta)x\bigg\|_{L^2(\mathbb{R}^{+})}
\to0 \ \ \textup{as} \ \ \varepsilon\to+0\,.
\end{equation}
Moreover
\begin{equation}
\label{InteF}
\int\limits_{{\text{\small\(\sigma\)}}(\mathscr{F}_{\mathbb{R}^{+}})}%
f(\zeta)\,\,\mathscr{P}_{_{{\scriptstyle\mathscr{F}}_{\mathbb{R}^{+}}}}(d\zeta)=
f(\mathscr{F}_{_{\!{\scriptstyle\mathbb{R}}^{+}}})\,,
\end{equation}
where the operator \(f(\mathscr{F}_{_{\!{\scriptstyle\mathbb{R}}^{+}}})\) is defined in the sense of
\textup{Definition \ref{DefFadFp}}.
\end{lemma}

\begin{proof}
 To justify the limiting relation
\eqref{ValPrS} and to establish the equality \eqref{InteF}, we
observe that
\begin{equation*}
\int\limits_{{\text{\small\(\sigma\)}}(\mathscr{F}_{\mathbb{R}^{+}})%
\setminus{}V_{\varepsilon}}f(\zeta)\,\,%
\mathscr{P}_{_{{\scriptstyle\mathscr{F}}_{\mathbb{R}^{+}}}}(d\zeta)=\int\limits_{\Delta(\varepsilon)}%
\mathds{1}_{\Delta(\varepsilon)}(\zeta)f(\zeta)\mathscr{P}_{_{{\scriptstyle\mathscr{F}}_{\mathbb{R}^{+}}}}(d\zeta)=
\big(\mathds{1}_{\Delta(\varepsilon)}f\big)(\mathscr{F}_{_{\!{\scriptstyle\mathbb{R}}^{+}}}).
\end{equation*}
(For functions vanishing outside the set \(\Delta(\varepsilon)\),
which is separated from zero, the equality \eqref{InteF} is
already established. In the present case, we apply the equality
\eqref{InteF} to the function
\(\mathds{1}_{\Delta(\varepsilon)}(\zeta)f(\zeta)\).)
 Since
\begin{equation*}%
\big(\mathds{1}_{\Delta(\varepsilon)}f\big)(\mathscr{F}_{_{\!{\scriptstyle\mathbb{R}}^{+}}})
=\mathds{1}_{\Delta(\varepsilon)}(\mathscr{F}_{_{\!{\scriptstyle\mathbb{R}}^{+}}})
f(\mathscr{F}_{_{\!{\scriptstyle\mathbb{R}}^{+}}})
=\mathscr{P}_{_{{\scriptstyle\mathscr{F}}_{\mathbb{R}^{+}}}}%
(\Delta(\varepsilon))f(\mathscr{F}_{_{\!{\scriptstyle\mathbb{R}}^{+}}})\,,
\end{equation*}
we have
\begin{equation*}%
\int\limits_{{\text{\normalsize\(\sigma\)}}(\mathscr{F}_{\mathbb{R}^{+}})%
\setminus{}V_{\varepsilon}}f(\zeta)\,\,%
\mathscr{P}_{_{{\scriptstyle\mathscr{F}}_{\mathbb{R}^{+}}}}(d\zeta)
=\mathscr{P}_{_{{\scriptstyle\mathscr{F}}_{\mathbb{R}^{+}}}}(\Delta(\varepsilon))\,%
f(\mathscr{F}_{_{\!{\scriptstyle\mathbb{R}}^{+}}})\,.
\end{equation*}%
According to Lemma \ref{StCon},
\begin{equation*}
\lim_{\varepsilon\to+0}\mathscr{P}_{\mathscr{F}_E}(\Delta(\varepsilon))\,%
f(\mathscr{F}_{_{\!{\scriptstyle\mathbb{R}}^{+}}})=
f(\mathscr{F}_{_{\!{\scriptstyle\mathbb{R}}^{+}}})\,,
\end{equation*}
 where convergence is the strong convergence
of operators. Thus under the assumptions of Lemma, there exists
the strong limit in \eqref{ValPr} and the equality \eqref{InteF}
holds.
\end{proof}%
\begin{remark}
For every fixed \(\varepsilon>0\), the spectral measure
\(\Delta\to\mathscr{P}_{\mathscr{F}_{\mathbb{R}^{+}}}\),
restricted on
\(\Delta:\,\Delta\subseteq{\text{\normalsize\(\sigma\)}}(\mathscr{F}_{\mathbb{R}^{+}})\setminus{}V_{\varepsilon}\),
is sigma-additive. Because of this, it is possible to
integrate an arbitrary bounded measurable function \(h(\zeta)\)
 over
\({\text{\normalsize\(\sigma\)}}(\mathscr{F}_{\mathbb{R}^{+}})\setminus{}V_{\varepsilon}\).
However,  the spectral measure \(\) is \emph{not}
sigma-additive and even unbounded on the family of \emph{all}
\(\mathscr{F}_{\mathbb{R}^{+}}\)\,-\,%
admissible sets \(\Delta\). \textup{(}See
\eqref{SpSing}.\textup{)} Therefore it is impossible to
integrate an \emph{arbitrary} bounded function \(h(\zeta)\) over the
whole
\({\text{\normalsize\(\sigma\)}}(\mathscr{F}_{\mathbb{R}^{+}})\).
We have to restrict ourself by bounded functions \(h\) which
furthermore have a certain symmetry near the point \(\zeta=0\).
Moreover we have to interprete the integral over
\({\text{\normalsize\(\sigma\)}}(\mathscr{F}_{\mathbb{R}^{+}})\)
as an improper integral. \textup{(}See \eqref{ValPr}.\textup{)}

 This reflect the fact that the point
\(\zeta=0\) is a spectral singularity for
\(\mathscr{F}_{\mathbb{R}^{+}}\). \textup{(\textit{See} Remark
\ref{spSing}.)}
\end{remark}
\section{ The selfadjoint differential operator
\mathversion{bold}%
\(\mathcal{L}\)\\
\hspace*{3.0ex} which commutes
with the operator
\(\bm{\mathscr{F}}_{\!_{\scriptstyle\mathbb{R}^{+}}}\).}
\mathversion{normal}
\noindent \textbf{1.}
It is well known that the eigenfunctions of the Fourier operator \(\mathscr{F}\)
 are the Hermite functions \(h_n(t)\):
\begin{equation*}
\psi_n(t)=e^{\frac{t^2}{2}}\bigg(\frac{d\,}{dt}\bigg)^ke^{-t^2}\,, \ \ (n=0,\,1,\,2,\,\ldots).
\end{equation*}
The equality
\begin{equation}
\label{hfef}
\mathscr{F}\psi_n=i^n\psi_n,\ \ n\in\mathbb{N}\,,
\end{equation}
can be checked by direct (but a little bit involved) calculation.
 The following facts helps both to guess the system \(\{\psi_{n}\}_{n\in\mathbb{N}}\) and to check the equalities \eqref{hfef}
in an organized way.
\begin{enumerate}
\item %
The selfadjoint differential operator \(\mathcal{L}_{_{\mathscr{F}}}\) generated by the formal differential operator \(L_{\mathscr{_F}}=-\frac{d^2\,}{dt^2}+t^2\)
commutes with the operator \(\mathscr{F}\). If \(x(t)\) is a smooth function and \(x(t),\,x^{\prime}(t)\),\,
\(x^{\prime\prime}(t)\) decay  fast enough as \(t\to\pm\infty\), then
\begin{equation}
\label{crho}
\mathscr{F}\mathcal{L}_{_\mathscr{F}}x=
\mathcal{L}_{_\mathscr{F}}\mathscr{F}x ;
\end{equation}
For arbitrary \(x\) from the domain of definition \(\mathcal{D}_{\mathcal{L}_{_\mathscr{F}}}\)
of the operator \(\mathcal{L}_{_\mathscr{F}}\), the equality \eqref{crho} can be justified
by passing to the limit.

\item
The functions \(\psi_n\) are eigenfunctions of the  operator \(\mathcal{L}_{_\mathscr{F}}\):
\begin{equation}
\label{efos}
\mathcal{L}_{_\mathscr{F}}\psi_n(t)=\lambda_n\psi_n(t),\ \ \lambda_n=2n+1,\ n\in\mathbb{N}\,.
\end{equation}
 This differential operator has no other eigenfunctions: if
\begin{equation}
\label{efeq}
\mathcal{L}_{_\mathscr{F}}x=\lambda\,x,\ \ x\in{}L^2(\mathbb{R}),\ x\not=0,
\end{equation}
then   \(\lambda=\lambda_n\)  and \(x\) is proportional to \(\psi_n\) for some \(n\in\mathbb{N}\).
\end{enumerate}

The spectral analysis of the operator \(-\frac{d^2\,}{dt^2}+t^2\), which is the energy operator
 of the quantum
harmonic oscillator, was done in \cite[Chapter VI, sec.\,34]{Dir}.
The operators \(\mathfrak{a}^{\phantom{\dag}}=\frac{d\,\,}{dt}+t\), \(\mathfrak{a}^{\dag}=-\frac{d\,\,}{dt}+t\) are involved essentially in this spectral analysis,
which is purely algebraic. The fact that the eigenvalues \(\lambda_n\)
of the operator \(\mathcal{L}_{_\mathscr{F}}\) are \emph{simple} (i.e. the appropriate
eigenspaces are one-dimensional) is of crucial importance.
Applying \eqref{crho} to the function \(x=\psi_n\), we see that the function
\(\mathscr{F}\psi_n\), as well as the function \(\psi_n\), is an eigenfunction of
\(\mathscr{F}\) corresponding to the eigenvalue \(\lambda_n\).
However, the eigenspace of \(\mathcal{L}_{_\mathscr{F}}\) corresponding to the eigenvalue \(\lambda_n\) is one-dimensional
and is generated by \(\psi_n\). Therefore \(\mathscr{F}\psi_n=\zeta_n\psi_n\) for
some \(\zeta_n\in\mathbb{C}\). Since \(\mathscr{F}^4=\mathscr{I}\),
\(\zeta_n^4=1\). So, \(\zeta_n\) can take only one of four values
\(1,-1,i,-i\). More detail analysis shows that \eqref{hfef} holds.

Since the spectrum of \(\mathcal{L}_{_\mathscr{F}}\) is of multiplicity one,
the operator \(\mathscr{F}\), as well as any
operator commuting with \(\mathscr{L}_{_\mathscr{F}}\), can can be interpreted as a function of
the operator \(\mathcal{L}_{_\mathscr{F}}\):
\begin{equation}
\mathscr{F}=h(\mathscr{L}_{_\mathscr{F}}),
\end{equation}
where
\begin{equation}%
h(\lambda_n)=i^n\,,\ n\in\mathbb{N}\,.
\end{equation}
In particular, we may take
\begin{equation*}
h(\mu)=e^{-i\pi/4}\cdot\,e^{i\mu\pi/4}, \ \ \mu\in\mathbb{R}\,.
\end{equation*}
 (Actually  the operator \(h(\mathcal{L}_{_\mathscr{F}})\) depends only on values of the function  \(h\) at the points \(\lambda_n,\,n\in\mathbb{N}\).)

\noindent \textbf{2.}
 To extend this way of reasoning to the operator \(\mathscr{F}_{\!_{\scriptstyle\mathbb{R}^{+}}}\), we have firstly to find the operator
 \(\mathscr{L}\) which commutes with  \(\mathscr{F}_{\!_{\scriptstyle\mathbb{R}^{+}}}\).
 It turns out that this is the differential operator \(\mathcal{L}\) generated by  the formal differential operator \(L\):
 \begin{equation}
 \label{DOCTF}
 (Lx)(t)=-\frac{d\,}{dt}\bigg(t^2\frac{dx(t)}{dt}\bigg)
 \end{equation}
This \emph{formal} operator \(L\) generates the \emph{minimal} operator
\(\mathcal{L}_{\textup{min}}\) and \emph{maximal} operator \(\mathcal{L}_{\textup{max}}\).
Namely, \(L\) describes how act the operators
\(\mathcal{L}_{\textup{min}}\) and \(\mathcal{L}_{\textup{max}}\)
on functions from the appropriate domain of definition.

\begin{definition}
\label{ACHI}%
 The set \(\mathcal{A}\) is the set of complex valued functions
 \(x(t)\) defined on the open half-axis \(\mathbb{R}^{+}\) and
 satisfying the following conditions:\\
 \hspace*{1.5ex}\textup{1.}\,
\begin{minipage}[t]{0.92\linewidth} The derivative
\(\dfrac{dx(t)}{dt}\) of the function \(x(t)\) exists at every
point \(t\) of the interval \(\mathbb{R}^{+}\);
\end{minipage}\\
\hspace*{1.5ex}\textup{2.}\,
\begin{minipage}[t]{0.92\linewidth}
The function \(\dfrac{dx(t)}{dt}\) is absolutely continuous on
every compact subinterval of the interval \(\mathbb{R}^{+}\);
\end{minipage}\\
\end{definition}
\begin{definition}
\label{ACHIz}%
The set \(\mathring{\mathcal{A}}\) is the set of complex-valued
functions \(x(t)\) defined on \(\mathbb{R}^{+}\) and
satisfied the
following conditions:\\
\hspace*{1.5ex}\textup{1.}\,
\begin{minipage}[t]{0.92\linewidth}
The function \(x(t)\) belongs to the set \(\mathcal{A}\) defined
above;
\end{minipage}\\[1.0ex]
\hspace*{1.5ex}\textup{2.}\,
\begin{minipage}[t]{0.92\linewidth}
The support \(\textup{supp}\,x\) of the function \(x(t)\) is a
compact subset of the open half-axis \(\mathbb{R}^{+}\):
\((\textup{supp}\,x)\!\Subset{}\mathbb{R}^{+}\).
\end{minipage}
\end{definition}
\begin{definition}
\label{MaxDOhi}
The differential operator \(\mathcal{L}_{\textup{max}}\) is defined as follows:\\[1.0ex]
\hspace*{1.5ex}\textup{1.}\,
\begin{subequations}
\label{maxdohi}
\begin{minipage}[t]{0.92\linewidth}
The domain of definition \(\mathcal{D}_{\mathcal{L}_{\textup{max}}}\)%
of the operator \(\mathcal{L}_{\textup{max}}\) is:
\begin{equation}%
\label{maxdo1}
\mathcal{D}_{\mathcal{L}_{\textup{max}}}=\lbrace{}x:\,x(t)\in%
L^2(\mathbb{R}^{+})\cap{\mathcal{A}}\ \  \textup{and} \ \
 (Lx)(t)\in{}L^2(\mathbb{R}^{+})\rbrace,
 \end{equation}%
 where \((Lx)(t)\) is defined\,\footnotemark%
by \eqref{DOCTF}.
\end{minipage}\\[1.0ex]
\hspace*{1.5ex}\textup{2.}\,
\begin{minipage}[t]{0.92\linewidth}
The action of the operator  \(\mathcal{L}_{\textup{max}}\) is:
\begin{equation}
\label{maxdo2}
\textup{For}\ x\in\mathcal{D}_{\mathcal{L}_{\textup{max}}}\,,\,\ %
\mathcal{L}_{{}_\textup{max}}x=Lx\,.
\end{equation}
\end{minipage}{\ }\\[2.0ex]
\end{subequations}
\footnotetext{Since \(x\in\mathcal{A}\), the expression \((Lx)(t)\) is well defined.}%
The operator \(\mathcal{L}_{\textup{max}}\) is said to be the
\emph{maximal differential operator generated by the formal
differential expression} \(L\).
\end{definition}
The minimal differential operator
\(\mathcal{L}_{{}_\textup{min}}\) is the restriction of the
maximal differential operator \(\mathcal{L}_{{}_\textup{max}}\) on
the set of functions which is some sense vanish at the endpoint of
the interval \(\mathbb{R}^{+}\). The precise definition is presented
below.
\begin{definition}
\label{MinDOha}
\begin{subequations}
\label{mindo} The operator \(\mathcal{L}_{{}_\textup{min}}\)
is the closure\,\footnote{%
Since the operator \(\mathring{\mathcal{L}}\) is symmetric, it  is closable.}%
 of the operator \(\mathring{\mathcal{L}}\):
\begin{equation}
\label{mindo1ha} \mathcal{L}_{{}_\textup{min}}=
\textup{clos}\big(\mathring{\mathcal{L}}\,\big)\,,
\end{equation}
where the operator \(\mathring{\mathcal{L}}\) is the restriction
of the operator \(\mathcal{L}_{{}_\textup{max}}\):
\begin{equation}%
\label{mindo2ha}
\mathring{\mathcal{L}}\subset{}\mathcal{L}_{\textup{max}},\quad %
\mathring{\mathcal{L}}=%
{\mathcal{L}_{\textup{max}}}_{|_{\scriptstyle\mathcal{D}_{\mathring{\mathcal{L}}}}},
\quad
\mathcal{D}_{\mathring{\mathcal{L}}}=%
\mathcal{D}_{\!{}_{\scriptstyle\mathcal{L}_{\textup{max}}}}\!\cap{}\mathring{\mathcal{A}}\,.
\end{equation}%
\end{subequations}
\end{definition}
By \(\langle\,,\,\rangle\) we denote the standard scalar product
in \(L^2(\mathbb{R}^{+})\):%
 \begin{equation*}\textup{For
}u,\,v\in{}L^2(\mathbb{R}^{+}),\ \ \ \langle{}u,v\rangle=
\int\limits_{0}^{\infty}u(t)\overline{v(t)}\,dt\,.
\end{equation*}
\vspace{1.0ex} \noindent \textsf{The properties of the operators}
\(\mathcal{L}_{{}_\textup{min}}\)
\textsf{and} \(\mathcal{L}_{{}_\textup{max}}\):\\[1.0ex]
\hspace*{2.5ex}1.\
\begin{minipage}[t]{0.90\linewidth}
\textit{The operator \(\mathcal{L}_{{}_\textup{min}}\) is
symmetric:
\begin{equation}%
\label{MinSym}%
\langle{}\mathcal{L}_{{}_\textup{min}}x,y\rangle=
\langle{}x,\mathcal{L}_{{}_\textup{min}}y{}\rangle\,,
\quad \forall {}x,\,y\in{}%
\mathcal{D}_{\!{}_{\scriptstyle\mathcal{L}_{\textup{min}}}};
\end{equation}%
In other words, the operator \(\mathcal{L}_{{}_\textup{min}}\) is
contained in its adjoint}:
\begin{equation*}%
\mathcal{L}_{{}_\textup{min}}\subseteq(\mathcal{L}_{{}_\textup{min}})^\ast\,;
\end{equation*}%
\end{minipage}{\ }\\[1.0ex]
\hspace*{2.5ex}2.\
\begin{minipage}[t]{0.90\linewidth}
\textit{The operators \(\mathcal{L}_{{}_\textup{min}}\) and
\(\mathcal{L}_{{}_\textup{max}}\) are mutually adjoint}:
\begin{equation}%
\label{ATDOho}%
(\mathring{\mathcal{L}})^\ast=%
(\mathcal{L}_{{}_\textup{min}})^\ast=\mathcal{L}_{{}_\textup{max}},\quad
(\mathcal{L}_{{}_\textup{max}})^\ast=\mathcal{L}_{{}_\textup{min}}\,;
\end{equation}%
\end{minipage}{\ }\\[1.0ex]

The fact that \((\mathring{\mathcal{L}})^\ast=\mathcal{L}_{{}_\textup{max}}\) is
a very general fact related to ordinary differential operators, regular or singular,
of finite or infinite interval.

Let as calculate the deficiency indices of the symmetric operator
\(\mathcal{L}_{{}_\textup{min}}\). In view of \eqref{ATDOho}, we
have to investigate the dimension of the space of solutions of the
equation \(\mathcal{L}_{{}_\textup{max}}x=\lambda{}x\) for
\(\lambda\) from the upper half plane and for \(\lambda\) from the
lower half plane. The equation
\(\mathcal{L}_{{}_\textup{max}}x=\lambda{}x\) is the differential
equation of the form
\begin{equation}
\label{EDDIha}
 -\frac{d\,}{dt}\bigg(t^2\frac{dx(t)}{dt}\bigg)=\lambda{}x(t)\,,
 \quad 0<t<\infty\,.
\end{equation}
We are interested in solutions of this equation which belong to
\(L^2(\mathbb{R}^{+})\).

The equation \eqref{EDDIha} can be solved explicitly.
Seeking its solution on the form \(x(t)=t^a\), we come to the equation
\begin{equation*}
a(a+1)+\lambda=0\,.
\end{equation*}
The roots of this equation are
\begin{equation}
\label{RAE}
a_1=-\frac{1}{2}+\sqrt{\frac{1}{4}-\lambda}, \quad
a_2=-\frac{1}{2}-\sqrt{\frac{1}{4}-\lambda}\,.
\end{equation}
These roots are different if \(\lambda\not=\frac{1}{4}\).
Thus if \(\lambda\not=\frac{1}{4}\)\,, the general solution of the differential equation \eqref{EDDIha} is
 of the form
\begin{equation}
\label{GeS}
x(t)=c_1t^{a_1}+c_2t^{a_2}\,,
\end{equation}
where \(c_1,\,c_2\) are arbitrary constants. If \(\lambda=\frac{1}{4}\),
the general solution of \eqref{EDDIha} is
 of the form
\begin{equation}
\label{GeSe}
x(t)=c_1t^{-1/2}+c_2t^{-1/2}\ln{}t\,.
\end{equation}
However the function \(x(t)\) of the form \eqref{GeS} (or \eqref{GeSe})
belongs to \(L^2((0,\infty))\) only if \(x(t)\equiv{}0\).
Thus, the following result is proved
\begin{lemma}
Whatever \(\lambda\in\mathbb{C}\) is, the differential equation
\eqref{EDDIha} has no solutions \(x(t)\not\equiv{}0\) belonging to
\(L^2(\mathbb{R}^{+})\).
\end{lemma}
In particular, taking \(\lambda=i\) and \(\lambda=-i\), we see
that the deficiency indices \(n_{+}\) and \(n_{-}\) of the symmetric operator
\(\mathcal{L}_{{}_\textup{min}}\)
are equal to zero. Applying the von Neumann criterion of the selfadjointness,
we obtain
\begin{lemma}
\label{MOSA}
The differential operator \(\mathcal{L}_{{}_\textup{min}}\) is self-adjoint.
\end{lemma}

In other words, we prove that
\(\mathcal{L}_{{}_\textup{min}}=\mathcal{L}_{{}_\textup{max}}\).
\begin{notation}
\label{DefL}
From now till the end this paper we use the notation \(\mathcal{L}\) for the operator
\(\mathcal{L}_{{}_\textup{min}}=\mathcal{L}_{{}_\textup{max}}\):
\begin{equation}
\label{DefLDi}
\mathcal{L}\stackrel{\textup{\tiny def}}{=}\mathcal{L}_{{}_\textup{min}}=
\mathcal{L}_{{}_\textup{max}}\end{equation}
\end{notation}
\noindent
 Since
\(\mathcal{L}=\mathcal{L}_{{}_\textup{min}}\),
\begin{equation}
\label{FNTE}
\mathcal{L}=\textup{clos}\,\mathring{\mathcal{L}}\,.
\end{equation}
Since \(\mathcal{L}=\mathcal{L}_{{}_\textup{max}}\),
\begin{equation}
\label{DDCV}
\mathcal{D}_{\mathcal{L}}=\big\lbrace{}x:\,\,x\in\mathcal{A}\cap{}L^2(\mathbb{R}^{+}),\,\,
Lx\in{}L^2(\mathbb{R}^{+})\big\rbrace\,.
\end{equation}
\noindent \textbf{3.}  The following relationship between the operators \(\mathcal{L}\)
and \(\mathscr{F}_{\!_{{\scriptstyle\mathbb{R}}^{+}}}\) helps
in the spectral analysis of the operator \(\mathscr{F}_{\!_{{\scriptstyle\mathbb{R}}^{+}}}\)
in much the same as the relationship \eqref{crho} helps in the spectral analysis of the operator
\(\mathscr{F}\).
\begin{theorem}
\label{Comhi}
The (selfadjoint) operator \(\mathcal{L}\) commutes with the truncated Fourier
operator \(\mathscr{F}_{\!_{{\scriptstyle\mathbb{R}}^{+}}}\), as well as with the adjoint operator
\(\mathscr{F}_{\!_{{\scriptstyle\mathbb{R}}^{+}}}^{\ast}\):\\[1.0ex]
\hspace*{2.5ex}\textup{1}.\
\begin{minipage}[t]{0.92\linewidth}
\hspace*{8.0ex}
If \(x\in\mathcal{D}_{\mathcal{L}}\), then \(\mathscr{F}_{\!_{{\scriptstyle\mathbb{R}}^{+}}}\,x\in\mathcal{D}_{\mathcal{L}}\),
\(\mathscr{F}_{\!_{{\scriptstyle\mathbb{R}}^{+}}}^{\ast}\,x\in\mathcal{D}_{\mathcal{L}}\).
\end{minipage}\\[1.5ex]
\hspace*{2.5ex}\textup{2}.\
\begin{minipage}[t]{0.92\linewidth}
\vspace*{-3ex}
\begin{equation}
\label{ComReha}
\mathscr{F}_{\!_{{\scriptstyle\mathbb{R}}^{+}}}\mathcal{L}\,x=
\mathcal{L}\mathscr{F}_{\!_{{\scriptstyle\mathbb{R}}^{+}}}\,x,\ \ \
\mathscr{F}_{\!_{{\scriptstyle\mathbb{R}}^{+}}}^{\ast}\mathcal{L}\,x=
\mathcal{L}\mathscr{F}_{\!_{{\scriptstyle\mathbb{R}}^{+}}}^{\ast}\,x,
\quad \forall \,x\in\mathcal{D}_{\mathcal{L}}\,.
\end{equation}
\end{minipage}\\[1.5ex]
\hspace*{2.5ex}\textup{3}.\
\begin{minipage}[t]{0.92\linewidth}
If \(\mathcal{E}(\Delta)\) is the spectral projector of the operator \(\mathcal{L}\)
corresponding to a Borelian subset \(\Delta\) of the real axis, then
\begin{equation}
\label{ComReSP}
\mathscr{F}_{\!_{{\scriptstyle\mathbb{R}}^{+}}}\mathcal{E}(\Delta)=
\mathcal{E}(\Delta)\mathscr{F}_{\!_{{\scriptstyle\mathbb{R}}^{+}}}, \ \
\mathscr{F}_{\!_{{\scriptstyle\mathbb{R}}^{+}}}^{\ast}\mathcal{E}(\Delta)=
\mathcal{E}(\Delta)\mathscr{F}_{\!_{{\scriptstyle\mathbb{R}}^{+}}}^{\ast}
\quad \forall\, \Delta\,.
\end{equation}
\end{minipage}
\end{theorem}
\begin{proof} {\ }\\ %
\(1^{\circ}.\)
Let \(x\in\mathcal{D}_{\mathring{\mathcal{L}}}\).
Then the function \(\mathscr{F}_{_{\!\scriptstyle\mathbb{R}^+}}\,x\)
is the Fourier transform of a summable  finite function, hence \(\mathscr{F}_{_{\!\scriptstyle\mathbb{R}^+}}\,x\in\mathcal{A}\).
Since \(\mathscr{F}_{_{\!\scriptstyle\mathbb{R}^+}}L^2(\mathbb{R}^{+})\subseteq{}L^2(\mathbb{R}^{+})\),
 and \(x\in{}L^2(\mathbb{R}^{+})\),
 then \(\mathscr{F}_{_{\scriptstyle\mathbb{R}^{+}}}x\in{}L^2(\mathbb{R}^{+})\).
Thus
\begin{equation}
\label{FIDd}
\mathscr{F}_{_{\!\scriptstyle\mathbb{R}^{+}}}\mathcal{D}_{\mathring{\mathcal{L}}}
\in\mathcal{A}\cap{}L^{2}(\mathbb{R}^{+})\,.
\end{equation}
\(2^{\circ}.\)   Let as before \(x\in\mathcal{D}_{\mathring{\mathcal{L}}}\).
Integrating by parts twice, we obtain
\begin{equation}
\label{siip}
\int\limits_{0}^{\infty}\left(-\frac{d\,\,}{d\xi}\bigg(\xi^2
\frac{dx(\xi)}{d\xi}\bigg)\right)e^{it\xi}d\xi=
it\int\limits_{0}^{\infty}x(\xi)\Bigg(\frac{d\,}{d\xi}\,\bigg(\xi^2e^{it\xi}\bigg)\Bigg)\,d\xi\,.
\end{equation}
Since the support of the function
\(x(t)\) is a compact subset of the open interval \(\mathbb{R}^+\),
the terms outside the integral vanish.
Transforming the integral in the right hand side of \eqref{siip}
and denoting
\begin{math}
 y(t)=\int\limits_{0}^{\infty}x(\xi)e^{it\xi}\,d(\xi),
\end{math}
 we obtain
 \begin{multline}
 \label{tra}
 -it\int\limits_{0}^{\infty}x(\xi)\Bigg(\frac{d\,}{d\xi}\,\bigg(\xi^2e^{it\xi}\bigg)\Bigg)\,d\xi=
 -it\int\limits_{0}^{\infty}x(\xi)(it\xi^2+2\xi)e^{it\xi}\,d\xi=
 \\
 =t^2\int\limits_{0}^{\infty}x(\xi)\xi^2e^{it\xi}d\xi-
 2it\int\limits_{0}^{\infty}x(\xi)\xi{}e^{it\xi}d\xi=\\
 =-t^2\frac{d^2y(t)}{dt^2}-2t\frac{dy(t)}{dt}= -\frac{d\,\,}{dt}\bigg(t^2\frac{dy(t)}{dt}\bigg)\,.
\end{multline}
The equality \eqref{tra} means that
\begin{equation}
\label{CRgha}
\mathscr{F}_{_{\!\scriptstyle \mathbb{R}^{+}}}L\,x=
L\mathscr{F}_{\!_{{\scriptstyle\mathbb{R}}^{+}}}\,x, \quad \forall x\in\mathcal{D}_{\mathring{\mathcal{L}}}\,.
\end{equation}
Since \(x\in\mathcal{D}_{\mathcal{L}}\), \(Lx\in{}L^2(\mathbb{R})\). Therefore
\(\mathscr{F}_{_{\!\scriptstyle \mathbb{R}^{+}}}L\,x\in{}L^2(\mathbb{R}^{+})\). In view of
\eqref{FIDd} and \eqref{CRgha},
\(L(\mathscr{F}_{_{\!\scriptstyle \mathbb{R}^{+}}}x)\in{}L^{2}(\mathbb{R}^{+})\). Thus
\begin{equation}
\label{ImDDha}
\mathscr{F}_{\!E}\mathcal{D}_{\mathring{\mathcal{L}}}\subseteq{}
\mathcal{D}_{\mathcal{L}}\,.
\end{equation}
\(3^{\circ}.\)
Let \(x\in\mathcal{D}_{\mathcal{L}}\) now. In view  of \eqref{FNTE},
there exists a sequence
\(x_n\in\mathcal{D}_{\mathring{\mathcal{L}}}\) such that \(x_n\to{}x,\,
\mathcal{L}x_n\to{}\mathcal{L}x\) as \(n\to\infty\).
(The convergence is the strong convergence,
that  is the convergence in \(L^2(\mathbb{R}^{+})\).) According to
\eqref{CRgha}, for every \(n\) the equality
\begin{equation}
\label{CRghan}
\mathscr{F}_{\!_{{\scriptstyle\mathbb{R}}^{+}}}\mathcal{L}\,x_n=
\mathcal{L}\mathscr{F}_{\!_{{\scriptstyle\mathbb{R}}^{+}}}\,x_n
\end{equation}
holds. The operator \(\mathscr{F}_{\!E}\) is continuous. Therefore
\(\mathscr{F}_{\!_{{\scriptstyle\mathbb{R}}^{+}}}\,x_n\to\mathscr{F}_{\!_{{\scriptstyle\mathbb{R}}^{+}}}\,x\), and
\(\mathscr{F}_{\!_{{\scriptstyle\mathbb{R}}^{+}}}\,\mathcal{L}x_n\to%
\mathscr{F}_{\!_{{\scriptstyle\mathbb{R}}^{+}}}\,\mathcal{L}x\) as \(n\to\infty\).
Now from \eqref{CRghan} it follows that there exists limit of the sequence
\(\mathcal{L}(\mathscr{F}_{\!_{{\scriptstyle\mathbb{R}}^{+}}}\,x_n)\). Since the operator \(\mathcal{L}\)
is closed, then \(\mathscr{F}_{\!_{{\scriptstyle\mathbb{R}}^{+}}}\,x\in\mathcal{D}_{\mathcal{L}}\) and
\(\mathscr{F}_{\!_{{\scriptstyle\mathbb{R}}^{+}}}\mathcal{L}\,x=
\mathcal{L}\mathscr{F}_{\!_{{\scriptstyle\mathbb{R}}^{+}}}\,x\). The inclusion
 \(\mathscr{F}^{\ast}_{\!_{{\scriptstyle\mathbb{R}}^{+}}}\,x\in\mathcal{D}_{\mathcal{L}}\) and the equality
\(\mathscr{F}^{\ast}_{\!_{{\scriptstyle\mathbb{R}}^{+}}}\mathcal{L}\,x=
\mathcal{L}\mathscr{F}^{\ast}_{\!_{{\scriptstyle\mathbb{R}}^{+}}}\,x\) can be established analogously.\\
\(4^{\circ}.\)
Since the operator \(\mathcal{L}\) is selfadjoint,
its spectrum is real. In particular,
for every non-real number \(z\), the operator \(\mathcal{L}-z\mathscr{I}\)
is invertible, and its inverse operator \( (\mathcal{L}-z\mathscr{I})^{-1}\)
is bounded and defined everywhere.
Taking  \(x=(\mathcal{L}-z\mathscr{I})^{-1}y\) in \eqref{ComReha}, where \(y\) is
an arbitrary vector,
we obtain that
\begin{equation}
\label{CreNre}
(\mathcal{L}-z\mathscr{I})^{-1}\mathscr{F}_{\!_{{\scriptstyle\mathbb{R}}^{+}}}=
\mathscr{F}_{\!_{{\scriptstyle\mathbb{R}}^{+}}}(\mathcal{L}-z\mathscr{I})^{-1}\quad
\forall\,z:\,\,\textup{Re}\,z\not=0\,.
\end{equation}
The equality \eqref{ComReSP} is a consequence of \eqref{CreNre}.
\end{proof}

\noindent
\begin{center}
\fbox{
\begin{minipage}[t]{0.9\linewidth}
Theorem  \ref{Comhi} sagest to interpret the operator
\(\mathscr{F}_{\!_{{\scriptstyle\mathbb{R}}^{+}}}\)
as a~function of the operator \(\mathcal{L}\) and to use the spectral analysis
of the \emph{selfadjoint} operator \(\mathcal{L}\) for study of the \emph{non-normal} operator
\(\mathscr{F}_{\!_{{\scriptstyle\mathbb{R}}^{+}}}\).
However since
the spectrum of the operator \(\mathcal{L}\) is of \emph{multiplicity two}, this
function should be not a \emph{scalar valued} function, but a \emph{matrix-valued} one.
\end{minipage}}
\end{center}

\section{ Spectral analysis
of the operator
\mathversion{bold}%
\(\mathcal{L}\).\\ The functional model of the operator
\mathversion{bold}%
\(\mathcal{L}\).}\mathversion{normal}
\(1^{\circ}.\) The spectral analysis of the operator \(\mathcal{L}\) can be reduced
to the spectral analysis of the operator \(-\frac{d^2\,}{ds^2}\) in \(L^2(\mathbb{R})\).
Changing variables
\begin{equation}%
\label{ChVar}
t=e^s,\,-\infty<s<\infty,\quad  z(s)=e^{s/2}x(e^s)\,,
\end{equation}%
we reduce the equation \eqref{EDDIha} to the form
\begin{equation}%
\label{RedEq}%
-\frac{d^2z(s)}{ds^2}+\frac{1}{4}z(s)=\lambda{}\,z(s),\quad -\infty<s<\infty\,.
\end{equation}%
The correspondence
\begin{equation}
\label{UnEqu}
z=Vx, \quad \text{where}\ \
z(s)=e^{s/2}x(e^s)\,,
\end{equation}
is an unitary operator from \(L^2(\mathbb{R}^{+},\,dt)\) onto
\(L^2(\mathbb{R},\,ds)\):
\begin{equation}
\label{UnEqEq}
\int\limits_{0}^{\infty}|x(t)|^2\,dt=\int\limits_{-\infty}^{\infty}|z(s)|^2ds\,.
\end{equation}
The operator \(\mathcal{L}\) is unitarily equivalent to the operator \(\mathcal{T}+\frac{1}{4}I\):
\begin{equation}%
\label{OUEO}
\mathcal{L}=V^{-1}\Big(\mathcal{T}+\frac{1}{4}I\Big)V,
\end{equation}
where
\begin{equation}%
\label{OSeD}
(\mathcal{T}z)(s)=-\frac{d^2z(s)}{ds^2}
\end{equation}
is the differential operator in \(L^2(\mathbb{R})\)
defined on the "natural" domain. The spectral structure of the operator
\(\mathcal{T}\) is well known. Its spectrum \(\sigma(\mathcal{T})\)
is absolutely continuous of \emph{multiplicity two} and
fills the positive half-axis:
\begin{math}%
\sigma(\mathcal{T})=[0,\infty)\,.
\end{math} %
The (generalized) eigenfunctions of the operator \(\mathcal{T}\) corresponding
to the point \(\rho\in(0,\infty)\) are
\begin{equation}%
\label{EFSLea}%
 z_{+}(s,\mu)=e^{i\mu{}s}\,,\ \
z_{-}(s,\mu)=e^{-i\mu{}s},\quad s\in\mathbb{R}\,,
\end{equation}
where \(\mu=\rho^{1/2}>0\,.\)
Changing variable in the expressions
\eqref{EFSLea} for eigenfunctions of the operator \(\mathcal{T}\)
according to \eqref{ChVar}, we come to the functions
\begin{equation}%
\label{EFSLni}%
 \psi^{+}(t)=t^{-\frac{1}{2}+i\mu},\quad
\psi^{-}(t)=t^{-\frac{1}{2}-i\mu}, \quad t\in\mathbb{R}^{+},\
\mu\in\mathbb{R}^{+}\,.
\end{equation}
Both of the functions \(\psi^{+}(t,\mu),\psi^{-}(t,\mu)\) are
eigenfunctions of the operator \(\mathcal{L}\) corresponding to
\emph{the same} eigenvalue \(\lambda(\mu)\),
\begin{equation}%
\label{laotmu}
\lambda(\mu)=\mu^2+1/4,\ \ \ \mu\in\mathbb{R}^{+}\,.
\end{equation}
\begin{equation}%
\label{EFEq}%
\mathcal{L}\psi^{+}(t,\mu)=\lambda(\mu)\psi^{+}(t,\mu),\quad
\mathcal{L}\psi^{-}(t,\mu)=\lambda(\mu)\psi^{-}(t,\mu)\,.
\end{equation}
 In view of \eqref{OUEO}, the spectral properties
of the operator \(\mathcal{T}\) can be reformulated as the spectral properties of
the operator \(\mathcal{L}\). Reindexing the spectral parameter \(\mu\) in such a manner
that the value of the parameter to be coincide with the eigenvalue, we come to the functions
\begin{equation}%
\label{ReIEF}%
\varphi^{+}(t,\lambda)=\psi^{+}(t,\mu(\lambda)),\quad
\varphi^{-}(t,\lambda)=\psi^{-}(t,\mu(\lambda)).
\end{equation}
where
\begin{equation}
\label{ChSpP}
\mu=\mu(\lambda)=\sqrt{\lambda-\frac{1}{4}},\ \ \mu>0\,,\ \ 1/4<\lambda<\infty\,.
\end{equation}
\begin{equation}
\label{RIEvEq}
\mathcal{L}\varphi^{+}(t,\lambda)=\lambda\varphi^{+}(t,\lambda),\quad
\mathcal{L}\varphi^{-}=\lambda\varphi^{-}(t,\lambda)\,,\ \ \
1/4<\lambda<\infty\,.
\end{equation}
In what follow we work mainly with the system
\[\big\lbrace\psi^{+}(t,\mu),\psi^{-}(t,\mu)\big\rbrace_{\mu\in(0,\infty)}\] of
"non-reindexed" eigenfunctions, but not with the system
\[\big\lbrace\varphi^{+}(t,\lambda),%
\varphi^{-}(t,\lambda)\big\rbrace_{\lambda\in(1/4,\infty)}\] of
"reindexed" eigenfunctions. The reindexing procedure is useful if
we would like to feet the eigenfunctions to the operator
\(\mathcal{L}\) in a most natural way. However, the operator
\(\mathcal{L}\) plays the heuristic role only. What we actually
need these are eigenfunctions of \(\mathcal{L}\) but not
\(\mathcal{L}\) itself.

The spectrum \(\text{\large\(\sigma\)}(\!\mathcal{L})\) of the
operator \(\mathcal{L}\) is absolutely continuous of multiplicity
two and fills of the semi-infinite interval:
\begin{equation}%
\label{specL}
\sigma(\mathcal{L})=\big[1/4,\infty\big).
\end{equation}
 To the point
\(\lambda\in(\frac{1}{4},\infty)\) of the spectrum of the operator
\(\mathcal{L}\) corresponds the \emph{two-dimensional} "generalized
eigenspace" generated by the 'generalized" eigenfunctions
\(\psi^{+}(t,\mu(\lambda)),\,\psi^{-}(t,\mu(\lambda))\),
\(\mu(\lambda)\) is defined in \eqref{ChSpP}.

Given \(\mu\in(0,\infty)\), the "eigenfunctions"
\(\psi^{+}(t,\,\mu),\,\psi^{-}(t,\,\mu)\) do not belong to the
space \(L^2(\mathbb{R}^{+},\,dt)\), but almost belong. Their averages
with respect to the spectral parameter
\[\frac{1}{2\varepsilon}%
\int\limits_{\mu-\varepsilon}^{\mu+\varepsilon}\psi^{\pm}(t,\zeta)\,d\zeta=
t^{-\frac{1}{2}\pm{}i\mu}\frac{\sin(\varepsilon\ln{t})}{\varepsilon\ln{t}}\]
over an arbitrary small interval
\((\mu-\varepsilon,\mu+\varepsilon)\subset(0,\,\infty)\) already
belongs to \(L^2((0,\infty),\,dt)\). These eigenfunctions satisfy
the generalized "orthogonality relations":
\begin{gather}
\int\limits_{0}^{\infty}\overline{\psi^{-}(t,\mu_2)}\,\psi^{+}(t,\mu_1)\,dt=0
\ \ ,\notag\\
\int\limits_{0}^{\infty}\overline{\psi^{+}(t,\mu_2)}\,\psi^{+}(t,\mu_1)\,dt=
2\pi\,\delta(\mu_1-\mu_2),\notag\\
\int\limits_{0}^{\infty}\overline{\psi^{-}(t,\mu_2)}\,\psi^{-}(t,\mu_1)\,dt=
2\pi\,\delta(\mu_1-\mu_2),\notag\\
\forall\, \mu_1,\mu_2>0,
\ \ \textup{where}\ \ \delta\ \ \ \textup{is the Dirac \(\delta\)-function.}
\label{GeOrRe}
\end{gather}
The integrals in \eqref{GeOrRe} diverge, so the relations \eqref{GeOrRe} are nonsense
if they are understood literally. Nevertheless the equalities \eqref{GeOrRe} can be provide
with  a meaning in the sense of distributions.

However we prefer to stay on the `classical' point of view, and to
to formulate the `orthogonality properties' of the
`eigenfunctions' \(\psi_{\pm}(t,\lambda)\) in the language of the
\(L^2\)-theory of the Fourier integrals.\\[1.0ex]

\noindent \textbf{Notation.} In what follows we use the matrix
notation, matrix language, and matrix operations. We organize the
pair \(\psi_{+}(t,\mu),\,\psi_{-}(t,\mu)\), \eqref{EFSLni}, into
the matrix-row
\begin{subequations}
\label{matrpsi}
\begin{equation}
\label{matrpsi1}%
\psi(t,\mu)=
\begin{bmatrix}
\,\psi_{+}(t,\mu)\!&\!
\psi_{-}(t,\mu)\,
\end{bmatrix}\,.
\end{equation}
According to the matrix algebra notation, the matrix adjoint to
the matrix-column \(\psi(t,\mu)\) is the matrix-column
\begin{equation}
\label{matrpsi2}
 \psi^{\ast}(t,\mu)=
\begin{bmatrix}
\vspace*{-2.3ex}\\\overline{\psi_{+}(t,\mu)}\\[1.3ex]\overline{\psi_{-}(t,\mu)}
\end{bmatrix}\,.
\end{equation}
\end{subequations}
In this notation, the orthogonality relation \eqref{GeOrRe} can be
presented  as
\begin{equation}
\label{GeOrt}
\int\limits_{t\in\mathbb{R}^{+}}%
\psi^{\ast}(t,\mu_{2})\psi(t,\mu_{1})\,dt=
2\pi\,\delta(\mu_1-\mu_2)I,
\end{equation}
where \(I\) is the \(2\times2\) identity matrix.

\noindent%
\(2^{\circ}.\) One of fundamental results of the theory of selfadjoint operators
can be roughly formulated as follows.

\emph{Every selfadjoint operator is unitary equivalent
to a multiplication operator in a space of square integrable vector functions on a subset of the real axis.}

In the case of the operator \(\mathcal{L}\), this is a space of two-component
vector functions defined on the positive half-axis \(\mathbb{R}^{+}\) which are square integrable with respect to the Lebesgue measure on \(\mathbb{R}^{+}\).
(The number \emph{two} is the spectral multiplicity of the operator \(\mathcal{L}\).)
The precise definition is as follows.

\begin{definition}
\label{DefModSp}
The space \(\mathscr{K}\) is the space of all two component vector functions
\(y(\mu)=\begin{bmatrix}
y^{+}(\mu)\\[0.8ex]y^{-}(\mu)
\end{bmatrix}\)
which are defined on the positive half-axis \(\mu\in\mathbb{R}^{+}\), are Borel measurable
and satisfy the condition
\begin{equation}
\label{sqIntMoSp}
\int\limits_{\mu\in\mathbb{R}^{+}}y^{\ast}(\mu)\,y(\mu)\,\frac{d\mu}{2\pi}<\infty\,.
\end{equation}
 The set \(\mathscr{K}\) provided by pointwise algebraic operations and the scalar product
\begin{equation}
\label{ScPrMoSp_}
\langle\,y_1\,,\,y_2\,\rangle_{_{\scriptstyle\mathscr{K}}}=
\int\limits_{\mu\in\mathbb{R}^{+}}y_2^{\ast}(\mu)\,y_2(\mu)\,\frac{d\mu}{2\pi}\,\cdot
\end{equation}
is a Hilbert space.
(Strictly speaking elements of \(\mathscr{K}\) are equivalency classes of vector functions.)

We use the terminology \emph{the model space}  for the Hilbert space \(\mathscr{K}\).
\end{definition}

\noindent%
\(3^{\circ}.\) With the system \(\{\psi(t,\,\mu)\}_{\mu\in\mathbb{R}^{+}}\) of
eigenfunctions of the operator \(\mathcal{L}\) we relate the `Fourier transform'
\(x(t)\to\hat{x}(\mu)\) in this eigenfunctions.
For a function \(x\in{}L^2(\mathbb{R}^{+})\)
 compactly supported on \(\mathbb{R}^{+}\): \(\textup{supp\,}x\Subset{}\mathbb{R}^{+}\),
we  set
 \begin{equation}
 \label{FTEiFu}
\hat{x}(\mu)=\int\limits_{t\in{}\mathbb{R}^{+}}\psi^{\ast}(t,\,\mu)\,x(t)\,dt\,,\ \ \mu\in\mathbb{R}^{+}\,.
 \end{equation}
 For a vector-function \(y(\mu)\in\mathscr{K}\) compactly supported on \(\mathbb{R}^{+}\): \(\textup{supp\,}y\Subset{}\mathbb{R}^{+}\), we  set
 \begin{equation}
 \label{FiTEiFu}
\check{y}(t)=\int\limits_{\mu\in{}\mathbb{R}^{+}}\psi(t,\,\mu)\,y(\mu)\,\frac{d\mu}{2\pi}\,,\ \ %
t\in\mathbb{R}^{+}\,.
 \end{equation}
 Since \(\psi(t,\mu)\psi^{\ast}(t,\mu)=2t^{-1}\) for every %
 \(t\in\mathbb{R}^{+}\!,\,\mu\in\mathbb{R}^{+}\), the integrals in \eqref{FTEiFu}, \eqref{FTEiFu}
 converge absolutely for every \(\mu\in\mathbb{R}^{+}\), \(t\in\mathbb{R}^{+}\) respectively.
  Thus for compactly supported
 \(x\in{}L^2(\mathbb{R}^{+})\), \(y\in\mathscr{K}\), the functions \(\hat{x}(\mu)\), \(\check{y}(t)\)
 are well defined function on the half-axis \(\mu\in\mathbb{R}^{+}\), \(t\in\mathbb{R}^{+}\)
 respectively.
\begin{remark}
\label{Mell}
The transformations \eqref{FTEiFu} and \eqref{FiTEiFu} are closely related to the Mellin transforms
(direct and inverse). Let \(x(t)\) be a function defined on \(\mathbb{R}^{+}\). The Mellin transform
\(\big(\mathscr{M}x\big)(\zeta)\) of the function \(x\) is defined as
\begin{equation}
\label{melt}
\hat{x}(\zeta)=\big(\mathscr{M}x\big)(\zeta)=\int\limits_{0}^{\infty}t^{\,\zeta}\,x(t)\frac{dt}{t}
\end{equation}
The inverse Mellin transform is
\begin{equation}
\label{imelt}
x(t)=\big(\mathscr{M}^{-1}x\big)=\frac{1}{2\pi{}i}
\int\limits_{c-i\infty}^{c+i\infty}t^{\,-\zeta}\,\hat{x}(\zeta)\,d\zeta\,.
\end{equation}
which is defined for those complex \(\zeta\) for which the integral exists.
The restriction of the
 Mellin transform \eqref{melt} on the vertical line \mbox{\(\zeta=1/2+i\mu\),}
\(\mu\in\mathbb{R}\), coincides essentially with the transform \(x(t)\to\hat{x}(\mu)\)
defined by \eqref{FTEiFu}. The inverse Mellin transform \eqref{imelt}, with the value \(c=1/2\),
 coincides essentially with the transform \(y(\mu)\to\check{y}(t)\) defined by
\eqref{FiTEiFu}.
\end{remark}
\begin{lemma} {\ }  %
 \label{PaIdeL1}%
 \begin{enumerate}
 \item[\textup{1}.] Assume that \(x\in{}L^2(\mathbb{R}^{+})\) and \(\textup{supp}\,x\Subset{}\mathbb{R}^{+}\). Let us define the function
  \(\hat{x}(\mu)\) by \eqref{FTEiFu}. Then the Parseval equality
  \begin{equation}
  \label{PaIde1}
  \int\limits_{\mu\in\mathbb{R}^{+}}\hat{x}^{\ast}(\mu)\hat{x}(\mu)\frac{d\mu}{2\pi}=
  \int\limits_{t\in\mathbb{R}^{+}}\overline{x}(t)x(t)\,dt\,.
  \end{equation}
  holds.
  \item[\textup{2}.] Assume that \(y\in\mathscr{K}\) and \(\textup{supp\,}y\Subset{}\mathbb{R}^{+}\). Let us define the function
  \(\check{y}(t)\) by \eqref{FiTEiFu}. Then the Parseval equality
  \begin{equation}
  \label{PaIde2}
  \int\limits_{t\in\mathbb{R}^{+}}\overline{\check{y}}(t)\check{y}(t)\,dt=
  \int\limits_{\mu\in\mathbb{R}^{+}}y^{\ast}(\mu)y(\mu)\frac{d\mu}{2\pi}
  \,.
  \end{equation}
  holds.
  \end{enumerate}
 \end{lemma}
 This lemma says that the
 mappings \(x(t)\to\hat{x}(\mu)\), \(y(\mu)\to\check{y}(t)\), defined so far only  for
 compactly supported \(x\in{}L^{2}(\mathbb{R}^{+})\), \(y\in\mathscr{K}\) , are isometric mappings from \(L^{2}(\mathbb{R}^{+})\) into \(\mathscr{K}\) and from \(\mathscr{K}\) into \(L^{2}(\mathbb{R}^{+})\) respectively. In particular, each of these two mappings is \emph{uniformly continuous} on its domain of definition. The sets of compactly supported \(x\), \(y\)
 are \emph{dense subsets} of the spaces \(L^2(\mathbb{R}^{+})\) and \(\mathscr{K}\) respectively. Therefore the mapping \(x(t)\to\hat{x}(\mu)\) can be extended to a continuous mapping from \(L^2(\mathbb{R}^{+})\) to \(\mathscr{K}\)
 defined on the the whole space \(L^2(\mathbb{R}^{+})\). Analogously the mapping \(y(\mu)\to\check{y}(t)\) can be extended to a continuous mapping from \(L^2(\mathbb{R}^{+})\) to \(\mathscr{K}\) defined on the the whole space \(\mathscr{K}\).

 \begin{definition} {\ }%
 \label{ParExtL}%
 \begin{enumerate}%
 \item[\textup{1.}]
 The mapping \(U\), \(U:L^2(\mathbb{R}^{+})\to\mathscr{K}\), is the continuous mapping
 defined on the whole space \(L^2(\mathbb{R}^{+})\) which acts as
 \begin{equation}
 \label{unit}
 (Ux)(\mu)=\int\limits_{t\in{}\mathbb{R}^{+}}\psi^{\ast}(t,\,\mu)\,x(t)\,dt
 \end{equation}
  for \(x\in{}L^2(\mathbb{R})\) which are compactly supported on \(\mathbb{R}^{+}\).
 \item[\textup{2.}]
 The mapping \(U^{-1}\), \(U^{-1}\!:\mathscr{K}\to{}L^2(\mathbb{R}^{+})\), is the continuous mapping
 defined on the whole space \(\mathscr{K}\) which acts as
 \begin{equation}
 \label{uniti}
 (U^{-1}y)(t)=\int\limits_{\mu\in{}\mathbb{R}^{+}}\psi(t,\,\mu)\,y(\mu)\,\frac{d\mu}{2\pi}
 \end{equation}
 for \(y\in{}\mathscr{K}\) which are compactly supported on \(\mathbb{R}^{+}\).\\
   \textup{(}Here \(U^{-1}\) is a notation only.
 We do not claim for the present that the operator \(U^{-1}\) is the operator inverse to \(U\).\textup{)}
 \end{enumerate}
 \end{definition}
 \begin{remark}
 \label{GeDe}
 If \(x\in{}L^2(\mathbb{R}^{+})\), but \(x\) is not compactly supported on \(\mathbb{R}^{+}\),
 than the function \(\psi(t,\mu)x(t)\) may be not integrable, i.e. the
  integral in \eqref{FTEiFu} may not exist in the proper sense. In this case, the
 function \(Ux\) is defined by means of the above mentioned extension procedure rather
 by the integral \eqref{FTEiFu}.
 Nevertheless we will use the expression \eqref{FTEiFu} for the function \(\hat{x}=Ux\) as a \emph{notation} no matter whether the integral in \eqref{FTEiFu} exists or not as
 a Lebesgue integral. The same is also related to the function \(\check{y}=U^{-1}y\).
 \end{remark}
 \begin{theorem} {\ } %
 \label{UbMaToM}
 \begin{enumerate}
 \item[\textup{1.}]
 The mapping \(U\) introduced in \textup{Definition \ref{ParExtL}} is an unitary mapping from the space
 \(L^2(\mathbb{R}^{+})\) \textsf{on}to the model space \(\mathscr{K}\).
 \item[\textup{2.}] The mapping \(U^{-1}\) introduced in \textup{Definition \ref{ParExtL}} is an unitary mapping from the space \(\mathscr{K}\)
  \textsf{on}to the model space \(L^2(\mathbb{R}^{+})\).
 \item[\textup{3.}] The mappings \(U\) and \(U^{-1}\) are mutually inverse:
 \begin{equation}
 \label{MuIn}
 U^{-1}U=\mathscr{I}_{_{\!\scriptstyle{L^{2}(\mathbb{R}^{+})}}},\ \
 UU^{-1}=\mathscr{I}_{\!_{\scriptstyle\mathscr{K}}}\,.
 \end{equation}
 \end{enumerate}
 \end{theorem}

Statement 3 of Theorem \ref{UbMaToM} claims that for arbitrary \(x\in{}L^2(\mathbb{R}^{+})\)
the pair of formulas
\begin{subequations}
\label{PInF}
\begin{align}
\label{PInF1}
\hat{x}(\mu)&=\int\limits_{t\in{}\mathbb{R}^{+}}\psi^{\ast}(t,\,\mu)\,x(t)\,dt\,,\\
\label{PInF2}
x(t)&=\int\limits_{\mu\in{}\mathbb{R}^{+}}\psi(t,\,\mu)\,\hat{x}(\mu)\,\frac{d\mu}{2\pi}
\end{align}
\end{subequations}
holds. This pair of formulas can be considered as the \emph{expansion of arbitrary
\(x\in{}L^2(\mathbb{R}^{+})\) in eigenfunction of the operator} \(\mathcal{L}\).
 \begin{proof}[Proof of Lemma \ref{PaIdeL1} and Theorem \ref{UbMaToM}.] {\ }\\ %
Lemma \ref{PaIdeL1} and Theorem \ref{UbMaToM} are paraphrases of the classical facts from the  \(L^{2}\)-theory of the Fourier integral.
Given the function \(z(s)\in{}L^2(\mathbb{R},ds)\), its Fourier transform
\(\tilde{z}(\mu)\) is \[\tilde{z}(\mu)=\int\limits_{s\in\mathbb{R}}z(s)e^{-i\mu{}s}ds\,,
\ \ \mu\in\mathbb{R}\,.\]
We split the function \(\tilde{z}(\mu)\) into the pair \(\tilde{z}_{+}(\mu)\) and
\(\tilde{z}_{-}(\mu)\), both functions \(\tilde{z}_{+}(\mu)\) and \(\tilde{z}_{-}(\mu)\)
are defined for \(\mu\in\mathbb{R}^{+}\):
\begin{equation}
\label{FtrSpl}
\tilde{z}_{+}(\mu)=\int\limits_{s\in\mathbb{R}}z(s)e^{-i\mu{}s}\,ds,\quad
\tilde{z}_{-}(\mu)=\int\limits_{s\in\mathbb{R}}z(s)e^{i\mu{}s}\,ds,\quad\mu\in\mathbb{R}^{+}\,.
\end{equation}
The Parseval identity and the inversion formula can by presented in the form
\begin{equation}
\label{ParIdSpl}%
\int\limits_{s\in\mathbb{R}}|z(s)|^2ds=
\int\limits_{\mu\in\mathbb{R}^{+}}|\tilde{z}_{+}(\mu)|^2\frac{d\mu}{2\pi}+
\int\limits_{\mu\in\mathbb{R}^{+}}|\tilde{z}_{-}(\mu)|^2\frac{d\mu}{2\pi}\,,
\end{equation}
and
\begin{equation}
\label{InFSpl}%
z(s)=\int\limits_{\mathbb{R}^{+}}\tilde{z}_{+}(\mu)\,e^{i\mu{}s}\frac{d\mu}{2\pi}+
\int\limits_{\mu\in\mathbb{R}^{+}}\tilde{z}_{-}(\mu)\,e^{-i\mu{}s}\frac{d\mu}{2\pi}\,.
\end{equation}
Changing variable
\[x(t)=t^{-1/2}z(\ln{}t),\,\ \mu(\lambda)=\sqrt{\lambda-1/4}\,,\]
(see \eqref{ChVar}), we present the formulas \eqref{FtrSpl} and \eqref{ParIdSpl}
in the form \eqref{FTEiFu} and \eqref{PaIde1} respectively. The inversion formula
\eqref{InFSpl} corresponds to the formula \eqref{FiTEiFu}, where \(\hat{x}(\mu)\) from \eqref{FTEiFu}
is taken for \(y(\mu)\) and \(\check{y}(t)=x(t)\). This means that the operators \(U,\,U^{-1}\)
are mutually inverse.
 \end{proof}

 \noindent%
\(4^{\circ}.\)
 The operator \(U\), which is constructed from eigenfunctions of the operator \(\mathcal{L}\),
 diagonalizes the operator \(\mathcal{L}\). More precisely, this means that the operator
 \(U\mathcal{L}U^{-1}\) is a `diagonal' operator in the space \(\mathscr{K}\).

 To explain this, let us introduce the multiplication operator \(\mathcal{M}\) acting in the space \(\mathscr{K}\).
 \begin{definition} {\ }%
 \label{DeMOLdD}
 \begin{enumerate}
 \item[\textup{1.}] The domain of definition \(\mathcal{D}_{\mathcal{M}}\) of the operator
 \(\mathcal{M}\) is
 \begin{equation}%
 \label{DeMOLd}
 \mathcal{D}_{\mathcal{M}}=\big\{y\in\mathscr{K}:\,\lambda(\mu)y(\mu)\in\mathscr{K}\big\}\,,
\end{equation}%
where \(\lambda(\mu)\) is defined in \eqref{laotmu}\,.
 \item[\textup{2.}] For \(y\in\mathcal{D}_{\mathcal{M}}\),
 \begin{equation}
 \label{AcOpM}
 (\mathcal{M}y)(\mu)=\lambda(\mu)y(\mu),\ \ \mu\in\mathbb{R}^{+}\,.
 \end{equation}
 \end{enumerate}
 \end{definition}
 \begin{theorem}
 \label{DiMaOpT}
 The equality
 \begin{equation}
 \label{DiMaOp}
 \mathcal{L}=U^{-1}\mathcal{M}U
 \end{equation}
 holds. In particular,
 \[U\mathcal{D}_{\mathcal{L}}=\mathcal{D}_{\mathcal{M}}.\]
 \end{theorem}
 Theorem \ref{DiMaOpT} is a paraphrase of a standard result from the theory of Fourier integral.
 This result tell how to express the Fourier transform of the derivative of some function in terms of
 the Fourier transform of the function itself. We do not present the detail explanation.
 This theorem plays an heuristic role only. The only what we need is the expression
 \eqref{EFSLni} for generalized eigenfunction of the operator \(\mathcal{L}\)
 corresponding to the point \(\mu=\mu(\lambda)\) of the spectrum \(\mathcal{L}\).

\noindent
\begin{center}
\fbox{
\begin{minipage}[t]{0.9\linewidth}
 Theorem \ref{DiMaOpT} says that the differential operator \(\mathcal{L}\) is unitary equivalent
 to the multiplication operator \(\mathcal{M}\). The operator \(\mathcal{M}\) may be considered
 as a \emph{model} of the operator \(\mathcal{L}\). Spectral properties of the operator \(\mathcal{L}\)
 can be reformulated in  terms of spectral properties of the model operator \(\mathcal{M}\).
 From the other hand, since the model operator is diagonal, to study spectral properties of the model operator
 is easier than to study  spectral properties of the original operator \(\mathcal{L}\).
\end{minipage}}
\end{center}

\section{\!Matrix-functional calculus
for the operator~\mathversion{bold}%
\(\mathcal{L}\)\,.
}\mathversion{normal}

In Theorem \ref{DiMaOpT} it was stated that the differential operator \(\mathcal{L}\) is unitary equivalent to the multiplication operator  \(\mathcal{M}\) in the space \(\mathscr{K}\) of vector-functions. This allows us to built
functional calculus for the operator~\(\mathcal{L}\).

If
\(g(\mu)\), \(f:\,\mathbb{R}^{+}\to\mathbb{C}\), is a bounded
 complex-values measurable function, we define the operator \(\mathcal{M}_{g}\), which acts in the model space \(\mathscr{K}\), \(\mathcal{M}_{g}:\,\mathscr{K}\to\mathscr{K}\), as
\begin{equation}
\label{scfop}
(\mathcal{M}_{_g}y)(\mu)\stackrel{\textup{\tiny def}}{=}g(\mu)y(\mu),\ \ y\in\mathscr{K}\,.
\end{equation}
In this notation, the equality \eqref{DiMaOp} can be presented in the form
\begin{equation}
\label{AnF}
\mathcal{L}=U^{-1}\mathcal{M}_{\lambda}U,
\end{equation}
where \(U,\,U^{-1}\) are the unitary operators defined in \eqref{unit}, \eqref{uniti},
and the function \(\lambda(\mu)\)  is defined in \eqref{laotmu}.

This motivates the following definition.

\emph{If \(h(\lambda)\) is a  bounded measurable function defined on the spectrum
\(\text{\large\(\sigma\)}(\!\mathcal{L})\) of
 \(\mathcal{L}\), \eqref{specL}, then the operator \(h(\mathcal{L})\),
\(h(\mathcal{L}):\,L^2(\mathbb{R}^{+})\to~\,L^2(\mathbb{R}^{+})\),
is defined as}
\begin{equation}
\label{hotL}
h(\mathcal{L})\stackrel{\textup{\tiny def}}{=}U^{-1}\mathcal{M}_{_g}U.
\end{equation}
where
\begin{equation}
g(\mu)=h(\lambda(\mu)),\ \
\end{equation}
\(\lambda(\mu)\) is defined in \eqref{laotmu}.

The operator \(\mathcal{M}_{_g}\) acts in the space \(\mathscr{K}\)
of two-component \emph{vector} functions and multiplies a vector-function \(y(\mu)\)
on the \emph{scalar} function \(g(\mu)\). However, it is natural to consider
an operator which multiplies a vector-function \(y(\mu)\) on  a \(2\times2\) \emph{matrix}-%
valued function \(G(\mu)\).
\begin{definition} {\ }\\[-2.0ex]
\label{Linf}
\begin{enumerate}
\item[\textup{1.}]
\(L^{\infty}_{_{\mathfrak{M}_2}}(\mathbb{R}^{+})\) is the set of all
\(2\times2\) matrix functions which
entries are complex valued functions defined almost everywhere on \(\mathbb{R}^{+}\)
and essentially bounded there.
\item[\textup{2.}] The set \(L^{\infty}_{_{\mathfrak{M}_2}}(\mathbb{R}^{+})\) is provided by
pointwise algebraic operation: the addition, the multiplication and the multiplication with complex scalars.
\item[\textup{3.}] For
\(G(\mu)\in{}L^{\infty}_{_{\mathfrak{M}_2}}(\mathbb{R}^{+})\), the
norm \(\|G\|_{_{L^{\infty}_{_{\mathfrak{M}_2}}}}\) is defined as
\begin{equation}
\label{dMaS}
\|G\|_{_{L^{\infty}_{_{\mathfrak{M}_2}}}}\!\!\stackrel{\textup{\tiny def}}{=}\,
\underset{\mu\in\mathbb{R}^{+}}{\textup{ess\,sup}}\,\|G(\mu)\|\,,
\end{equation}
where for each \(\mu\), the expression \(\|G(\mu)\|\) means the norm of the
matrix \(F(\mu)\) considered as an operator in the two-dimensional complex Euclidean space,
and \(\textup{ess\,sup}\) is the essential supremum with respect to the Lebesgue measure
on \(\mathbb{R}^{+}\).
\item[\textup{4.}] For
\(G(\mu)\in{}L^{\infty}_{_{\mathfrak{M}_2}}(\mathbb{R}^{+})\), the conjugate
matrix function \(G^{\ast}(\mu)\) is defined as
\begin{equation}
G^{\ast}(\mu)\stackrel{\textup{\tiny def}}{=}(G(\mu))^{\ast}\,,\quad \mu\in\mathbb{R}^{+},
\end{equation}
where where for each \(\mu\), the expression \(G(\mu))^{\ast}\) means the matrix
Hermitian conjugated to the matrix \(G(\mu)\).
\end{enumerate}
\end{definition}

The set \(L^{\infty}_{_{\mathfrak{M}_2}}(\mathbb{R}^{+})\)
provided by pointwise algebraic operations and the norm \eqref{dMaS} is a Banach algebra.
\begin{definition}\ \
\label{FuCa}
Let \(G(\mu)\) be a matrix function from \(L^{\infty}_{_{\mathfrak{M}_2}}(\mathbb{R}^{+})\).

The operator \(\mathcal{M}_{_{{\scriptstyle G}}}\), which acts in the model space \(\mathscr{K}\), \,
 \(\mathcal{M}_{_{\scriptstyle G}}:\,\mathscr{K}\to~\mathscr{K}\), is defined as
\begin{equation}
\label{scfopM}
(\mathcal{M}_{_{\scriptstyle G}}y)(\mu)\stackrel{\textup{\tiny def}}{=}G(\mu)y(\mu),\ \ y\in\mathscr{K}\,.
\end{equation}
\end{definition}
\begin{definition}
Let \(H(\lambda)\),
\begin{equation}
\label{MaS}
H(\lambda)=
\begin{bmatrix}
h_{++}(\lambda)&h_{+-}(\lambda)\\[1.0ex]
h_{-+}(\lambda)&h_{--}(\lambda)
\end{bmatrix},\quad \ \lambda\in\text{\large\(\sigma\)}(\!\mathcal{L})\,,
\end{equation}
be a measurable \(2\times2\) matrix function with complex-valued
entries \(h_{\pm}(\lambda)\)
defined almost everywhere on the spectrum \(\text{\large\(\sigma\)}(\!\mathcal{L})\)
of the operator~\(\mathcal{L}\). Assume that the matrix function \(H\) is essentially bounded on
\(\text{\large\(\sigma\)}(\!\mathcal{L})\):
\begin{equation}
\label{EBa}
\underset{\lambda\in{\textstyle\sigma}(\!\mathcal{L})}{\textup{ess\,sup}}\,\|H(\lambda)\|<\infty\,.
\end{equation}
The operator \(H(\mathcal{L})\),
\(H(\mathcal{L}):\,L^2(\mathbb{R}^{+})\to~\,L^2(\mathbb{R}^{+})\),
is defined as
\begin{equation}
\label{HotL}
H(\mathcal{L})\stackrel{\textup{\tiny def}}{=}U^{-1}\mathcal{M}_{_F}U,
\end{equation}
where \(U,\,U^{-1}\) are the unitary operators defined in \eqref{unit}, \eqref{uniti}, and
\begin{equation}
F(\mu)=H(\lambda(\mu)),\ \ \lambda(\mu) \textup{ is defined in \eqref{laotmu}.}
\end{equation}
\end{definition}
\begin{lemma}
\label{Homom}
 The mapping \(G\to\mathcal{M}_{_{\scriptstyle G}}\) ,
is a norm preserving homomorphism of the algebra \(L^{\infty}_{_{\mathfrak{M}_2}}(\mathbb{R}^{+})\)
of matrix functions into the algebra of all bounded operators in the model space
\(\mathscr{K}\):
\begin{enumerate}
\item[\textup{1}.] If \(G(\mu)\equiv{}I\), then \(\mathcal{M}_{_{\scriptstyle G}}=
\mathscr{I}_{_\mathscr{K}}\), where \(I\) it the identity \(2\times2\) matrix, and
\(\mathscr{I}_{_\mathscr{K}}\) is the identity operator in \(\mathscr{K}\).
\item[\textup{2}.] If \(G(\mu)=\alpha_1G_1(\mu)+\alpha_2G_2(\mu)\), where
\(G_1,\,G_2\in{}L^{\infty}_{_{\mathfrak{M}_2}}(\mathbb{R}^{+})\) %
 and
\(\alpha_1,\,\alpha_2\in\mathbb{C}\), then %
\(\mathcal{M}_{_{\scriptstyle G}}=
\alpha_1\mathcal{M}_{_{{\scriptstyle G}_1}}+\alpha_2\mathcal{M}_{_{{\scriptstyle G}_2}}\).
\item[\textup{3}.] If \(G(\mu)=G_1(\mu)\cdot{}G_2(\mu)\), where \(G_1,\,G_2%
\in{}L^{\infty}_{_{\mathfrak{M}_2}}(\mathbb{R}^{+})\),
 then \\ %
\(\mathcal{M}_{_{\scriptstyle G}}=\mathcal{M}_{_{{\scriptstyle G}_1}}\cdot\mathcal{M}_{_{{\scriptstyle G}_2}}\).
\item[\textup{4}.] If \(G%
\in{}L^{\infty}_{_{\mathfrak{M}_2}}(\mathbb{R}^{+})\), then
\begin{equation}
\label{CoMO}
(\mathcal{M}_{_{\scriptstyle G}})^{\ast}=\mathcal{M}_{_{{\scriptstyle G}^{\ast}}},
\end{equation}
where \(G^{\ast}\) is the matrix function conjugated to the matrix-function \(G\)
and \(\mathcal{M}_{_{{\scriptstyle G}^{\ast}}}\) is the operator conjugated to the operator
\(\mathcal{M}_{_{\scriptstyle G}}\) with respect to the scalar product in \(\mathscr{K}\).
\item[\textup{5}.] If \(G%
\in{}L^{\infty}_{_{\mathfrak{M}_2}}(\mathbb{R}^{+})\),
then
\begin{equation}
\label{NoP}
\|\mathcal{M}_{_{\scriptstyle G}}\|_{_{\scriptstyle{\mathscr{K}\to\mathscr{K}}}}=
\|G\|_{_{{\scriptstyle{}L}^{\infty}_{_{\mathfrak{M}_2}}}}.
\end{equation}
\end{enumerate}
\end{lemma}
\begin{lemma}
\label{StrConL}
Let \(\{G_n\}_{1\leq{}n<\infty}\) be a sequence of
\(2\times2\) matrix functions,
\(G_n\in{}L^{\infty}_{_{\mathfrak{M}_2}}(\mathbb{R}^{+})\) for every \(n\).

We assume that
\begin{enumerate}
\item[\textup{1}.] The sequence \(\{G_n\}_{1\leq{}n<\infty}\) is uniformly bounded, that is
\begin{equation}
\label{UniBou}
\sup_n\|G_n\|_{_{{\scriptstyle{}L}^{\infty}_{_{\mathfrak{M}_2}}}}<\infty\,.
\end{equation}
\item[\textup{2}.] For \textsf{almost every} \(\mu\in\mathbb{R}^{+}\) there exists the limit of matrices \(G_n(\mu)\)\textup{:}
\begin{equation}
\label{PWLim}
\lim_{n\to\infty}G_n(\mu)=G(\mu)\,.
\end{equation}
\end{enumerate}
Then the sequence of operators \(\{\mathcal{M}_{_{\,{\scriptstyle G}_n}}\}_{1\leq n<\infty}\)
converges \textsf{strongly} to the operator \(\mathcal{M}_{_{\scriptstyle G}}\)\textup{:}
\begin{equation}
\label{StrConM}
 \lim_{n\to\infty}\|\mathcal{M}_{_{\,{\scriptstyle G}_n}}y-\mathcal{M}_{_{\scriptstyle G}}y\|_{_{\scriptstyle \mathscr{K}}}=0\,
\ \ \  \textup{for every} \ \ y\in\mathscr{K},
\end{equation}
\end{lemma}
\begin{proof}
According to \eqref{UniBou}, there exists a constant \(C<\infty\) such that the norms of the matrices
\(G_n(\mu),\,G(\mu)\)
admit the estimates
\(\|G_n(\mu)\|\leq{}C\),
\(\|G(\mu)\|\leq{}C\) for almost every \(\mu\in\mathbb{R}^{+}.\)
 Therefore for every \(y\in\mathscr{K}\),
the inequalities
\begin{equation}
\label{DomCon}
\|G_n(\mu)y(\mu)-G(\mu)y(\mu)\|^{2}_{_{{\scriptstyle\mathbb{C}}^2}}\leq4C^2\|y(\mu)\|^{2}_{_{{\scriptstyle\mathbb{C}}^2}}
\end{equation}
hold for almost every \(\mu\in\mathbb{R}^{+}\) and for every \(n=1,\,2,\,3,...\,\,.\)
From the condition \eqref{PWLim} it follows that
\begin{equation}
\label{CToZ}
\lim_{n\to\infty}\|G_n(\mu)y(\mu)-G(\mu)y(\mu)\|^{2}_{_{{\scriptstyle\mathbb{C}}^2}}=0  \ \
\textup{for almost every} \ \ \mu\in\mathbb{R}^{+}\,.
\end{equation}
Since \(y\in\mathscr{K}\),
\begin{equation}
\label{SumMaj}
\int\limits_{\mu\in\mathbb{R}^{+}}\|y(\mu)\|^{2}_{_{{\scriptstyle\mathbb{C}}^2}}\frac{d\mu}{2\pi}=
\|y\|^{2}_{_{_{\scriptstyle \mathscr{K}}}}<\infty\,.
\end{equation}
From \eqref{DomCon}, \eqref{CToZ}, \eqref{SumMaj} and the Lebesgue dominated convergence theorem
it follows that
\begin{equation*}
\label{TeToZe}
\lim_{n\to\infty}\int\limits_{\mathbb{R}^{+}}
\|G_n(\mu)y(\mu)-G(\mu)y(\mu)\|^{2}_{_{{\scriptstyle\mathbb{C}}^2}}\frac{d\mu}{2\pi}=0\,.
\end{equation*}
The last equality is the equality \eqref{StrCon}.
\end{proof}
\begin{definition}
\label{InvMa}
Let \(G(\mu),\,\mu\in\mathbb{R}^{+},\) be a \(2\times2\) matrix function defined almost everywhere
on \(\mathbb{R}^{+}\). Then \emph{the matrix-function \(G^{-1}(\mu)\) is defined for those
\(\mu\in\mathbb{R}^{+}\) for which the matrix \(G(\mu)\) is defined and invertible:}
\[G^{-1}(\mu)\stackrel{\textup{\tiny def}}{=}(G(\mu))^{-1}\,,\]
 where \((G(\mu))^{-1}\) is the matrix inverse to the matrix \(G(\mu)\).

In particular, if the matrix \(G(\mu)\) is invertible for almost every \mbox{\(\mu\in\mathbb{R}^{+}\)},
then the matrix function \(G^{-1}(\mu)\) is defined almost everywhere.
\end{definition}

\begin{theorem}
\label{InvCon}
Let \(G(\mu),\,\mu\in\mathbb{R}^+\),
be a matrix function from the set \(L^{\infty}_{_{\mathfrak{M}_2}}(\mathbb{R}^{+})\),
\textup{(Definition \ref{Linf})},
and \(\mathcal{M}_{_{\scriptstyle G}}: \mathscr{K}\to{}\mathscr{K}\) be the operator generated by
the matrix function \(G\).

\begin{enumerate}
\item[\textup{1}.]
In order to the operator \(\mathcal{M}_{_{\scriptstyle G}}\) be invertible it is necessary and sufficient that
both of the following two conditions are satisfied:
\begin{enumerate}
\item[\textup{a)}.]
For almost every \(\mu\in\mathbb{R}^{+}\), the matrix \(G(\mu)\) is invertible,
i.e. the condition
\begin{equation}
\label{DIC}
\det G(\mu)\not=0\quad \textup{for almost every }\ \mu\in\mathbb{R}^{+}
\end{equation}
holds.
\item[\textup{b)}.] The matrix function \(G^{-1}(\mu)\), which under the condition
\eqref{DIC}
is defined for almost every \(\mu\in\mathbb{R}^{+}\), belongs to the set
\(L^{\infty}_{_{\mathfrak{M}_2}}(\mathbb{R}^{+})\)\textup{:}
\begin{equation}
\label{BoIM}
\underset{\mu\in\mathbb{R}^{+}}{\textup{ess\,sup}}\,\|G^{-1}(\mu)\|<\infty\,.
\end{equation}
\end{enumerate}
\item[\textup{2}.]
If the conditions \eqref{DIC} and \eqref{BoIM} are satisfied, then the inverse operator
\((\mathcal{M}_{G})^{-1}\) is expressible as
\begin{equation}
\label{ExInOp}
(\mathcal{M}_{G})^{-1}=\mathcal{M}_{G^{-1}}\,.
\end{equation}
\item[\textup{3}.] If the condition \eqref{DIC} is violated, then the point \(\zeta=0\) belongs
to both the point spectrum \(\textup{\large\(\sigma\)}_p((\mathcal{M}_{F})\)
and the residual spectrum \(\textup{\large\(\sigma\)}_r((\mathcal{M}_{G})\) of the operator \(\mathcal{M}_{G}\).
\item[\textup{4}.] If the condition \eqref{DIC} is satisfied, but  the condition \eqref{BoIM} is
violated, then the  point \(\zeta=0\) belongs to the continuous spectrum
\(\textup{\large\(\sigma\)}_c((\mathcal{M}_{G})\)
 of the operator \(\mathcal{M}_{F}\),
but \(0\not\in\textup{\large\(\sigma\)}_p((\mathcal{M}_{G})\),
\(0\not\in\textup{\large\(\sigma\)}_r((\mathcal{M}_{G})\).
\end{enumerate}
\end{theorem}
\begin{proof} \ \\
\hspace*{2.0ex}\(\circ\)\,\,Assume that the conditions \eqref{DIC} and \eqref{BoIM} are satisfied,
so the operator \(\mathcal{M}_{G^{-1}}: \,L^{2}(\mathbb{R}^{+})\to{}L^{2}(\mathbb{R}^{+})\) exists
and is a bounded operator. According to the statements \textsf{3} and \textsf{1} of Lemma \ref{Homom},
\begin{equation*}
\mathcal{M}_{G^{-1}}\cdot{}\mathcal{M}_{G}=\mathcal{M}_{G}\cdot{}\mathcal{M}_{G^{-1}}=
\mathcal{M}_{G^{-1}\,\cdot\,G}=\mathcal{M}_{G\,\cdot\,G^{-1}}=\mathcal{M}_{I}=\mathscr{I}_{_\mathscr{K}}\,.
\end{equation*}
So the operator \(\mathcal{M}_{G}\) is invertible, and the inverse operator \((\mathcal{M}_{G})^{-1}\)
is expressed by the equality \eqref{ExInOp}.\\
\hspace*{2.0ex}\(\circ\)\,\,Assume that the condition \eqref{DIC} is violated.
Then there exists a measurable set \(S\), \(S\in\mathbb{R}^{+},\,m(S)>0\), such that for
every \(\mu\in{}S\) there exists a vector \(y(\mu)\in\mathbb{C}^{\,2},\,y(\mu)\not=0\), such that
\(G(\mu)y(\mu)=0\). One can choose the vectors \(y(\mu)\) in such a way that the function
\(y(\mu),\,\mu\in{}S\), is measurable. One can normalize the vectors \(y(\mu)\) in such a way
that \(\int\limits_{S}y^{\ast}(\mu)\,y(\mu)\,\frac{d\mu}{2\pi}=1.\) Until now the function \(y(\mu)\) is defined
only on \(S\). Let us extend the function \(y(\mu)\) on the whole \(\mathbb{R}^{+}\) putting
\(y(\mu)=0\) for \(\mu\in\mathbb{R}^{+}\setminus{}S\). Now the function \(y(\mu)\) is defined
everywhere on \(\mathbb{R}^{+}\) and satisfies the condition
\(\int\limits_{\mathbb{R}^{+}}y^{\ast}(\mu)\,y(\mu)\,\frac{d\mu}{2\pi}=1\), that is
\begin{equation*}
y(\mu)\in\mathscr{K},\quad\,y\not=0\ \ \textup{in}\ \  \mathscr{K}\,.
\end{equation*}
From the other hand,
\begin{equation*}
G(\mu)y(\mu)=0\quad \textup{for every}\ \ \mu\in\mathbb{R}^{+}\,.
\end{equation*}
This means that the vector \(y,\,y\not=0\), belongs to the null space of the operator \(\mathcal{M}_{G}\).
In other words, the point \(\zeta=0\) belongs to the point spectrum of the operator \(\mathcal{M}_{G}\).
 Considering the matrix function
\(G^{\ast}(\mu)\) we obtain that the point \(\zeta=0\) belongs to the point spectrum of the operator \((\mathcal{M}_{G})^{\ast}\). In other words, the point \(\zeta=0\) belongs to the residual spectrum of the operator \(\mathcal{M}_{G}\). In particular the operator \(\mathcal{M}_{G}\) is not invertible.\\
\hspace*{2.0ex}\(\circ\)\,\,Assume that the condition \eqref{DIC} is satisfied. Then
the point \(\zeta=0\) belongs neither to the point spectrum, nor to the residual spectrum of the
operator \(\mathcal{M}_{G}\). Indeed let \(y\in\mathscr{K},\,y\not=0\), but
\(\mathcal{M}_{_{\scriptstyle G}}\,y=0\). Let \(S=\{\mu\in\mathbb{R}^{+}:\,y(\mu)\not=0\}\). Since \(y\not=0\) in
\(\mathscr{K}\), \(m(S)>0\). From the other hand, since
\(\mathcal{M}_{_{\scriptstyle G}}\,y=0\) in
\(\mathscr{K}\),
 \(G(\mu)y(\mu)=0\) for almost every
\(\mu\in\mathbb{R}^{+}\). Therefore \(\det G(\mu)=0\) for almost every \(\mu\in{}S\).
This contradicts the condition \eqref{DIC}. Thus, \(0\not\in\textup{\large\(\sigma\)}_{\!\!p}(\mathcal{M}_{_{\scriptstyle G}})\).
Analogously \(0\not\in\textup{\large\(\sigma\)}_{\!\!p}(\mathcal{M}_{_{\scriptstyle G^{\ast}}})=
\overline{\textup{\large\(\sigma\)}_{\!r}(\mathcal{M}_{_{\scriptstyle G}})}\).\\
\hspace*{2.0ex}\(\circ\)\,\,Let the condition \eqref{DIC} be satisfied.  Then the matrix function \((G(\mu))^{-1}\) is defined
for almost all \(\mu\in\mathbb{R}^{+}\). Given \(\varepsilon>0\), let
\(S_{\varepsilon}=\{\mu\in\mathbb{R}^{+}:\,\|(G(\mu))^{-1}\|\geq\varepsilon^{-1}\}\).
Assume moreover that
the condition \eqref{BoIM} is violated. Then for every \(\varepsilon>0\), the condition
\(m(S_{\varepsilon})>{}0\) holds. If \(\|(G(\mu))^{-1}\|\geq\varepsilon^{-1}\) for some \(\mu\),
then there exists a vector \(v(\mu)\in\mathbb{C}^{2}\), \(v(\mu)\not=0\) such that
\(\|(G(\mu))^{-1}v(\mu)\|\geq\varepsilon^{-1}\|v(\mu)\|\). The vector \(y(\mu)=(G(\mu))^{-1}v(\mu)\)
satisfies the conditions \(y(\mu)\not=0 \),
\begin{math}
\|G(\mu)y(\mu)\|\leq\varepsilon\|y(\mu)\|\,.
\end{math}
Thus if the condition \eqref{DIC} is satisfied, but the condition \eqref{BoIM} is violated,
then for every fixed \(\varepsilon>0\) there exists the vector function \(y(\mu)\) defined
 on the set \(S_{\varepsilon}\) which satisfies the conditions
 \begin{equation*}
 y(\mu)\not=0, \ \|G(\mu)y(\mu)\|\leq\varepsilon\|y(\mu)\|\ \textup{for every}\ \mu\in{S_{\varepsilon}}\,,
 \ \textup{where}\  m(S_{\varepsilon})>0\,.
 \end{equation*}
 The function \(y(\mu)\), defined up to now only on \(S_{\varepsilon}\), can be chosen to be measurable.
 Normalizing the vectors \(y(\mu)\) appropriately, we can satisfy the condition
 \begin{math}
 \int\limits_{S_{\varepsilon}}y^{\ast}(\mu)y(\mu)\,\frac{d\mu}{2\pi}=1.
 \end{math}
 Let us extend the function \(y(\mu)\) from the set \(S_{\varepsilon}\) to the whole \(\mathbb{R}^{+}\)
 putting \(y(\mu)=0\) for \(\mu\in\mathbb{R}^{+}\setminus{}S_{\varepsilon}\). The extended function \(y(\mu)\)
 satisfies the condition
 \begin{math}
 \int\limits_{\mathbb{R}^{+}}y^{\ast}(\mu)y(\mu)\,\frac{d\mu}{2\pi}=1.
 \end{math}
 In other words,
 \begin{equation}
 \label{NoY}
 y\in\mathscr{K},\quad \|y\|_{_{\scriptstyle\mathscr{K}}}=1\,.
 \end{equation}
 On the other hand, the inequality \( \|G(\mu)y(\mu)\|\leq\varepsilon\|y(\mu)\|\), which holds for the extended function
 \(y(\mu)\) at every \(\mu\in\mathbb{R}^{+}\), implies the inequality
 \begin{equation}
 \label{NoFY}
 \|\mathcal{M}_{_{\scriptstyle{\!G}}}\,y\|_{_{\scriptstyle\mathscr{K}}}\leq\varepsilon\,.
 \end{equation}
 Thus if the condition  \eqref{DIC} is satisfied, but the condition \eqref{BoIM} is violated,
 then for every \(\varepsilon>0\) there exists \(y\in\mathscr{K}\) satisfying the conditions
 \eqref{NoY} and \eqref{NoFY}. Therefore the point \(\zeta=0\) belongs to the continuous spectrum
 \(\textup{\large\(\sigma\)}_{\!\!c}(\mathcal{M}_{_{\scriptstyle G}})\) of the operator
 \(\mathcal{M}_{G}\). (We already know that if the condition \eqref{DIC} is satisfied, then
 \(0\not\in\textup{\large\(\sigma\)}_{\!\!p}(\mathcal{M}_{_{\scriptstyle G}}),\,
 0\not\in\textup{\large\(\sigma\)}_{\!\!r}(\mathcal{M}_{_{\scriptstyle G}})\)).
 In particular, the operator  \(\mathcal{M}_{G}\) is not invertible.
 \end{proof}

\section{The functional model
of the operator
\mathversion{bold}%
\(\mathscr{F}_{\!_{\scriptstyle\boldsymbol{\mathbb{R}^{+}}}}\).
}\mathversion{normal}
In this section we construct the functional model of the
 truncated Fourier operator \(\mathscr{F}_{\!_{\scriptstyle{\mathbb{R}}^{+}}}\).
First we do formal calculations. Then we justify them.

The operator \(\mathcal{L}\) commutes with the operators
\(\mathscr{F}_{\!_{\scriptstyle{\mathbb{R}}^{+}}},\,
\mathscr{F}_{\!_{\scriptstyle{\mathbb{R}}^{+}}}^{\ast}\). (Theorem \ref{Comhi}). Let \(\mu\in\mathbb{R}^{+}\).
The "eigenspace" of the operator \(\mathcal{L}\) corresponding to the eigenvalue \(\lambda(\mu)\)
is two-dimensional and is generated by the "eigenfunctions" \eqref{EFSLni}.
Would be
the "eigenfunctions" \eqref{EFSLni} of the operator
\(\mathcal{L}\) "true" \(L^2(\mathbb{R}^{+})\)-functions, then the
two-dimensional subspace generated by them will be invariant with
respect to each of the operators \(\mathscr{F}_{\!_{\scriptstyle{\mathbb{R}}^{+}}}\) and \(\mathscr{F}_{\!_{\scriptstyle{\mathbb{R}}^{+}}}^{\ast}\). This means that for
some  matrix
\begin{equation}
\label{ModTrFo}
F(\mu)=\begin{bmatrix}f_{++}(\mu)&f_{+-}(\mu)\\
f_{-+}(\mu)&f_{--}(\mu)\end{bmatrix},
\end{equation}
which is constants with
respect to \(t\),
the equality holds
\begin{gather*}
 \big(\mathscr{F}_{\!_{\scriptstyle{\mathbb{R}}^{+}}}\psi_{+}(\,.\,,\mu)\big)(t)=
\psi_{+}(t,\mu)f_{++}(\mu)+\psi_{-}(t,\mu)f_{-+}(\mu),\\
 \big(\mathscr{F}_{\!_{\scriptstyle{\mathbb{R}}^{+}}}\psi_{-}(\,.\,,\mu)\big)(t)
=\psi_{+}(t,\mu)f_{+-}(\mu)+\psi_{-}(t,\mu)f_{--}(\mu)\,.
\end{gather*}
The matrix form of these equalities is:
\begin{subequations}
\label{FTBF}
\begin{equation}
\label{FTBF1}
\big(\mathscr{F}_{\!_{\scriptstyle{\mathbb{R}}^{+}}}\psi(\,.\,,\mu)\big)(t)=\psi(t\,,\mu)F(\mu)\,.
\end{equation}
We show that
\begin{equation}
\label{FTBF2}
\big(\mathscr{F}_{\!_{\scriptstyle{\mathbb{R}}^{+}}}%
^{\ast}\psi(\,.\,,\mu)\big)(t)=\psi(t\,,\mu)F^{\ast}(\mu)\,,
\end{equation}
where \(F^{\ast}(\mu)\) is the matrix Hermitian conjugated to the matrix \(F(\mu)\).
\end{subequations}

 However the functions \(\psi_{\pm}(t,\mu)\) does not belong to
\(L^2(\mathbb{R}^{+})\). So the operators \(\mathscr{F}_{\!_{\scriptstyle{\mathbb{R}}^{+}}},\,
\mathscr{F}_{\!_{\scriptstyle{\mathbb{R}}^{+}}}^{\ast}\),
considered as an operator acting in \(L^2(\mathbb{R}^{+})\), are not applicable
to the functions \(\psi_{+}(t,\mu),\psi_{-}(t,\mu)\). Nevertheless
we can consider the Fourier integrals
\(\mathscr{F}_{\!_{\scriptstyle{\mathbb{R}}^{+}}}\psi_{\pm}(\,.\,,\mu)\), \(\mathscr{F}^{\ast}_{\!_{\scriptstyle{\mathbb{R}}^{+}}}\psi_{\pm}(\,.\,,\mu)\)
 in some \emph{Pickwick sense}.
Namely we  interpret the expressions
\((\mathscr{F}_{\!_{\scriptstyle{\mathbb{R}}^{+}}}\psi_{\pm}(\,.\,,\mu))(t)\) and
\((\mathscr{F}_{\!_{\scriptstyle{\mathbb{R}}^{+}}}^{\ast}\psi_{\pm}(\,.\,,\mu))(t)\)
as
\begin{subequations}
\label{PickFT}
 \begin{align}
 \label{PickFT1}
 \big(\mathscr{F}_{\!_{\scriptstyle{\mathbb{R}}^{+}}}%
 \psi_{\pm}(\,.\,,\mu)\big)(t)&=\lim_{\substack{\varepsilon\to+0\\
 R\to+\infty}}\frac{1}{\sqrt{2\pi}}%
\int\limits_{\varepsilon}^{R}\xi^{-1/2\pm{}i\mu}e^{i\xi{}t}\,d\xi\,,
\end{align}
\begin{align}
 \label{PickFT2}
 \big(\mathscr{F}^{\ast}_{\!_{\scriptstyle{\mathbb{R}}^{+}}}%
 \psi_{\pm}(\,.\,,\mu)\big)(t)&=\lim_{\substack{\varepsilon\to+0\\
 R\to+\infty}}\frac{1}{\sqrt{2\pi}}%
\int\limits_{\varepsilon}^{R}\xi^{-1/2\pm{}i\mu}e^{-i\xi{}t}\,d\xi\,,
\end{align}
\end{subequations}
 In \eqref{PickFT},
\(t\in\mathbb{R}^{+},\,\mu\in\mathbb{R}^{+}\). It turns out that the
limits in \eqref{PickFT} exist and are uniform if \(t\) belongs to
any fixed interval separated from zero and infinity. (We shall see
this when calculating the integrals.) Changing variable in
\eqref{PickFT1}: \(\xi\to\xi/t\), and using the homogeneity
properties of the functions \(\psi_{\pm}(t,\mu)\) with respect to
\(t\), we obtain that
\begin{subequations}
\label{FTEiS}
\begin{align}
 \label{FTEiS1}
 \big(\mathscr{F}_{\!_{\scriptstyle{\mathbb{R}}^{+}}}%
 \psi_{+}(\,.\,,\mu)\big)(t)&=\psi_{-}(t,\mu)f_{-+}(\mu)\,,\\
 \label{FTEiS2}
  \big(\mathscr{F}_{\!_{\scriptstyle{\mathbb{R}}^{+}}}%
  \psi_{-}(\,.\,,\mu)\big)(t)&=\psi_{+ }(t,\mu)f_{+-}(\mu),
 \end{align}
 \end{subequations}
 where \(t\in\mathbb{R}^{+},\ \mu\in\mathbb{R}^{+}\), and
\begin{subequations}
\label{FTEco}
\begin{align}
\label{FTEco1}
f_{-+}(\mu)=\lim_{\substack{\varepsilon\to+0\\
R\to+\infty}}\frac{1}{\sqrt{2\pi}}\int\limits_{\varepsilon}^{R}\xi^{-1/2+i\mu}e^{i\xi{}}\,d\xi\,,
\\
\label{FTEco2}
f_{+-}(\mu)=\lim_{\substack{\varepsilon\to+0\\
R\to+\infty}}\frac{1}{\sqrt{2\pi}}\int\limits_{\varepsilon}^{R}\xi^{-1/2-i\mu}e^{i\xi{}}\,d\xi\,.
 \end{align}
 \end{subequations}
 Changing variable in
\eqref{PickFT2}: \(\xi\to\xi/t\), and using the homogeneity
properties of the functions \(\psi_{\pm}(t,\mu)\) with respect to
\(t\), we obtain that
\begin{subequations}
\label{FTEiSA}
\begin{align}
 \label{FTEiS1A}
 \big(\mathscr{F}_{\!_{\scriptstyle{\mathbb{R}}^{+}}}^{\ast}%
 \psi_{+}(\,.\,,\mu)\big)(t)&=\psi_{-}(t,\mu)\overline{f_{+-}(\mu)}\,,\\
 \label{FTEiS2A}
  \big(\mathscr{F}_{\!_{\scriptstyle{\mathbb{R}}^{+}}}^{\ast}%
  \psi_{-}(\,.\,,\mu)\big)(t)&=\psi_{+ }(t,\mu)\overline{f_{-+}(\mu)},
 \end{align}
 \end{subequations}
 where \(t\in\mathbb{R}^{+},\ \mu\in\mathbb{R}^{+}\).
 Let as calculate the integrals in \eqref{FTEco}. These integrals
 can be presented as
 \begin{equation}
 \label{CoInRe}
\lim_{\substack{\varepsilon\to+0\\
R\to+\infty}}\int\limits_{\varepsilon}^{R}\xi^{-1/2\pm{}i\mu}e^{i\xi{}}\,d\xi
=e^{i\frac{\pi}{4}\mp{}\frac{\mu\pi}{2}}\
\lim_{\substack{\varepsilon\to+0\\
 R\to+\infty}}\ \
 \int\limits_{[-i\varepsilon,iR]}f(\zeta)d\zeta,\,
 \end{equation}
 where
 \begin{equation}
 \label{ComVa}
 f(\zeta)=\zeta^{-1/2\pm{}i\mu}e^{-\zeta}\,, \ \ \arg\zeta>0 \ \ \textup{for}
 \ \ \zeta\in(0,\infty)\,.
 \end{equation}
 Then we  `rotate' the ray of integration
 from the ray \((0,-i\infty)\) to the ray
 \((0,\infty)\).
 The function \(f(\zeta)\) is holomorphic in the domain \(\mathbb{C}\setminus(-\infty,0].\)
 According to Cauchy integral theorem,
 \begin{equation*}
\int\limits_{[-i\varepsilon,iR]}f(\zeta)\,d\zeta=\int\limits_{[\varepsilon,R]}f(\zeta)\,d\zeta+
\int\limits_{\gamma_{\varepsilon}}f(\zeta)d\zeta+\int\limits_{\gamma_{_R}}f(\zeta)\,d\zeta,
 \end{equation*}
 where \(\gamma_{\varepsilon}\) and \(\gamma_{_R}\) are the arcs \(-\pi/2\leq\arg{z}\leq{}0\),
 \(|z|=\varepsilon\) and \(|z|\) respectively. The functions \(f(\zeta)\) grows as \(|\varepsilon|^{-1/2}\) as \(\zeta\in\gamma_\varepsilon\), \(\varepsilon\to{}0\),
 and the length of the arc \(\gamma_{\varepsilon}\) decays as \(\varepsilon\), as
 \(\varepsilon\to{}0\). Therefore, \(\int\limits_{\gamma_{\varepsilon}}f(\zeta)d\zeta\to{}0\)
as \(\varepsilon\to{}0\). Applying Jordan lemma to the function \(f(\zeta)\) in the
quadrant \(-\pi/2\leq\arg{\zeta}\leq{}0\), we conclude that
 \(\int\limits_{\gamma_{_R}}f(\zeta)d\zeta\to{}0\)
as \(R\to{}\infty\).
Therefore
\begin{equation*}
\lim_{\substack{\varepsilon\to+0\\
 R\to+\infty}}\int\limits_{\varepsilon}^{R}\xi^{-1/2\pm{}i\mu}e^{i\xi{}}\,d\xi=
 e^{i\frac{\pi}{4}\mp{}\frac{\mu\pi}{2}}%
 \int\limits_{0}^{+\infty}\xi^{-1/2\pm{}i\mu}e^{-\xi{}}\,d\xi\,.
\end{equation*}
The integral in the right hand side of the last formula is the Euler integral
representing the \(\Gamma\)-function. Thus
\begin{equation}
\label{CalGa}
\lim_{\substack{\varepsilon\to+0\\
 R\to+\infty}}\int\limits_{\varepsilon}^{R}\xi^{-1/2\pm{}i\mu}e^{i\xi{}}\,d\xi=
 e^{i\frac{\pi}{4}\mp{}\frac{\mu\pi}{2}}%
 \Gamma(1/2\pm{}i\mu)\,,\quad -\infty<\mu<\infty\,,
\end{equation}
and
\begin{subequations}
\label{MaEnt}
\begin{align}
\label{MaEnt1}
f_{+-}(\mu)=\frac{1}{\sqrt{2\pi}} \,e^{i\frac{\pi}{4}+{}\frac{\mu\pi}{2}}%
 \Gamma(1/2-{}i\mu)\,,\\
 \label{MaEnt2}
f_{-+}(\mu)=\frac{1}{\sqrt{2\pi}} \,e^{i\frac{\pi}{4}-{}\frac{\mu\pi}{2}}%
 \Gamma(1/2+{}i\mu)\,.
\end{align}
\end{subequations}
So, the matrix \(F(\mu)=
\Big[\begin{smallmatrix}f_{++}(\mu)&f_{+-}(\mu)\\  %
f_{-+}(\mu)&f_{--}(\mu)\end{smallmatrix}\Big]\)
in \eqref{FTBF} is of the form
\begin{equation}
\label{Matr}
F(\mu)=%
\begin{bmatrix}
0& \frac{1}{\sqrt{2\pi}}\,e^{i\frac{\pi}{4}+{}\frac{\mu\pi}{2}}\Gamma(1/2-{}i\mu) \\[1.0ex]
\frac{1}{\sqrt{2\pi}}\, e^{i\frac{\pi}{4}-{}\frac{\mu\pi}{2}}\Gamma(1/2+{}i\mu) &0
\end{bmatrix}
\end{equation}
Thus the equalities \eqref{FTBF} hold with the matrix \(F(\mu)\) of the form \eqref{Matr}.
Given \(x(t)\in{}L^2(\mathbb{R}^{+})\), we apply the operators
\(\mathscr{F}_{\!_{\scriptstyle{\mathbb{R}}^{+}}},\,
\mathscr{F}_{\!_{\scriptstyle{\mathbb{R}}^{+}}}^{\ast}\) to the spectral expansion \eqref{PInF1},
\eqref{PInF2}. Applying the operators \(\mathscr{F}_{\!_{\scriptstyle{\mathbb{R}}^{+}}},\,
\mathscr{F}_{\!_{\scriptstyle{\mathbb{R}}^{+}}}^{\ast}\) to the
linear combination \(\psi(t,\mu)\hat{x}(\mu)\), we should take
into account that these operators act on functions
of variable \(t\) and the coefficients \(\hat{x}(\mu)\) of this
linear combination do not depend on \(t\). Therefore
\begin{subequations}
\label{CaTh}
\begin{align}%
\label{CaTh1}
\mathscr{F}_{\!_{\scriptstyle{\mathbb{R}}^{+}}}%
\big(\psi(\,.\,,\mu)\hat{x}(\mu)\big)(t)&=
\big(\mathscr{F}_{\!_{\scriptstyle{\mathbb{R}}^{+}}}\psi(\,.\,,\mu)\big)(t)\hat{x}(\mu)\,,\\
\label{CaTh2}
\mathscr{F}_{\!_{\scriptstyle{\mathbb{R}}^{+}}}^{\ast}%
\big(\hat{x}(\mu)\psi(\,.\,,\mu)\big)(t)&=
\big(\mathscr{F}_{\!_{\scriptstyle{\mathbb{R}}^{+}}}^{\ast}\psi(\,.\,,\mu)\big)(t)\hat{x}(\mu)\,.
\end{align}
\end{subequations}
 Carry the operator \(\mathscr{F}_{\!_{\scriptstyle{\mathbb{R}}^{+}}}\)
 trough the integral in \eqref{PInF2} and using
\eqref{CaTh}, we obtain
\begin{subequations}
\label{SpReF}
\begin{gather}
\label{SpReF1}%
 (\mathscr{F}_{\!_{\scriptstyle{\mathbb{R}}^{+}}}x)(t)=\!\!\!\!\!\!
\int\limits_{\mu\in\mathbb{R}^{+}}\!\!\!\!\!
\psi(t,\mu)\,u_{{}_{\mathscr{F}_{{{\mathbb{R}}^{+}}}}}(\mu)%
\,\frac{d\mu}{2\pi},\,\,\,\,\,\,
(\mathscr{F}_{\!_{\scriptstyle{\mathbb{R}}^{+}}}^{\ast}x)(t)=\!\!\!\!\!\!
\int\limits_{\mu\in\mathbb{R}^{+}}\!\!\!\!\!
\psi(t,\mu)\,u_{{}_{\mathscr{F}_{{{\mathbb{R}}^{+}}}^{\ast}}}(\mu)\,\frac{d\mu}{2\pi},\,
 \intertext{where}
\label{SpReF2}%
 u_{{}_{\mathscr{F}_{{{\mathbb{R}}^{+}}}}}(\mu)=
F(\mu)\,\hat{x}(\mu)\,,\quad
u_{{}_{\mathscr{F}_{{{\mathbb{R}}^{+}}}^{\ast}}}(\mu)=
F^{\ast}(\mu)\,\hat{x}(\mu)\,.
\end{gather}
\end{subequations}


Let us go to prove rigorously the formulas \eqref{SpReF} expressing the
spectral resolution of the vectors
\(\mathscr{F}_{\!_{\scriptstyle{\mathbb{R}}^{+}}}\,x\) in terms
of  the spectral resolution \eqref{PInF1} of the vector \(x\). In
this proof we use the following expressions for the absolute
values of the entries of the matrix \(F(\mu)\):
\begin{equation}
\label{AVMe}%
 |f_{+-}(\mu)|=(1+e^{-2\pi\mu})^{-1/2},\ \ \
|f_{-+}(\mu)|=(1+e^{2\pi\mu})^{-1/2},\quad \mu\in\mathbb{R}^{+}\,.
\end{equation}
The expressions \eqref{AVMe} are derived  from \eqref{MaEnt}.
 Since
\begin{equation}
\label{Refl}
\Gamma(1/2+{}i\mu)\Gamma(1/2-{}i\mu)=\frac{\pi}{\cosh{\pi\mu}}\,
\end{equation}
and the numbers \(\Gamma(1/2\pm{}i\mu)\) are complex conjugated,
then
\begin{equation}
\label{AbGa}
|\Gamma(1/2\pm{}i\mu)|^2=\frac{2\pi}{e^{\pi\mu}+e^{-\pi\mu}},\quad
\mu\in\mathbb{R}^{+}\,.
\end{equation}
The equalities \eqref{AVMe} follows from the last formula and from
\eqref{MaEnt}. We remark that in particular
\begin{equation}
\label{MaFC}
1/\sqrt{2} <|f_{+-}(\mu)|<1,\ \ \ |f_{-+}(\mu)|<1/\sqrt{2},\ \, \mu\in\mathbb{R}^{+}\,.
\end{equation}
If \(\mu\) runs over the interval \([0,\infty)\), then
\(|f_{+-}(\mu)|\) increases from \(2^{-1/2}\) to \(1\) and
\(|f_{-+}(\mu)|\) decreases from \(2^{-1/2}\) to \(0\)\,. In
particular,
\begin{equation}
\label{MaFCSu}
\sup_{\mu\in\mathbb{R}^{+}}|f_{+-}(\mu)|=\underset{\mu\in\mathbb{R}^{+}}%
{\textup{ess\,sup}}|f_{+-}(\mu)|=1\,.
\end{equation}
From \eqref{MaEnt} and \eqref{AbGa} it follows that
\begin{equation}
\label{SMc2}
 |f_{+-}(\mu)|^2+|f_{-+}(\mu)|^2=1\,,
\end{equation}
\begin{equation}
\label{SMc}
|f_{+-}(\mu)|+|f_{-+}(\mu)|=\sqrt{1+\frac{1}{\cosh\,\pi\mu}},
\end{equation}
thus
\begin{equation}
\label{SMe}
1\leq |f_{+-}(\mu)|+|f_{-+}(\mu)|\leq\sqrt{2},\quad 0\leq\mu<\infty\,.
\end{equation}
In view of the diagonal structure \eqref{Matr} of the matrix
\(F(\mu)\) and the estimates \eqref{MaFC}, \eqref{MaFCSu} for its
entries,  the equalities
\begin{subequations}
\label{EsNoF}%
\begin{equation}%
\label{EsNoF1}%
\|F(\mu)\|<1\ \ \ \forall\,\mu\in(0,\infty)
\end{equation}
and
\begin{equation}%
\label{EsNoF2}%
\underset{\mu\in\mathbb{R}^{+}}%
{\textup{ess\,sup}}\|F(\mu)\|=1\,.
\end{equation}
\end{subequations}
hold.
\begin{theorem}
\label{SRTFTha}%
 Let \(x(t)\in{}L^2(\mathbb{R}^{+})\), and
\(\hat{x}(\mu)\) be the Fourier transform of \(x\), \eqref{PInF1}:
\begin{equation*}%
\hat{x}(\mu)=\int\limits_{\xi\in\mathbb{R}^{+}}\psi^{\ast }(\xi,\mu)\,x(\xi)\,d\xi\,,
\quad\mu\in\mathbb{R}^{+}\,.
\end{equation*}%

Then the Fourier transforms \(u_{{}_{\mathscr{F}_E}}(\mu)\), \(u_{{}_{\mathscr{F}_E^{\ast}}}(\mu)\)
of the functions
\((\mathscr{F}_{\!_{\scriptstyle{\mathbb{R}}^{+}}}x)(t)\),
\((\mathscr{F}_{\!_{\scriptstyle{\mathbb{R}}^{+}}}^{\ast}x)(t)\) \textup{:}
\begin{equation}%
\label{FCFT1}%
 u_{{}_{\mathscr{F}_E}}(\mu)= \hspace*{-1.8ex}\int\limits_{\xi\in\mathbb{R}^{+}}\hspace*{-1.4ex}
\psi^{\ast}(\xi,\mu)\,(\mathscr{F}_{\!_{\scriptstyle{\mathbb{R}}^{+}}}\,x)(\xi)\,d\xi,\quad
 u_{{}_{\mathscr{F}_E^{\ast}}}(\mu)= \hspace*{-1.8ex}\int\limits_{\xi\in\mathbb{R}^{+}}\hspace*{-1.4ex}
\psi^{\ast}(\xi,\mu)(\mathscr{F}_{\!_{\scriptstyle{\mathbb{R}}^{+}}}^{\ast}\,x)(\xi)\,d\xi
\end{equation}
are expressed in terms of \(\hat{x}(\mu)\) by the formula
\eqref{SpReF2}.

The functions \((\mathscr{F}_{\!_{\scriptstyle{\mathbb{R}}^{+}}}\,x)(t)\),  \((\mathscr{F}_{\!_{\scriptstyle{\mathbb{R}}^{+}}}^{\ast}\,x)(t)\) are
expressed by the formula \eqref{SpReF1}:
\begin{gather*}
(\mathscr{F}_{\!_{\scriptstyle{\mathbb{R}}^{+}}}\,x)(t)=\int\limits_{\mu\in\mathbb{R}^{+}}\psi(t,\mu)F(\mu)\hat{x}(\mu)\,\frac{d\mu}{2\pi},\\
(\mathscr{F}_{\!_{\scriptstyle{\mathbb{R}}^{+}}}^{\ast}\,x)(t)=
\int\limits_{\mu\in\mathbb{R}^{+}}\psi(t,\mu)F^{\ast}(\mu)\hat{x}(\mu)\,\frac{d\mu}{2\pi}\,.
\end{gather*}
\end{theorem}
\begin{proof}
We substitute the expression
\begin{equation*}%
x(\xi)=\int\limits_{\mu\in\mathbb{R}^{+}}
\hat{x}(\mu)\,\psi(\xi,\mu)\,\frac{d\mu}{2\pi}
\end{equation*}%
for the function \(x\), \eqref{PInF2},  into the formulas
\eqref{DTFTr} and \eqref{DTFTrA} which defines the truncated Fourier operator
\(\mathscr{F}_{\!_{\scriptstyle{\mathbb{R}}^{+}}}\) and the adjoint operator
\(\mathscr{F}_{\!_{\scriptstyle{\mathbb{R}}^{+}}}^{\ast}\).
 To curry the operators \(\mathscr{F}_{\!_{\scriptstyle{\mathbb{R}}^{+}}}\),
\(\mathscr{F}_{\!_{\scriptstyle{\mathbb{R}}^{+}}}^{\ast}\) through
the integral in \eqref{PInF2}, we have to change the order of
integration in the iterated integrals
\begin{subequations}
\label{ItInt}
\begin{equation}%
\label{ItInt1}
(\mathscr{F}_{\!_{\scriptstyle{\mathbb{R}}^{+}}}x)(t)=\int\limits_{\xi\in{\mathbb{R}^{+}}}%
\bigg(\int\limits_{\mu\in{\mathbb{R}^{+}}}
\hat{x}(\mu)\,\psi(\xi,\mu)\,\frac{d\mu}{2\pi}\bigg)\,e^{it\xi}\,d\xi\,,
\end{equation}%
\begin{equation}%
\label{ItInt2}
(\mathscr{F}_{\!_{\scriptstyle{\mathbb{R}}^{+}}}^{\ast}x)(t)=\int\limits_{\xi\in{\mathbb{R}^{+}}}%
\bigg(\int\limits_{\mu\in{\mathbb{R}^{+}}}
\hat{x}(\mu)\,\psi(\xi,\mu)\,\frac{d\mu}{2\pi}\bigg)\,e^{-it\xi}\,d\xi\,.
\end{equation}%
\end{subequations}
Usual tool to justify the change of the order of integration is
the Fubini  theorem. However the Fubini theorem is not applicable
to the iterated integrals \eqref{ItInt}. The function under the
integral is not summable with respect to \(\xi\).

\textsf{To curry the operators \(\mathscr{F}_{\!_{\scriptstyle{\mathbb{R}}^{+}}}\), \(\mathscr{F}_{\!_{\scriptstyle{\mathbb{R}}^{+}}}^{\ast}\) through the integral in
\eqref{PInF2}, we use a regularization procedure.}  Given
\(\varepsilon>0\), we define \emph{the regularization operator
\(\mathcal{R}_{\varepsilon}:\,L^2(\mathbb{R}^{+})\to{}L^2(\mathbb{R}^{+})\)},
  \begin{equation}%
  \label{RegOp}
  \mathcal{R}_{\varepsilon}x(t)=e^{-\varepsilon{}t}x(t)\,,\quad\forall x\in{}L^2(\mathbb{R}^{+})\,.
  \end{equation}
  It is clear that for every \(x\in{}L^2(\mathbb{R}^{+})\),
\[\|\mathcal{R}_{\varepsilon}x-x\|_{L^2(\mathbb{R}^{+})}\to0\ \ \textup{as}\ \ \varepsilon\to+0\,.\]
The kernel of the operator
\(\mathscr{F}_{\!_{\scriptstyle{\mathbb{R}}^{+}}}\mathcal{R}_{\varepsilon}\) can be calculated
without difficulties. Let
\(x\in{}L^2(\mathbb{R}^{+})\cap{}L^1(\mathbb{R}^{+})\). Then
\begin{gather*}
(\mathscr{F}_{\!_{\scriptstyle{\mathbb{R}}^{+}}}\mathcal{R}_{\varepsilon}x)(t)=\frac{1}{\sqrt{2\pi}}
\int\limits_{\xi\in\mathbb{R}^{+}}e^{i\xi{}{}s}e^{-\varepsilon{}\xi}x(\xi)\,d\xi\,,\\
(\mathscr{F}_{\!_{\scriptstyle{\mathbb{R}}^{+}}}^{\ast}%
\mathcal{R}_{\varepsilon}x)(t)=\frac{1}{\sqrt{2\pi}}
\int\limits_{\xi\in\mathbb{R}^{+}}e^{-i\xi{}{}s}e^{-\varepsilon{}\xi}x(\xi)\,d\xi\,,
\end{gather*}
Substituting the expression \eqref{PInF2} for the function
\(x(\xi)\) into the last formula, we present the functions
\((\mathscr{F}_{\!_{\scriptstyle{\mathbb{R}}^{+}}}\mathcal{R}_{\varepsilon}x)(t),%
\,(\mathscr{F}_{\!_{\scriptstyle{\mathbb{R}}^{+}}}^{\ast}\mathcal{R}_{\varepsilon}x)(t)\)
as the iterated integrals
\begin{subequations}
\label{IteInt}
\begin{align}
\label{IteInt1}
(\mathscr{F}_{\!_{\scriptstyle{\mathbb{R}}^{+}}}%
\mathcal{R}_{\varepsilon}x)(t)&=\frac{1}{\sqrt{2\pi}}
\int\limits_{\xi\in\mathbb{R}^{+}}e^{i\xi{}{}(t+i\varepsilon)}
\bigg(\int\limits_{\mu\in\mathbb{R}^{+}}
\,\psi(\xi,\mu)\,\hat{x}(\mu)\,\frac{d\mu}{2\pi} \bigg)\,d\xi\,,\\
\label{IteInt2}%
\hspace*{-1.0ex}
(\mathscr{F}_{\!_{\scriptstyle{\mathbb{R}}^{+}}}^{\ast}\mathcal{R}_{\varepsilon}x)(t)%
&=\frac{1}{\sqrt{2\pi}}
\int\limits_{\xi\in\mathbb{R}^{+}}\!\!\!\!e^{-i\xi{}{}(t-i\varepsilon)}
\bigg(\int\limits_{\mu\in\mathbb{R}^{+}}
\,\psi(\xi,\mu)\,\hat{x}(\mu)\,\frac{d\mu}{2\pi}\bigg)\,d\xi\,,
\end{align}
\end{subequations}
 We assume firstly
that the function \(\hat{x}(\mu)\) belongs to
\(L^2(\mathbb{R}^{+})\cap{}L^1(\mathbb{R}^{+})\). The Fubini theorem is applicable to each of
the iterated integral \eqref{IteInt}. So for every fixed
\(\varepsilon>0\) we can change the order of integration there.
Changing the order, we obtain
\begin{subequations}
\label{ReEqua}%
\begin{gather}
\label{ReEqua1}%
 (\mathscr{F}_{\!_{\scriptstyle{\mathbb{R}}^{+}}}\mathcal{R}_{\varepsilon}x)(t)=
\int\limits_{\mu\in\mathbb{R}^{+}}
\big(\mathscr{F}_{\!_{\scriptstyle{\mathbb{R}}^{+}}}%
\mathcal{R}_{\varepsilon}\psi(\,.\,,\mu)\big)(t)\hat{x}(\mu)\,%
\frac{d\mu}{2\pi}\,,\\
\label{ReEqua2}%
(\mathscr{F}_{\!_{\scriptstyle{\mathbb{R}}^{+}}}^{\ast}%
\mathcal{R}_{\varepsilon}x)(t)=
\int\limits_{\mu\in\mathbb{R}^{+}}
\big(\mathscr{F}_{\!_{\scriptstyle{\mathbb{R}}^{+}}}^{\ast}%
\mathcal{R}_{\varepsilon}\psi(\,.\,,\mu)\big)(t)\hat{x}(\mu)\,%
\frac{d\mu}{2\pi}\,,
\end{gather}
\end{subequations}
where
\begin{subequations}
\label{RFTBE}
\begin{gather}
\label{RFTBE1}
\big(\mathscr{F}_{\!_{\scriptstyle{\mathbb{R}}^{+}}}%
\mathcal{R}_{\varepsilon}\psi_{\pm}(\,.\,,\mu)\big)(t)=
\frac{1}{\sqrt{2\pi}}\int\limits_{\xi\in\mathbb{R}^{+}}
\xi^{-1/2\pm{}i\mu}e^{i(t+i\varepsilon)\xi}\,d\xi\,,\\
\label{RFTBE2}
\big(\mathscr{F}_{\!_{\scriptstyle{\mathbb{R}}^{+}}}^{\ast}%
\mathcal{R}_{\varepsilon}\psi_{\pm}(\,.\,,\mu)\big)(t)=
\frac{1}{\sqrt{2\pi}}\int\limits_{\xi\in\mathbb{R}^{+}}
\xi^{-1/2\pm{}i\mu}e^{-i(t-i\varepsilon)\xi}\,d\xi\,.
\end{gather}
\end{subequations}
The integrals in \eqref{RFTBE} can be calculated explicitly:
\begin{subequations}
\label{RFTEF}
\begin{gather}
\label{RFTEF1}
\big(\mathscr{F}_{\!_{\scriptstyle{\mathbb{R}}^{+}}}%
\mathcal{R}_{\varepsilon}\psi_{+}(\,.\,,\mu)\big)(t)=
\psi_{-}(t+i\varepsilon,\,\mu)f_{-+}(\mu),\\
\label{RFTEF2}
\big(\mathscr{F}_{\!_{\scriptstyle{\mathbb{R}}^{+}}}%
\mathcal{R}_{\varepsilon}\psi_{-}(\,.\,,\mu)\big)(t)=
\psi_{+}(t+i\varepsilon,\,\mu)f_{+-}(\mu)\,,\\
\label{RFTEF3}
\big(\mathscr{F}_{\!_{\scriptstyle{\mathbb{R}}^{+}}}^{\ast}%
\mathcal{R}_{\varepsilon}\psi_{+}(\,.\,,\mu)\big)(t)=
\psi_{-}(t-i\varepsilon,\,\mu)\overline{f_{+-}(\mu)},\\
\label{RFTEF4}
\big(\mathscr{F}_{\!_{\scriptstyle{\mathbb{R}}^{+}}}^{\ast}%
\mathcal{R}_{\varepsilon}\psi_{-}(\,.\,,\mu)\big)(t)=
\psi_{+}(t-i\varepsilon,\,\mu)\overline{f_{-+}(\mu)}\,,
\end{gather}
\end{subequations}
where \(f_{+-}(\mu)\) and \(f_{-+}(\mu)\) are the same that in
\eqref{MaEnt}, and
\begin{subequations}
\label{RegEF}
\begin{gather}
\label{RegEF1}
\psi_{+}(t\pm{}i\varepsilon,\mu)=(t\pm{}i\varepsilon)^{-1/2+i\mu}=
e^{(-1/2+i\mu)(\ln|t+i\varepsilon|\pm{}i\arg(t+i\varepsilon))}\,,\\
\label{RegEF2}
\psi_{-}(t\pm{}i\varepsilon,\mu)=(t\pm{}i\varepsilon)^{-1/2-i\mu}=
e^{(-1/2-i\mu)(\ln|t+i\varepsilon|\pm{}i\arg(t+i\varepsilon))}\,.
\end{gather}
\end{subequations}
Here
\begin{equation}
\label{ArdL}
0<\arg(t+i\varepsilon)<\pi/2\ \ \ \textup{for}\ \ t>0,\,\varepsilon>0\,.
\end{equation}
Formulas \eqref{RegEF} are derived similarly
to formulas \eqref{FTEiS}. We change variable: \(\xi\to\xi/|t+i\varepsilon|\), and
then rotate the ray of integration. From \eqref{RegEF} and
\eqref{ArdL} it follows that for every \(t\in(0,\infty),\,\mu\in(0,\infty)\)
\begin{subequations}
\label{EstEF}
\begin{gather}
\label{EstEF1}
|\psi_{+}(t+i\varepsilon,\mu)|\leq{}t^{-1/2},\ \ %
|\psi_{-}(t+i\varepsilon,\mu)|\leq{}t^{-1/2}e^{\mu\pi/2},\quad\\
\label{EstEF2}
|\psi_{+}(t-i\varepsilon,\mu)|\leq{}t^{-1/2}e^{\mu\pi/2},\  \ %
|\psi_{-}(t-i\varepsilon,\mu)|\leq{}t^{-1/2}\,.
\end{gather}
\end{subequations}
Taking into account the estimates \eqref{AVMe}, we obtain from
\eqref{RFTEF} and \eqref{EstEF} the estimates
\begin{subequations}
\label{RFTEFf}
\begin{gather}
\label{RFTEFf1}
\big|\big(\mathscr{F}_{\!_{\scriptstyle{\mathbb{R}}^{+}}}%
\mathcal{R}_{\varepsilon}\psi_{\pm}(\,.\,,\mu)\big)(t)\big|\leq
t^{-1/2},\\
\label{RFTEFf2}
\big|\big(\mathscr{F}_{\!_{\scriptstyle{\mathbb{R}}^{+}}}^{\ast}%
\mathcal{R}_{\varepsilon}\psi_{\pm}(\,.\,,\mu)\big)(t)\big|\leq
t^{-1/2}\,,
\end{gather}
\end{subequations}
which hold for every \(\mu>0,\,t>0\) and \(\varepsilon>0\). In particular, the expressions in
the right hand sides of \eqref{RFTEFf} do not depend on \(\varepsilon\). Moreover,
from \eqref{RFTEF} it follows that for every fixed \(\mu>0\) and \(t>0\), there exist the limits
\begin{subequations}
\label{LiEF}
\begin{align}
\label{LiEF3}%
 \lim_{\varepsilon\to{}+0}
\big(\mathscr{F}_E\mathcal{R}_{\varepsilon}\psi(\,.\,,\mu)\big)(t)&=
\psi(t,\mu)F(\mu)\,,\\
\label{LiEF4}%
 \lim_{\varepsilon\to{}+0}
\big(\mathscr{F}_E^{\ast}\mathcal{R}_{\varepsilon}\psi(\,.\,,\mu)\big)(t)&=
\psi(t,\mu)F^{\ast}(\mu)\,.
\end{align}
\end{subequations}
Using the Lebesgue dominating convergence theorem, we conclude that
for every fixed \(t>0\)
\begin{gather*}
\lim_{\varepsilon\to{}+0}%
\int\limits_{\mu\in\mathbb{R}^{+}}\!\!
\big(\mathscr{F}_E\mathcal{R}_{\varepsilon}%
\psi(\,.\,,\mu)\big)(t)\,\hat{x}(\mu)\,d\mu
=\int\limits_{\mu\in\mathbb{R}^{+}}\big(\psi(t,\mu)F(\mu)\big)\hat{x}(\mu)
\,d\mu\,,\\
\lim_{\varepsilon\to{}+0}%
\int\limits_{\mu\in\mathbb{R}^{+}}\!\!
\big(\mathscr{F}_E^{\ast}\mathcal{R}_{\varepsilon}%
\psi(\,.\,,\mu)\big)(t)\,\hat{x}(\mu)\,d\mu
=\int\limits_{\mu\in\mathbb{R}^{+}}\big(\psi(t,\mu)F^{\ast}(\mu)\big)\,\hat{x}(\mu)\,d\mu\,.
\end{gather*}
 Involving \eqref{ReEqua1}, we see that
\begin{gather*}
\lim_{\varepsilon\to{}+0}(\mathscr{F}_E\mathcal{R}_{\varepsilon}x)(t)=
\int\limits_{\mu\in\mathbb{R}^{+}}\big(\psi(t,\mu)F(\mu)\big)\hat{x}(\mu)
\,\frac{d\mu}{2\pi}\,,\\
\lim_{\varepsilon\to{}+0}(\mathscr{F}_E^{\ast}\mathcal{R}_{\varepsilon}x)(t)=
\int\limits_{\mu\in\mathbb{R}^{+}}\big(\psi(t,\mu)F^{\ast}(\mu)\big)\hat{x}(\mu)
\,\frac{d\mu}{2\pi}\,.
\end{gather*}
 for every fixed \(t>0\). From the other hand,
\begin{gather*}
\lim_{\varepsilon\to{}+0}%
\|(\mathscr{F}_{\!_{\scriptstyle{\mathbb{R}}^{+}}}\mathcal{R}_{\varepsilon}x)(t)-%
(\mathscr{F}_{\!_{\scriptstyle{\mathbb{R}}^{+}}}x)(t)\|_{L^2(\mathbb{R}^{+})}=0\,,\\
\lim_{\varepsilon\to{}+0}%
\|(\mathscr{F}_{\!_{\scriptstyle{\mathbb{R}}^{+}}}^{\ast}\mathcal{R}_{\varepsilon}x)(t)-%
(\mathscr{F}_{\!_{\scriptstyle{\mathbb{R}}^{+}}}^{\ast}x)(t)\|_{L^2(\mathbb{R}^{+})}=0\,,
\end{gather*}
Comparing the last formulas, we obtain that
\begin{subequations}
\label{FEFTFT}
\begin{gather}
\label{FEFTFT1}
(\mathscr{F}_{\!_{\scriptstyle{\mathbb{R}}^{+}}}x)(t)%
=\int\limits_{\mu\in\mathbb{R}^{+}}\psi(t,\mu)F(\mu)\hat{x}(\mu)
\,\frac{d\mu}{2\pi}\,,\\
\label{FEFTFT2}
(\mathscr{F}_{\!_{\scriptstyle{\mathbb{R}}^{+}}}^{\ast}x)(t)
=\int\limits_{\mu\in\mathbb{R}^{+}}\psi(t,\mu)F^{\ast}(\mu)\hat{x}(\mu)
\,\frac{d\mu}{2\pi}\,,
\end{gather}
\end{subequations}
for every \(x(t)\in{}L^2(\mathbb{R}^{+})\) for which %
\(\hat{x}(\mu)\in{}L^2(\mathbb{R}^{+})\cap{}L^1(\mathbb{R}^{+})\). Since the set
\(L^2(\mathbb{R}^{+})\cap{}L^1(\mathbb{R}^{+})\) is dense in \(L^2(\mathbb{R}^{+})\), the last
equality can be extended to all \(x(t)\in{}L^2(\mathbb{R}^{+})\). To justify
such extension, one should involve the Parseval equality taking
into account that the matrix \(F(\mu)\) is bounded,
\eqref{AVMe}\,.
\end{proof}{\ \ }
\section{Spectrum and resolvent of the operator
~\mathversion{bold}%
\(\mathcal{M}_{_{\scriptstyle F}}\).
\label{SpARes}}
\vspace{-4.0ex}
Theorem \ref{SRTFTha} claims that the operator \(\mathscr{F}_{\!_{\scriptstyle{\mathbb{R}^{+}}}}:
L^2(\mathbb{R}^{+})\to{}L^2(\mathbb{R}^{+})\) is unitary equivalent to the matrix
multiplication operator \(\mathcal{M}_{_{\scriptstyle{}F}}\) in the model space \(\mathscr{K}\),
where \(F(\mu)\)  is the matrix function of the form \eqref{Matr}:
\begin{equation}
\label{UnEq}
\mathscr{F}_{\!_{\scriptstyle{\mathbb{R}^{+}}}}=U^{-1}\mathcal{M}_{_{\scriptstyle{}F}}U\,.
\end{equation}
The unitary operator \(U\) and its inverse \(U^{-1}\) are expressed%
\,\footnote{In the spectral language, the operators \(U\) and \(U^{-1}\) are
direct and inverse expansions in eigenfunctions of the self-adjoint differential
operator \(\mathcal{L}\), which commutes with \(\mathscr{F}_{\!_{\scriptstyle{\mathbb{R}^{+}}}}\).
In the integral transforms language, \(U\) and \(U^{-1}\) are direct and inverse Mellin transforms
restricted on the vertical line \(\textup{Im}\,\zeta=\frac{1}{2}\). (See Remark \ref{Mell}.)}
by the equalities \eqref{unit} and \eqref{uniti}.

The unitary equivalence \eqref{UnEq} allows to reduce a study of the spectrum and the resolvent of
the operator \(\mathscr{F}_{\!_{\scriptstyle{\mathbb{R}^{+}}}}:\,L^2(\mathbb{R}^{+})\to{}L^2(\mathbb{R}^{+})\)
to the spectral analysis of the operator \(\mathcal{M}_{_{\scriptstyle F}}:\,\mathscr{K}\to\mathscr{K}\).

To perform a spectral analysis of the \emph{operator} \(\mathcal{M}_{_{\scriptstyle F}}\),
acting in the \emph{infinite dimensional} space \(\mathscr{K}\),
 we have to perform the spectral analysis of the \emph{\(2\times2\)} matrix \(F(\mu)\), acting in the
\emph{2 dimensional} space \(\mathbb{C}^2\). The spectral analysis of the matrix \(F(\mu)\) can be done
for \emph{each} \(\mu\in\mathbb{R}^{+}\) separately. Then we can \emph{glue} the spectrum
\(\textup{\large\(\sigma\)}(\mathcal{M}_{F})\)  of the operator  \(\mathcal{M}_{_{\scriptstyle F}}\) from the spectra \(\textup{\large\(\sigma\)}(F(\mu))\) of the matrices \(F(\mu)\),
 as well as  the resolvent of the
operator \(\mathcal{M}_{_{\scriptstyle F}}\) from the resolvents of the matrices  \(F(\mu)\):
\begin{equation}
\label{GlSp}
\textup{\large\(\sigma\)}(\mathcal{M}_{_{\scriptstyle F}})=
\Big(\bigcup_{\mu\in[0,\,\infty)}\textup{\large\(\sigma\)}(F(\mu))\Big)\bigcup\{0\}.
\end{equation}
\begin{equation}
\label{GlRes}
(z\,\mathscr{I}-\mathcal{M}_{_{\scriptstyle F}})^{-1}=
\big(\mathcal{M}_{_{{\scriptstyle (zI-F)}}}\big)^{-1}=
\mathcal{M}_{_{{\scriptstyle (zI-F)}^{-1}}}\,.
\end{equation}
(The point \(\{0\}\), which appears in the right hand side of \eqref{GlSp},
corresponds to the value \(\mu=\infty\).)

Therefore we have first to perform a spectral analysis of the matrix \(F(\mu)\) for each \(\mu\) separately.\\
\noindent
\textsf{The spectrum of the matrix \(F(\mu)\).}
According to \eqref{Matr}, the matrix \(zI-F(\mu)\) is of the form
\begin{equation}
\label{ZMF}
zI-F(\mu)=
\begin{bmatrix}
z & -f_{+-}(\mu)\\
-f_{-+}(\mu)&z
\end{bmatrix}\,
\end{equation}
where \(f_{-+},\,f_{+-}\) are of the form \eqref{MaEnt}. According to the identity
\eqref{Refl},
\begin{equation}
\label{prdi}
f_{+-}(\mu)\cdot{}f_{-+}(\mu)=\frac{i}{2\,\cosh\pi\mu}\,.
\end{equation}
Let \(D(z,\mu)\) be the determinant of the matrix \(zI-F(\mu)\):
\begin{equation}
\label{detm}
D(z,\mu)=\det(zI-F(\mu))\,.
\end{equation}
From \eqref{ZMF} and \eqref{prdi} it follows that
\begin{equation}
\label{Detm}
D(z,\mu)=z^2-\frac{i}{2\,\cosh\pi\mu}\,.
\end{equation}
For \(\mu\in[0,\infty)\), let
\begin{equation}
\label{EiVa}
\zeta(\mu)=e^{i\pi/4}\frac{1}{\sqrt{2\,\cosh\pi\mu}},
\end{equation}
so \eqref{prdi} takes form
\begin{equation}
\label{UsEq}
\zeta^2(\mu)=f_{-+}(\mu)f_{+-}(\mu)\,.
\end{equation}
Let us denote the roots of the characteristic polynomial \(D(z,\mu)\) of the matrix \(F(\mu)\):
\begin{equation}
\label{FaDe}
D(z,\mu)=(z-\zeta_{+}(\mu))\cdot(z-\zeta_{-}(\mu))\,,
\end{equation}
where
\begin{equation}
\label{EiV}
\zeta_{+}(\mu)=\zeta(\mu),\ \ \zeta_{-}(\mu)=-\zeta(\mu)\,.
\end{equation}
It is clear that \(\zeta(\mu)\not=0\), so \(\zeta_{+}(\mu)\not=\zeta_{-}(\mu)\) for every \(\mu\in[0,\infty)\).
\begin{lemma}{ \ }
\label{Spmm}
\begin{enumerate}
\item[\textup{1}.]
For \(\mu\in[0,\infty)\), the spectrum
 \(\textup{\large\(\sigma\)}(F(\mu))\) of the matrix \(F(\mu)\), \eqref{Matr}, is simple,
and consists of two different points \(\zeta_{+}(\mu)\) and \(\zeta_{-}(\mu)\): \eqref{EiV},\,\eqref{EiVa}:
\begin{equation}
\label{SpMM}
\textup{\large\(\sigma\)}(F(\mu))=\{\zeta_{+}(\mu)\,,\zeta_{-}(\mu)\}\,.
\end{equation}
\item[\textup{2}.] If \(\mu_1,\mu_2\in[0,\infty),\,\,\mu_1\not=\mu_2\), then
\(\textup{\large\(\sigma\)}(F(\mu_1))\cap\textup{\large\(\sigma\)}(F(\mu_2))=\emptyset\).
\end{enumerate}
\end{lemma}
\noindent
\textsf{The resolvent of the matrix \(F(\mu)\).}\\
From \eqref{ZMF} and the rule of inversion of a \(2\times2\) matrix it follows:
\begin{lemma}{\ \ }\\[-2.0ex]
\label{InM}
\begin{enumerate}
\item[\textup{1.}]
Given \(\mu\in[0,\infty)\), the matrix \((zI-F(\mu))^{-1}\) is\textup{:}
\begin{equation}
\label{Rmm}
(zI-F(\mu))^{-1}=\frac{1}{D(z,\mu)}
\begin{bmatrix}
z & f_{-+}(\mu)\\
f_{+-}(\mu)&z
\end{bmatrix}
\end{equation}
\item[\textup{2.}]
The matrix function \((zI-F(\mu))^{-1}\), the resolvent of the matrix
\(F(\mu)\), is a rational matrix
function of \(z\). The only singularities of this matrix function
in the extended complex plain \(\mathbb{C}_{\textup{ext}}=\mathbb{C}\cup\infty\)
are simple poles at the points \(z=\zeta_{+}(\mu)\) and \(z=\zeta_{-}(\mu)\)\textup{:}
\begin{equation}
\label{ElFrDe}
(zI-F(\mu))^{-1}=\frac{E_{+}(\mu)}{z-\zeta_{+}(\mu)}+\frac{E_{-}(\mu)}{z-\zeta_{-}(\mu)}\,,
\end{equation}
where the residue matrices \(E_{+}(\mu)\),\,\(E_{-}(\mu)\) are projector matrices of rank one.
\end{enumerate}
\end{lemma}
  The fact that the residues matrices  \(E_{+}(\mu)\),\,\(E_{-}(\mu)\) are  projector matrices of rank one
  can be obtained analyzing  the explicit expression of these matrices. The identity \eqref{UsEq}
  has be used by such analysis. We do this analysis later.
  (See the  formulae \eqref{SpPr}-\eqref{SpPrAd}).
   This can also be derived from general facts about
  resolvents of operators.
\begin{lemma}
\label{EMME}
The norm of an arbitrary \(2\times2\) matrix  \(M\),
\begin{equation*}
M=
\begin{bmatrix}
m_{11}&m_{12}\\[1.5ex]
m_{21}&m_{22}
\end{bmatrix}\,,
\end{equation*}
considered as
an operator from \(\mathbb{C}^2\) to \(\mathbb{C}^2\), admits the estimates
from above and from below
\begin{equation}
\label{EMMED}
\tfrac{1}{2}\,\textup{trace}\,(M^{\ast}M)
\leq\|M\|^2
\leq\,\textup{trace}\,(M^{\ast}M)\,.
\end{equation}
Assuming that \(\det\,M\not=0\), the norm of the inverse matrix \(M^{-1}\) can
be estimated as follows:
\begin{multline}
\label{EMMEI}
(\det\,M)|^{-2}\,\textup{trace}\,(M^{\ast}M)
-\frac{2}{\textup{trace}\,(M^{\ast}M)}\leq\,\\
\leq\|M^{-1}\|^{\,2}\,\leq\,
|(\det\,M)|^{-2}\,\textup{trace}\,(M^{\ast}M)\,
\end{multline}
where
\begin{equation}
\label{Trac}%
\textup{trace}\,M^{\ast}M=|m_{11}|^2+|m_{12}|^2+|m_{21}|^2+|m_{22}|^2\,.
\end{equation}
\end{lemma}
\begin{proof}
Let \(s_0\) and \(s_1\) be singular values of the matrix \(M\), that is
\begin{equation}
\label{SinV}
0<s_1\leq{}s_0\,,
\end{equation}
 and the numbers \(s_0^2,\,s_1^2\) are eigenvalues of the matrix
\(M^{\ast}M\). Then
\begin{gather*}
\|M\|=s_0,\quad \|M^{-1}\|=s_1^{\,-1},\\
 \textup{trace}(M^{\ast}M)=s_0^2+s_1^2\,,
\quad{}|\det(M)|^2=\det{}(M^{\ast}M)=s_0^2\cdot{}s_1^2\,.
\end{gather*}%
Therefore the inequality \eqref{EMMED} takes the form
\[\frac{1}{2}(s_{0}^{2}+s_{1}^{2})\leq{}s_{0}^{2}\leq(s_0^{2}+s_1^{2})\,,\]
and the inequality \eqref{EMMEI} takes the form
\[(s_0s_1)^{\,-2}(s_0^{\,2}+s_{1}^{\,2})\,-\,\frac{2}{s_0^{\,2}+s_1^{\,2}}\,\leq\,s_1^{\,-2}\,\leq\,
(s_0s_1)^{\,-2}(s_0^{\,2}+s_{1}^{\,2})\,.\]
The last inequalities hold for arbitrary numbers \(s_0,\,s_1\) which satisfy the
inequalities \eqref{SinV}.
\end{proof}
\begin{lemma}
\label{eremf}
For every \(\mu\in[0,\infty)\) and every \(z\in\mathbb{C}\setminus\!\textup{\large\(\sigma\)}(F(\mu))\)
the matrix \((zI-F(\mu))^{-1}\) admits the estimates
\begin{multline}
\label{EsSRe}
 |D(z,\mu)|^{-2}\big(2|z|^2+1\big)-\frac{2}{2|z|^2+1}\leq\\
\leq\|(zI-F(\mu))^{-1}\|^{2}\leq{}|D(z,\mu)|^{-2}\big(2|z|^2+1\big)\,,
\end{multline}%
where \(D(z,\mu)=\det(zI-F(\mu))\)\textup{:}\,\eqref{detm},\eqref{FaDe}, and
\(\textup{\large\(\sigma\)}(F(\mu))\) is the spectrum of the matrix \(F(\mu)\)\textup{:}\,
\eqref{SpMM},\,\eqref{EiV},\,\eqref{EiVa}.
\end{lemma}
\begin{proof} We apply the estimate \eqref{EMMEI} to the matrix \(M=zI-F(\mu)\). According to \eqref{ZMF},
 \eqref{Trac} and \eqref{SMc2},
 \begin{equation*}%
 \textup{trace}\,\big((zI-F(\mu))^{-1}\big)^{\ast}(zI-F(\mu))^{-1}=2|z|^2+1.
 \end{equation*}
\end{proof}
\noindent
\textsf{The spectrum of the operator \(\mathcal{M}_{F}\)}.\\
Let us describe the spectrum \(\textup{\large\(\sigma\)}(\mathcal{M}_{F})\) of the multiplication
operator \(\mathcal{M}_{F}\), where \(F(\mu)\) is the matrix function appeared in \eqref{Matr}.

When \(\mu\) runs over the interval \([0,\,\infty)\), the points \(\zeta_{+}(\mu)\) fill
the interval
\(\Big(0,\,\frac{1}{\sqrt{2}}\,e^{i\pi/4}\Big]\) and the points \(\zeta_{-}(\mu)\) fill
the interval \(\Big[-e^{i\pi/4}\frac{1}{\sqrt{2}},\,0\Big)\).
When \(\mu\) increases, the points \(\zeta_{+}(\mu)\), \(\zeta_{-}(\mu)\) move monotonically,
so the mappings \(\mu\to\zeta_{+}(\mu)\) is a homeomorphism of \([0,\,\infty)\) onto
\(\Big(0,\,\frac{1}{\sqrt{2}}\,e^{i\pi/4}\Big]\) and the mapping \(\mu\to\zeta_{-}(\mu)\)
is a homeomorphism of \([0,\,\infty)\) onto \(\Big[-e^{i\pi/4}\frac{1}{\sqrt{2}},\,0\Big)\).
\begin{lemma}
\label{psMO}
The multiplication operator \(\mathcal{M}_{F}\), where \(F(\mu)\) is the matrix function from
\eqref{Matr}, has neither the point nor the residual spectrum.
\end{lemma}
\begin{proof}
Let \(z\) be a complex number. Then
\(\mathcal{M}_{_{\scriptstyle F}}-z\,\mathscr{I}=\mathcal{M}_{_{\scriptstyle F-zI}}\).
According to \eqref{detm},\,\eqref{FaDe}, the determinant \(\det(F(\mu)-zI)\)
does not vanish for \(\mu\in[0,\infty)\) if
\(z\not\in\Big[-e^{i\pi/4}\frac{1}{\sqrt{2}},\,0\Big)\cup\Big(0,\,\frac{1}{\sqrt{2}}\,e^{i\pi/4}\Big]\)
and vanishes precisely in one point \(\mu(z)\) if
\(z\in\Big[-e^{i\pi/4}\frac{1}{\sqrt{2}},\,0\Big)\cup\Big(0,\,\frac{1}{\sqrt{2}}\,e^{i\pi/4}\Big]\).
In each of these two cases, \(m(\{\mu:\,\det(F(\mu)-zI)=0\})=0\).
According to Theorem~\ref{InvCon}, \(0\not\in\textup{\large\(\sigma\)}_p(\mathcal{M}_{F}-z\,\mathscr{I})\),
\(0\not\in\textup{\large\(\sigma\)}_r(\mathcal{M}_{F}-z\,\mathscr{I})\), that is
\(z\not\in\textup{\large\(\sigma\)}_p(\mathcal{M}_{F})\),\,\(z\not\in\textup{\large\(\sigma\)}_r(\mathcal{M}_{F})\).
\end{proof}
\begin{lemma}
\label{csMO}
The continuous spectrum \(\textup{\large\(\sigma\)}_c(\mathcal{M}_{F})\)
of the multiplication operator \(\mathcal{M}_{F}\), where \(F(\mu)\) is the matrix function from \eqref{Matr}, is the interval
\(\Big[-\frac{1}{\sqrt{2}}\,e^{i\pi/4},\,\frac{1}{\sqrt{2}}\,e^{i\pi/4}\Big]\).
\end{lemma}
\begin{proof}
When \(\mu\) runs  over the interval \([0,\infty)\), the complex
number \(\dfrac{i}{2\cosh{\pi\mu}}\), which appears in the right
hand side of the equality \eqref{Detm}, fill the interval
\((0,i/2]\). Therefore
\begin{equation}
\label{DistCo}%
\inf_{\mu\in(0,\infty)}|D(z,\mu)|=\textup{dist}(z^2\,,\,[0,i/2]\,)
\end{equation}%
In particular,
\begin{equation*}
\Big(\inf_{\mu\in[0,\infty)}|D(z,\mu)|>0\Big)\Leftrightarrow
\big(\,z^2\not\in[0,i/2]\,\big)\,,
\end{equation*}%
or, what is the same,
\begin{equation}
\label{SepCo}
\Big(\inf_{\mu\in[0,\infty)}|D(z,\mu)|>0\Big)\Leftrightarrow
\Big(\,z\not\in\Big[-\frac{1}{\sqrt{2}}\,e^{i\pi/4},\,
\frac{1}{\sqrt{2}}\,e^{i\pi/4}\Big]\,\Big)\,,
\end{equation}%
From the inequalities \eqref{EsSRe} it follows that
\begin{multline}
\label{CoFRe}
\frac{2|z|^2+1}{\big(\inf_{\mu\in[0,\infty)}|D(z,\mu)|\big)^{2}}-\frac{2}{2|z|^2+1}\leq\\
\leq\sup\limits_{\mu\in[0,\infty)}\|(zI-F(\mu))^{-1}\|^{2}\leq
\frac{2|z|^2+1}{\big(\inf_{\mu\in[0,\infty)}|D(z,\mu)|\big)^{2}}
\end{multline}
From \eqref{SepCo} and \eqref{CoFRe} it follows that
\begin{equation}
\label{CoFRe1}
\left(\,\sup\limits_{\mu\in[0,\infty)}\|(zI-F(\mu))^{-1}\|<\infty\right)\Leftrightarrow
\left(\,z\not\in\Big[-\frac{1}{\sqrt{2}}\,e^{i\pi/4},\,
\frac{1}{\sqrt{2}}\,e^{i\pi/4}\Big]\,\right)
\end{equation}
Lemma \ref{csMO} is a consequence of \eqref{CoFRe1} and Theorem \ref{InvCon}. (Theorem \ref{InvCon}
should be applied to the matrix function \(G(z)=zI-F(\mu)\).)
\end{proof}
\noindent
From Lemmas \eqref{psMO}, \ref{csMO} and Statement 2 of Theorem \ref{InvCon} we derive:
\begin{theorem}{\ \ }\\[-3.0ex]
\label{SpMuOpe}
\begin{enumerate}
\item[\textup{1.}]
The spectrum \(\textup{\large\(\sigma\)}(\mathcal{M}_{F})\) of the multiplication
operator \(\mathcal{M}_{F}\), where \(F(\mu)\) is the matrix function appeared in \eqref{Matr}, is:
\begin{equation}
\label{SpemuOper}
\textup{\large\(\sigma\)}(\mathcal{M}_{F})=\left[-\frac{1}{\sqrt{2}}\,e^{i\pi/4},\,
\frac{1}{\sqrt{2}}\,e^{i\pi/4}\right]
\end{equation}
\item[\textup{2.}]
If \(z\) does not belong to the spectrum \(\textup{\large\(\sigma\)}(\mathcal{M}_{F})\),
then  the resolvent \((z\mathscr{I}-\mathcal{M}_{F})^{-1}\) can be expressed as
\begin{equation}
\label{RAsMO}
(z\mathscr{I}-\mathcal{M}_{F})^{-1}=\mathcal{M}_{_{\scriptstyle (zI-F)^{-1}}}\,.
\end{equation}
\end{enumerate}
\end{theorem}
\noindent
\textsf{Estimate for the resolvent of the operator \(\mathcal{M}_{F}\)}.\\
Let \(z\) does not belong to the spectrum \(\textup{\large\(\sigma\)}(\mathcal{M}_{F})\) of the operator \(\mathcal{M}_{F}\).
According to Statement 2 of Theorem \ref{InvCon}, the resolvent \((z\mathscr{I}-\mathcal{M}_{F})^{-1}\)
of this operator can be expressed by \eqref{RAsMO}.
According to Statement 5 of Theorem \ref{InvCon}, applied to the matrix function \(G(\mu)=(zI-F(\mu))^{-1}\),
the estimate
\begin{equation}
\label{EFRe}
\|\mathcal{M}_{_{\scriptstyle (zI-F)^{-1}}}\|_{_{\mathscr{K}\to\mathscr{K}}}=\,\sup\limits_{\mu\in[0,\infty)}\|(zI-F(\mu))^{-1}\|
\end{equation}
holds. Comparing \eqref{RAsMO},\,\eqref{EFRe},\,\eqref{CoFRe} and \eqref{DistCo}, we obtain two sided estimates
for the norm of the resolvent \((z\mathscr{I}-\mathscr{F}_E)^{-1}\) in term of the value \(\textup{dist}(z^2\,,\,[0,i/2])\):
\begin{subequations}
\label{ToSiEs}%
\begin{gather}
\label{UpEs}%
\big\|(z\mathscr{I}-\mathscr{F}_E)^{-1}\big\|
\leq{}\frac{\big(2|z|^2+1\big)^{1/2}}{\textup{dist}(z^2\,,\,[0,i/2])}\,, \\
\notag%
\frac{\big(2|z|^2+1\big)^{1/2}}{\textup{dist}(z^2\,,\,[0,i/2])}
\sqrt{1-\tfrac{2\textup{dist}^2(z^2\,,\,[0,i/2])}{\big(2|z|^2+1\big)^2}}
\leq\big\|(z\mathscr{I}-\mathscr{F}_E)^{-1}\big\|\,.
\end{gather}
The value under the square root in \eqref{LoEs} is positive
since
\begin{equation*}
\label{GOH}%
\frac{2\textup{dist}^2(z^2\,,\,[0,i/2])}{\big(2|z|^2+1\big)^2}\leq
\frac{2|z|^2}{\big(2|z|^2+1\big)^2}\leq\frac{1}{2}\,.
\end{equation*}
Since \((1-\alpha)\leq\sqrt{1-\alpha}\) for \(0\leq\alpha\leq1\), then
\[1-\frac{2\textup{dist}^2(z^2\,,\,[0,i/2])}{\big(2|z|^2+1\big)^2}
\leq\sqrt{1-\frac{2\textup{dist}^2(z^2\,,\,[0,i/2])}{\big(2|z|^2+1\big)^2}}
\,.\]
Thus, the lower estimate for the norm of resolvent is
\begin{equation}
\label{LoEs}
\frac{\big(2|z|^2+1\big)^{1/2}}{\textup{dist}(z^2\,,\,[0,i/2])}-
\frac{2\textup{dist}(z^2\,,\,[0,i/2])}{\big(2|z|^2+1\big)^{3/2}}
\leq\big\|(z\mathscr{I}-\mathscr{F}_E)^{-1}\big\|\,.
\end{equation}
\end{subequations}
The smaller is the value \(\textup{dist}(z^2\,,\,[0,i/2])\), the closer are the lower estimate \eqref{LoEs}
and the upper estimate \eqref{UpEs}.
 However, we would like to estimate of the resolvent in term of the value
 \(\textup{dist}(z,\,\textup{\large\(\sigma\)}(\mathcal{M}_{F}))\).
\begin{lemma}
\label{DfDi}
Let \(\zeta\) be a point of the spectrum \(\textup{\large\(\sigma\)}(\mathcal{M}_{F})\)\textup{:}
\begin{equation}
\label{zePsp}
\zeta\in\bigg[-\frac{1}{\sqrt{2}}\,e^{i\pi/4},\,
\frac{1}{\sqrt{2}}\,e^{i\pi/4}\bigg]\,,
\end{equation}
and the point \(z\)  lies on the normal to the interval
\(\Big[-\frac{1}{\sqrt{2}}\,e^{i\pi/4},\,
\frac{1}{\sqrt{2}}\,e^{i\pi/4}\Big]\) at the point \(\zeta\):
\begin{equation}
\label{znorspm}
z=\zeta\pm{}|z-\zeta|e^{i3\pi/4}\,.
\end{equation}
Then
\begin{equation}
\label{distot}
\textup{dist}\big(z^2\,,\big[0\,,i/2\big]\,\big)\,=\,
\begin{cases}
2|\zeta|\,|z-\zeta|\,,&\textup{if \ }|z-\zeta|\leq|\zeta|\,,\\
|\zeta|^2+|z-\zeta|^2|\,=\,|z|^2,&\textup{if \ }|z-\zeta|\geq|\zeta|\,.
\end{cases}
\end{equation}
\end{lemma}
\begin{proof}
The condition \eqref{zePsp} means that
\(\zeta=\pm|\zeta|e^{i\pi/4}\). Substituting this expression for \(\zeta\) into
\eqref{znorsp}, we obtain
\begin{equation*}
z^2=\pm2|\zeta|\,|z-\zeta|+i(|\zeta|^2-|z-\zeta|^2)\,.
\end{equation*}
If \(|z-\zeta|\leq|\zeta|\), then the point \(i(|\zeta|^2-|z-\zeta|^2)\)
lies on the interval \([0,i/2]\). In this case,
\(\textup{dist}(z^2,\,[0,i/2])=2|\zeta|\,|z-\zeta|\). If \(|z-\zeta|\geq|\zeta|\),
then the point \(i(|\zeta|^2-|z-\zeta|^2)\) lies on the half-axis \([0,-i\infty)\).
In this case,
\begin{multline*}
\textup{dist}\,(z^2\,,\,[0,\,\,i/2]))=\sqrt{{\big(|\zeta|^2-|z-\zeta|^2\big)^2+
4|\zeta|^2|z-\zeta|^2}}=\\=|\zeta|^2+|z-\zeta|^2=|z|^2\,.
\end{multline*}
Since \(|\zeta|^2+|z-\zeta|^2\geq2|\zeta||z-\zeta|\), in any case, either \(|z-\zeta|\leq|\zeta|\), or \(|z-\zeta|\geq|\zeta|\), the inequality
\begin{equation}
\label{ERFB}
\textup{dist}\,(z^2\,,\,[0,\,i/2]))\geq2|\zeta|\,|z-\zeta|\,.
\end{equation}
holds.
\end{proof}
\begin{theorem}%
\label{TEstRes}%
Let \(\zeta\) be a point of the spectrum \(\textup{\large\(\sigma\)}(\mathcal{M}_{F})\)
of the operator \(\mathcal{M}_{F}\), and let the point \(z\) lie
on the normal to the interval \(\textup{\large\(\sigma\)}(\mathcal{M}_{F})\)
at the point \(\zeta\).\\
Then
\begin{enumerate}
\item[\textup{1.}]
The resolvent \((z\mathscr{I}-\mathcal{M}_{F})^{-1}\) admits the estimate from above:
\begin{equation}
\label{UpEsResSM}
\big\|(z\mathscr{I}-\mathcal{M}_{F})^{-1}\big\|\leq
A(z)\frac{1}{|\zeta|}\cdot\frac{1}{|z-\zeta|}\,,
\end{equation}
where \(A(z)=\frac{(2|z|^2+1)^{1/2}}{2}\).
\item[\textup{2.}]
If moreover the condition \(|z-\zeta|\leq|\zeta|\) is satisfied,
then the resolvent \((z\mathscr{I}-\mathcal{M}_{F})^{-1}\) also admits the estimate from below:
\begin{equation}
\label{LoEsResSM}
A(z)\frac{1}{|\zeta|}\cdot\frac{1}{|z-\zeta|}
-B(z)|\zeta||z-\zeta|
\leq\big\|(z\mathscr{I}-\mathcal{M}_{F})^{-1}\big\|
\,,
\end{equation}
 where \(A(z)\) is the same that in \eqref{UpEsResS} and
\(B(z)=\frac{4}{(2|z|^2+1)^{3/2}}\)\,.
\item[\textup{3.}] For \(\zeta=0\), then the resolvent \((z\mathscr{I}-\mathcal{M}_{F})^{-1}\)
admits the estimates
\begin{equation}
\label{EZEZM}
2A(z)\frac{1}{|z|^2}-B(z)
\leq\big\|(z\mathscr{I}-\mathcal{M}_{F})^{-1}\big\|\leq{}2A(z)\frac{1}{|z|^2},
\end{equation}
where \(A(z)\) and \(B(z)\) are the same that in \eqref{UpEsResS},
 \eqref{LoEsResS}, and \(z\) is an arbitrary point of the normal.
\end{enumerate}
In particular, if \(\zeta\not=0\), and \(z\) tends to \(\zeta\)
along the normal to the interval \(\textup{\large\(\sigma\)}(\mathcal{M}_{F})\),
then
\begin{equation}
\label{AsNzpM}
\big\|(z\mathscr{I}-\mathcal{M}_{F})^{-1}\big\|=\frac{A(\zeta)}{|\zeta|}\,\frac{1}{|z-\zeta|}
+O(1)\,.
\end{equation}
If \(\zeta=0\) and \(z\) tends to \(\zeta\)
along the normal to the interval \(\textup{\large\(\sigma\)}(\mathcal{M}_{F})\),
then
\begin{equation}
\label{AszpM}
\big\|(z\mathscr{I}-\mathcal{M}_{F})^{-1}\big\|=|z|^{-2}+O(1)\,,
\end{equation}
where \(O(1)\) is a value which remains bounded as \(z\) tends to \(\zeta\).
\end{theorem}%
\begin{proof}
The proof is based on the estimates \eqref{ToSiEs} for the resolvent
and on Lemma \ref{DfDi}. Combining the inequality \eqref{ERFB}
with the estimate \eqref{UpEs}, we obtain the estimate \eqref{UpEsResS},
which holds for \emph{all} \(z\) lying on the normal to
the interval \(\textup{\large\(\sigma\)}(\mathcal{M}_{F})\) at the point \(\zeta\).
\emph{If moreover  \(z\) is close enough to \(z\)}, namely the condition \(|z-\zeta|\leq|\zeta|\)
is satisfied, then the equality holds in \eqref{ERFB}. Combining \emph{the equality}
\eqref{ERFB} with the estimate \eqref{LoEs}, we obtain the estimate \eqref{LoEsResS}.

The asymptotic relation \eqref{AsNzp} is a consequence of the inequalities
\eqref{UpEsResS} and \eqref{LoEsResS} since \(\frac{|A(z)-A(\zeta)|}{|z-\zeta|}=O(1)\)
as \(z\) tends to \(\zeta\).

The asymptotic relation \eqref{Aszp} is a consequence of the inequalities
\eqref{ToSiEs} and the equality \(\textup{dist}\,(z^2\,,\,[0\,,\,i/2])=|z|^2\)
which holds for all \(z\) lying on the normal to
the interval \(\textup{\large\(\sigma\)}(\mathcal{M}_{F})\) at the point \(\zeta=0\).
(See \eqref{distot} for \(\zeta=0\).)
\end{proof}
\begin{remark}
\label{FormTr}
The estimates \eqref{UpEsResS} and \eqref{LoEsResS} are formally true also for \(\zeta=0\),
but in this case they are not rich in content.
\end{remark}
\begin{corollary}
\label{NoSNO}
From the asymptotic relations \eqref{AsNzp} and \eqref{Aszp} it~follows
that the operator \(\mathcal{M}_{F}\) is not similar to any normal operator. \ \
Were the operator \(\mathcal{M}_{F}\)
similar to a normal operator \(\mathscr{N}\), the resolvent  \((z\mathscr{I}-\mathcal{M}_{F})^{-1}\) would admit the estimate
\[\big\|(z\mathscr{I}-\mathcal{M}_{F})^{-1}\big\|%
\leq\,C(\mathscr{N})\,\textup{dist}\,(z,\,\textup{\large\(\sigma\)}(\mathcal{M}_{F}))\,,\]
where \(C(\mathscr{N})<\infty\) is a constant which  does not depend on \(z\). However,
this estimate  is not compatible with the asymptotic relation \eqref{Aszp}.
\end{corollary}
\section{Functional calculus for the operator
~\mathversion{bold}%
\(\mathcal{M}_{_{\scriptstyle F}}\).
\label{FuCalM}}
A direct consequence of the unitary equivalency \eqref{UnEq} is:
\begin{theorem}{\ \ }\\[-3.5ex]
\label{CoinSpe}
\begin{enumerate}
\item[\textup{1.}]
The  spectra
\(\textup{\large\(\sigma\)}(\mathscr{F}_{\!_{\scriptstyle{\mathbb{R}^{+}}}})\)
 and
 \(\textup{\large\(\sigma\)}(\mathcal{M}_{_{\scriptstyle F}})\)
coincide:
\begin{equation}
\label{CoiSpa}
\textup{\large\(\sigma\)}(\mathscr{F}_{\!_{\scriptstyle{\mathbb{R}^{+}}}})=
\textup{\large\(\sigma\)}(\mathcal{M}_{_{\scriptstyle F}})
\,.
\end{equation}
\item[\textup{2.}]
For \(z\) out of these spectra, the resolvents  and are related by the equality
\begin{equation}
\label{RelRes}
(z\,\mathscr{I}-\mathscr{F}_{\!_{\scriptstyle{\mathbb{R}^{+}}}})^{-1}=U^{-1}
(z\,\mathscr{I}-\mathcal{M}_{_{\scriptstyle F}})^{-1}U
\,.
\end{equation}
\end{enumerate}
\end{theorem}
\vspace{3.0ex}
\noindent
The unitary equivalency \eqref{UnEq} suggests the equality
\begin{equation}
\label{EqFOp}
h(\mathscr{F}_{\!_{\scriptstyle{\mathbb{R}^{+}}}})=U^{-1}h(\mathcal{M}_{_{\scriptstyle F}})U\,
\end{equation}
for functions \(h\) of the operator. The equality \eqref{EqFOp} can be interpreted differently.
From \eqref{UnEq} follows directly that the equality \eqref{EqFOp} holds for any polynomial \(h\).
Invoking Theorem \ref{CoinSpe}, we conclude that the equality \eqref{EqFOp} holds for every
rational function\footnote{%
For a rational function \(h\), the function \(h(A)\) of an operator \(A\) can be defined directly.
For a rational \(h\), all reasonable functional calculuses leads to the same result for \(h(A)\).
} %

If  \(h\) is holomorphic on the spectra \(\textup{\large\(\sigma\)}(\mathscr{F}_{\!_{\scriptstyle{\mathbb{R}^{+}}}})=
\textup{\large\(\sigma\)}(\mathcal{M}_{_{\scriptstyle F}})\),
then the equality \eqref{EqFOp} still holds:
\begin{math}
h_{\textup{hol}}(\mathscr{F}_{\!_{\scriptstyle{\mathbb{R}^{+}}}})
=U^{-1}h_{\textup{hol}}(\mathcal{M}_{_{\scriptstyle F}})U\,,
\end{math}
where the operators \(h_{\textup{hol}}(\mathscr{F}_{\!_{\scriptstyle{\mathbb{R}^{+}}}})\),
\(h_{\textup{hol}}(\mathcal{M}_{_{\scriptstyle F}})\) are defined in the sense of
the holomorphic functional calculus. To see this, we should to multiply the equality
\eqref{RelRes} with \(h(z)\) and then integrate the product along a contour which encloses the spectra
and is contained in the domain of holomorphy of the function \(h\).

If  \(h\) is a function which is not holomorphic on the
spectra \(\textup{\large\(\sigma\)}(\mathscr{F}_{\!_{\scriptstyle{\mathbb{R}^{+}}}})=
\textup{\large\(\sigma\)}(\mathcal{M}_{_{\scriptstyle F}})\), then the question ``What are
\(h(\mathscr{F}_{\!_{\scriptstyle{\mathbb{R}^{+}}}})\) and \(h(\mathcal{M}_{_{\scriptstyle F}})\)\,?\,''\,arises.
For a class of functions \(h\) which were called \(\mathscr{F}_{\!_{\scriptstyle{\mathbb{R}^{+}}}}\)\!-\,admissible,
Definition \eqref{DFAdF}, the definition of the operator \(h(\mathscr{F}_{\!_{\scriptstyle{\mathbb{R}^{+}}}})\)
was announced, \eqref{ExFCal}. However the existence of the limit in \eqref{ExFCal} is not yet proved, and the properties
of the mapping \(h\to{}h(\mathscr{F}_{\!_{\scriptstyle{\mathbb{R}^{+}}}})\) are not yet established.

In fact, we use the equality \eqref{EqFOp} as a \emph{definition} of the operator \(h(\mathscr{F}_{\!_{\scriptstyle{\mathbb{R}^{+}}}})\) in terms of the operator
\(h(\mathcal{M}_{_{\scriptstyle F}})\).

Namely, we first work not with
the the operator \(\mathscr{F}_{\!_{\scriptstyle{\mathbb{R}^{+}}}}\), but with its model -- the
operator \(\mathcal{M}_{_{\scriptstyle F}}\). To work with the operator
 \(\mathcal{M}_{_{\scriptstyle F}}\) is much easier: we can write
 explicit formulae. So, \textsf{given an \(\mathscr{F}_{\!_{\scriptstyle{\mathbb{R}^{+}}}}\)-admissible function
 \(h\), we first define the operator \(h(\mathcal{M}_{_{\scriptstyle F}})\) by an explicit for\-mu\-la.}
 In particular we justify the passage to the limit analogous to \eqref{ExFCal},
 where the resolvent \((z\,\mathscr{I}-\mathscr{F}_{\!_{\scriptstyle{\mathbb{R}^{+}}}})^{-1}\)
 is replaced by the resolvent \((z\,\mathscr{I}-\mathcal{M}_{_{\scriptstyle F}})^{-1}\)
  \textsf{Then we \emph{transplant} the constructions and the results, related to the model
 operator \(\mathcal{M}_{_{\scriptstyle F}}\), to the operator \(\mathscr{F}_{\!_{\scriptstyle{\mathbb{R}^{+}}}}\) itself.}

\begin{definition}
\label{DEMadm}
 Let \(h\) be a function defined almost everywhere on the spectrum   \(\textup{\large\(\sigma\)}(\mathcal{M}_{F})\)
 of the operator \(\mathcal{M}_{F}\). Then for almost every \(\mu\in[0,\infty)\)
 the function \(h\) is defined
 on the spectrum \(\textup{\large\(\sigma\)}(F(\mu))\) of the matrix \(F(\mu)\). Since the matrix \(F(\mu)\) is a
 matrix with simple spectrum, the matrix \(h(F(\mu))\) is defined in the sense of elementary functional calculus of matrices.
 \textup{(}See \textup{Definition \ref{DeFuMa}} below\textup{)}.

 The function \(h\) is said to be \(\mathcal{M}_{F}\)-admissible,
 if the condition
\begin{equation}
\label{ImCoOnh}
\underset{\mu\in[0,\infty)}{\textup{ess\,sup}}\,\|h(F(\mu))\|<\infty\,,
\end{equation}
is fulfilled.
\end{definition}

\begin{definition}
\label{DehM}
Let \(h\) be an \(\mathcal{M}_{F}\)-admissible function.
Then the operator \(\mathcal{M}_{h(F)}\) is defined according \textup{Definition \ref{FuCa}}\,%
\footnote{Definition \ref{FuCa} should be applied to the function \(G(\mu)=h(F(\mu))\)}%
, and we set
\begin{equation}
\label{DHoF}
h(\mathcal{M}_{_{\scriptstyle F}})\stackrel{\textup{\tiny def}}{=}\mathcal{M}_{h(F)}\,.
\end{equation}
\end{definition}

If the function \(h\) is a polynomial or a rational function holomorphic on the spectrum
 \(\textup{\large\(\sigma\)}(\mathcal{M}_{F})\), then the operator \(h(\mathcal{M}_{F})\)
 can be defined directly, without the formula \eqref{DHoF}. In this case both definition of the operator \(h(\mathcal{M}_{F})\),
 the definition \eqref{DHoF} and the direct one, coincides.

\emph{To describe the class  of \(\mathcal{M}_{F}\)-admissible functions more explicitly,
we have to obtain
the explicit expression for the matrix \(h(F(\mu))\).}\\[1.0ex]
\textsf{The spectral projectors of the matrix \(F(\mu)\).}\\
Calculated the residues \(E_{\pm}(\mu)\) at the poles \(\zeta_{\pm}(\mu)\) of the the resolvent \((zI-F(\mu))^{-1}\) explicitly,
we obtain from \eqref{Rmm} that
\begin{equation}
\label{SpPr}
E_{+}(\mu)=
\begin{bmatrix}
\dfrac{1}{2}&\dfrac{f_{-+}(\mu)}{2\zeta(\mu)}\\[2.5ex]
\dfrac{f_{+-}(\mu)}{2\zeta(\mu)}&\dfrac{1}{2}
\end{bmatrix},
\quad
E_{-}(\mu)=
\begin{bmatrix}
\dfrac{1}{2}&-\dfrac{f_{-+}(\mu)}{2\zeta(\mu)}\\[2.5ex]
-\dfrac{f_{+-}(\mu)}{2\zeta(\mu)}&\dfrac{1}{2}
\end{bmatrix}.
\end{equation}
The matrices \(E_{+}(\mu),\,E_{-}(\mu)\) satisfy the equalities
\begin{equation}
\label{PrPr}
E_{+}(\mu)^2=E_{+}(\mu),\ \ E_{-}(\mu)^2=E_{-}(\mu), \ \
\end{equation}
\begin{equation}
\label{SPr}
E_{+}(\mu)E_{-}(\mu)=E_{-}(\mu)E_{+}(\mu)=0\,,\ \ E_{+}(\mu)+E_{-}(\mu)=I\,,
\end{equation}
\begin{equation}
\label{eigPr}
F(\mu)E_{+}(\mu)=\zeta_{+}(\mu)E_{+}(\mu),\quad F(\mu)E_{-}(\mu)=\zeta_{-}(\mu)E_{-}(\mu)\,.
\end{equation}
The equalities \eqref{PrPr} means that the matrices \(E_{+}(\mu),\,E_{-}(\mu)\) are projectors, the equalities \eqref{SPr} means that these projectors are `complementary', the equality \eqref{eigPr}
means that the image subspaces of these projectors are eigensubpaces of the matrix \(F(\mu)\).
The equalities \eqref{PrPr}, \eqref{SPr}, \eqref{eigPr} are special cases of general facts related to
root subspaces of a matrix. However these equalities can be obtained directly, using the equality
\eqref{UsEq}.

The conjugate numbers \(\overline{\zeta_{+}(\mu)}\), \(\overline{\zeta_{-}(\mu)}\) are the eigenvalues
of the hermitian conjugate matrix \(F^{\ast}(\mu)\). The hermitian conjugate matrices
\(E_{+}^{\ast}(\mu)\), \(E_{-}^{\ast}(\mu)\),
\begin{equation}
\label{SpPrAd}
E_{+}^{\ast}(\mu)=
\begin{bmatrix}
\dfrac{1}{2}&\dfrac{\overline{f_{+-}(\mu)}}{2\overline{\zeta(\mu)}}\\[2.5ex]
\dfrac{\overline{f_{-+}(\mu)}}{2\overline{\zeta(\mu)}}&\dfrac{1}{2}
\end{bmatrix},
\quad
E_{-}^{\ast}(\mu)=
\begin{bmatrix}
\dfrac{1}{2}&-\dfrac{\overline{f_{+-}(\mu)}}{2\overline{\zeta(\mu)}}\\[2.5ex]
-\dfrac{\overline{f_{-+}(\mu)}}{2\overline{\zeta(\mu)}}&\dfrac{1}{2}
\end{bmatrix},
\end{equation}
are the spectral
 projectors onto the eigenspaces of the matrix \(F^{\ast}(\mu)\) corresponding to these eigenvalues:
 \begin{equation}
\label{PrPrAd}
E_{+}^{\ast}(\mu)^2=E_{+}^{\ast}(\mu),\ \ E_{-}^{\ast}(\mu)^2=E_{-}^{\ast}(\mu), \ \
\end{equation}
\begin{equation}
\label{SPrAd}
E_{+}^{\ast}(\mu)E_{-}^{\ast}(\mu)=E_{-}^{\ast}(\mu)E_{+}^{\ast}(\mu)=0\,,
\ \ E_{+}^{\ast}(\mu)+E_{-}^{\ast}(\mu)=I\,,
\end{equation}
\begin{equation}
\label{eigPrAd}
F^{\ast}(\mu)E_{+}^{\ast}(\mu)=\overline{\zeta_{+}}(\mu)E_{+}^{\ast}(\mu),\quad F^{\ast}(\mu)E_{-}^{\ast}(\mu)=\overline{\zeta_{-}}(\mu)E_{-}^{\ast}(\mu)\,.
\end{equation}

\vspace{2.0ex}
\noindent
\textsf{Functions of the matrix \(F(\mu)\).}
\vspace{-1.0ex}
\begin{definition}
\label{DeFuMa}
Given \(\mu\in[0,\infty)\), let \(h\) be a function defined on the spectrum \(\textup{\large\(\sigma\)}(F(\mu))\). Such a function is specified by a pair
 of numbers \(h(\zeta_{+}(\mu))\) and \(h(\zeta_{-}(\mu))\).
 Then the matrix \(h(F(\mu))\), which is called \textsf{the function \(h\)
 of the matrix \(F(\mu)\)}, is defined as
 \begin{equation}
 \label{hotF}
 h(F(\mu))\stackrel{\textup{\tiny def}}{=}h(\zeta_{+}(\mu))E_{+}(\mu)+h(\zeta_{-}(\mu))E_{-}(\mu)\,.
 \end{equation}
 \end{definition}
 Substituting the expressions \eqref{SpPr} for \(E_{+}(\mu),\,E_{-}(\mu)\) into
 \eqref{hotF},\,we obtain
 \begin{multline}
 \label{WoEx}
 h(F(\mu))=\\[2.0ex]
 =\begin{bmatrix}
 \dfrac{h(\zeta_{+}(\mu))+h(\zeta_{-}(\mu))}{2}\phantom{\,f_{+-}(\mu)}&
 \dfrac{h(\zeta_{+}(\mu))-h(\zeta_{-}(\mu))}{2\zeta(\mu)}\,f_{+-}(\mu)
 \\
 \dfrac{h(\zeta_{+}(\mu))-h(\zeta_{-}(\mu))}{2\zeta(\mu)}\,f_{-+}(\mu)&
 \dfrac{h(\zeta_{+}(\mu))+h(\zeta_{-}(\mu))}{2}\phantom{\,f_{+-}(\mu)}
 \end{bmatrix}
 \end{multline}
 The following lemma summarizes  the well known facts about functions of matrices.
\begin{lemma}
\label{MaFuCal}
 It is well known that the correspondence \(h(\zeta)\to{}h(F(\mu))\) is
 the algebraic homomorphism of the set of functions defined on the two-point set
 \(\textup{\large\(\sigma\)}(F(\mu))\) into the set
 \(\mathfrak{M}_{2,2}\)
  of \(2\times2\) matrices:
\begin{enumerate}
\item[\textup{1.}]
If \(h(\zeta)=1\), that is \(h(\zeta_{+}(\mu))=1,\,h(\zeta_{-}(\mu))=1\), then \(h(F(\mu))=~I.\)
\item[\textup{2.}]
If \(h(\zeta)=\zeta\), that is \(h(\zeta_{+}(\mu))=\zeta_{+}(\mu),\,h(\zeta_{-}(\mu))=\zeta_{-}(\mu)\),
 then \(h(F(\mu))=F(\mu).\)
 \item[\textup{3.}] If \(h(\zeta)=\alpha_{1}h_{1}(\zeta)+\alpha_{2}h_{2}(\zeta)\), where \(\alpha_1,\,\alpha_2\)
 are complex numbers and \(h_{1},\,h_{2}\) are functions defined on  the spectrum
 \(\textup{\large\(\sigma\)}(F(\mu))\) : \(h(\zeta_{\pm}(\mu))=\alpha_{1}h_{1}(\zeta_{\pm}(\mu))
 +\alpha_{2}h_{2}(\zeta_{\pm}(\mu))\), then
 \begin{equation*}
 h(F(\mu))=\alpha_1[h_{1}(F(\mu))+\alpha_2h_{2}(F(\mu))\,.
 \end{equation*}
  \item[\textup{4.}] If \(h(\zeta)=h_{1}(\zeta)\,h_{2}(\zeta)\), where
   \(h_{1},\,h_{2}\) are functions defined on  the spectrum
 \(\textup{\large\(\sigma\)}(F(\mu))\) : \(h(\zeta_{\pm}(\mu))=h_{1}(\zeta_{\pm}(\mu))\,
 h_{2}(\zeta_{\pm}(\mu))\), then
 \begin{equation*}
 h(F(\mu))=h_{1}(F(\mu))\,h_{2}(F(\mu))\,.
 \end{equation*}
  \item[\textup{5.}] If \(z\not\in\textup{\large\(\sigma\)}(F(\mu))\), that is \(z\not=\zeta_{\pm}(\mu)\), and
  \(h(\zeta)=(z-\zeta)\) for \(\zeta=\zeta_{+}(\mu)\) and \(\zeta=\zeta_{-}(\mu)\), then
  \(h(F(\mu))=(zI-F(\mu))^{-1}\).
  \item[\textup{6.}] The definition \eqref{hotF} of the function \(h\) of the matrix \(F(\mu)\) is compatible with
  the holomorphic operator calculus. If \(h\) is a function holomorphic on \(\textup{\large\(\sigma\)}(F(\mu))\),
  then \(h(F(\mu))=h_{\textup{hol}}(F(\mu))\)
\end{enumerate}
\end{lemma}
The properties 1\,-\,4 of the correspondence \(h\to{}h(F(\mu))\) are not evident from
the expression \eqref{WoEx} for \(h(F(\mu))\), but are evident from the expression \eqref{hotF}
and the properties \eqref{PrPr}, \eqref{SPr}, \eqref{eigPr} of the matrices \(E_{+}(\mu),\,E_{-}(\mu)\).

\vspace{2.0ex}
\noindent
\textsf{The norm in the set of functions defined on \(\textup{\large\(\sigma\)}(F(\mu))\).}\\
Let \(\mu\in[0,\infty)\).
For a function \(h(\zeta)=\{h(\zeta_{+}(\mu)),\,h(\zeta_{-}(\mu))\}\), defined on
\(\textup{\large\(\sigma\)}(F(\mu))\),  we set
\begin{equation}
\label{NFMa}
\|h\|_{\mu}\stackrel{\textup{\tiny def}}{=}
\left|\frac{h(\zeta_{+}(\mu))+h(\zeta_{-}(\mu))}{2}\right|+
\left|\frac{h(\zeta_{+}(\mu))-h(\zeta_{-}(\mu))}{2\zeta(\mu)}\right|\,.
\end{equation}
It is clear that the mapping \ \(h\to{}\|h\|_{\mu}\) \ is a norm on the set all  functions defined on
\(\textup{\large\(\sigma\)}(F(\mu))\).
\begin{lemma}
\label{RiPr}
If \(h=h_{1}\,h_{2}\), where \(h_1,\,h_2\) be functions defined on \(\textup{\large\(\sigma\)}(F(\mu))\), then
\begin{equation}
\label{RiPrE}
\|h\|_{\mu}\leq\|h_{1}\|_{\mu}\,\|h_{2}\|_{\mu}\,.
\end{equation}
\end{lemma}
\begin{proof}
For \(j=1,\,2\), let \(\frac{h_{j}(\zeta_{+}(\mu)+h_{j}(\zeta_{-}(\mu))}{2}=a_{j}\),
 \(\frac{h_{j}(\zeta_{+}(\mu)-h_{j}(\zeta_{-}(\mu))}{2}=~b_{j}\). Then
 \(h_{j}(\zeta_{+}(\mu)=a_{j}+b_{j},\,h_{j}(\zeta_{-}(\mu)=a_{j}-b_{j}\), so
 \[h(\zeta_{+}(\mu))=(a_1+b_1)(a_2+b_{2})\,,h(\zeta_{-}(\mu))=(a_1-b_1)(a_2-b_{2})\,.\]
 The inequality
 \eqref{RiPrE} takes the form
 \begin{multline*}
\left| \tfrac{(a_1+b_1)(a_2+b_{2})+(a_1-b_1)(a_2-b_{2})}{2}\right|+
\left| \tfrac{(a_1+b_1)(a_2+b_{2})-(a_1-b_1)(a_2-b_{2})}{2\zeta(\mu)}\right|\leq\\
\leq\left(|a_{1}|+|b_{1}|/|\zeta(\mu)|\right)\,\left(|a_{2}|+|b_{2}|/|\zeta(\mu)|\right)\,,
 \end{multline*}
 which can be reduced to the form
 \begin{multline*}
 \left|{a_1\,a_2+b_1\,b_{2}}\right|+
\left| \frac{a_1\,b_2+b_1\,a_{2}}{\zeta(\mu)}\right|\leq\\
\leq\left(|a_{1}|+\frac{|b_{1}|}{|\zeta(\mu)|}\right)\,\left(|a_{2}|+\frac{|b_{2}|}{|\zeta(\mu)|}\right)\,.
 \end{multline*}
 Since \(|\zeta(\mu)\leq{}1\), the last inequality holds for any \(a_{1},\,b_{1},\,a_{2},\,b_{2}\).
\end{proof}
 The definition \eqref{NFMa} of the norm \(\|h\|_{\mu}\)  is motivated by the expression
\eqref{WoEx} for the matrix \(h(F(\mu)\). It is natural to estimate the norm
\(\|h(F(\mu))\|\) of the matrix \(h(F(\mu))\) in terms of the norm \(\|h\|_{\mu}\) of the function \(\|h\|\).\\[2.0ex]
\textsf{One more estimate for the norm of \(2\times2\) matrix.}
\vspace*{-1.5ex}
\begin{lemma}
\label{Estm1}
Let
\begin{equation}
\label{2mat}
M=
\begin{bmatrix}
m_{11}&m_{12}\\[1.5ex]
m_{21}&m_{22}
\end{bmatrix}
\end{equation}
be \(2\times2\) matrix. Then the norm of  the matrix \(M\), considered as
an operator from \(\mathbb{C}^2\) to \(\mathbb{C}^2\), admits the estimates
from above
\begin{equation}
\label{Eab}
\|M\|\leq\max\big(|m_{11}|+|m_{12}|,\,|m_{21}|+|m_{22}|,\,
|m_{11}|+|m_{21}|,\,|m_{12}|+|m_{22}|
\big)\,,
\end{equation}
and from below
\begin{multline}
\label{Ebe}
\frac{1}{\sqrt{2}}\cdot\max\big(|m_{11}|+|m_{12}|,\,|m_{21}|+|m_{22}|,\,
|m_{11}|+|m_{21}|,\,|m_{12}|+|m_{22}|
\big)\leq\\
\leq\,\|M\|\,.
\end{multline}
\end{lemma}
\begin{proof} The estimate \eqref{Eab} is a special case of the Schur estimate.
(See for example \cite{}).
To obtain the estimate \eqref{Ebe}, we apply the matrix \(M\) to the vector
\(x=\begin{bmatrix}1\\ 0\end{bmatrix}\). Then
\(Mx=\begin{bmatrix}m_{11}\\ m_{21}\end{bmatrix}\). For the considered \(x\), the inequality
\(\|Mx\|\leq\|M\|\cdot\|x\|\) is the inequality
\(|m_{11}|^2+|m_{21}|^2\leq\|M\|^2\). Since
\(\frac{1}{2}\cdot(|m_{11}|+|m_{21}|)^2\leq{}|m_{11}|^2+|m_{21}|^2\)\,, the
inequality
\(\frac{1}{\sqrt{2}}\cdot(|m_{11}|+|m_{21}|)\leq\|M\|\) holds.
Choosing \(x=\begin{bmatrix}0\\ 1\end{bmatrix}\), we come to the inequality
\(\frac{1}{\sqrt{2}}\cdot(|m_{12}|+|m_{22}|)\leq\|M\|\). Applying these reasoning
to the ajoint matrix \(M^{\ast}\) and taking into account that
\(\|M\|=\|M^{\ast}\|\), we get the inequalities
\(\frac{1}{\sqrt{2}}\cdot(|m_{11}|+|m_{12}|)\leq\|M\|\)
and
\(\frac{1}{\sqrt{2}}\cdot(|m_{21}|+|m_{22}|)\leq\|M\|\).The estimate
\eqref{Ebe} summarizes these four inequalities.
\end{proof}
\noindent
\textsf{Estimates of the matrix \(h(F(\mu))\).}
\vspace*{-1.5ex}
\begin{lemma}
\label{emhm}
 Given \(\mu\in[0,\infty)\), let \(F(\mu)\) be the matrix
\eqref{Matr}, \eqref{MaEnt}, and \(h\) be a function defined on the
spectrum \(\textup{\large\(\sigma\)}(F(\mu))=\{\zeta_{+}(\mu)\,,\zeta_{-}(\mu)\}\)
of the matrix \(F(\mu)\). Then the matrix \(h(F(\mu))\), \eqref{WoEx}, admits the
estimates
\begin{equation}
\label{TSEf}
\hfill\tfrac{1}{2}\,\|h\|_{\mu}
\leq\|h(F(\mu))\|\leq\|h\|_{\mu}\,,
\end{equation}
where \(\|h\|_{\mu}\) is defined by \eqref{NFMa}.
\end{lemma}
\begin{proof} We apply the estimates \eqref{Eab}, \eqref{Ebe}, where the matrix which appears in the right
hand side of \eqref{WoEx} is taken as the matrix \(M\).
 Applying the estimates  \eqref{Eab}, \eqref{Ebe}, we should take
into account the estimates \eqref{MaFC} for \(|f_{\pm}(\mu)|\).
\end{proof}
\begin{lemma}
\label{CoinsTAd}
Let \(h(\zeta)\) be a function defined almost everywhere on the interval
\(\textup{\large\(\sigma\)}(\mathscr{F}_{\!_{\scriptstyle{\mathbb{R}^{+}}}})=
\textup{\large\(\sigma\)}(\mathcal{M}_{_{\scriptstyle F}})\).

The function \(h\) is \(\mathscr{F}_{\!_{\scriptstyle{\mathbb{R}^{+}}}}\)-admissible,
\textup{(Definition \ref{DFAdF})},
if and only if it is \(\mathcal{M}_{_{\scriptstyle F}}\)-admissible
\textup{(Definition \ref{DEMadm})}.
\end{lemma}
\begin{proof} In view of \eqref{NFAF},
\[\|h\|_{{}_{\scriptstyle\mathscr{F}_{_{\mathbb{R}_{+}}}}}=
\underset{\mu\in[0,\infty)}{\textup{ess\,sup}}\,\|h\|_{\mu}\,\]
where \(\|h\|_{\mu}\) is defined by \eqref{NFMa}.
In view of \eqref{TSEf},
\begin{equation}
\label{TwSiEsh}
\frac{1}{2}\,\|h\|_{{}_{\scriptstyle\mathscr{F}_{_{\mathbb{R}_{+}}}}}\!\!\!\!\leq
\underset{\mu\in[0,\infty)}{\textup{ess\,sup}}\,\|h(F(\mu))\|
\leq\|h\|_{{}_{\scriptstyle\mathscr{F}_{_{\mathbb{R}_{+}}}}}\,.
\end{equation}
\end{proof}

Now we can \emph{define} the operator \(h(\mathscr{F}_{\!_{\scriptstyle{\mathbb{R}^{+}}}})\)
by the formula \eqref{EqFOp}, where the operator \(h(\mathcal{M}_{_{\scriptstyle F}})\)
is defined according to Definition \ref{DehM}, by the equality \eqref{DHoF}:
\begin{definition}
\label{ModFuCalDe}
Let \(h\) be an \(\mathscr{F}_{\!_{\scriptstyle{\mathbb{R}^{+}}}}\)-admissible function.
The operator \(h_{\textup{mod}}(\mathscr{F}_{\!_{\scriptstyle{\mathbb{R}^{+}}}})\) is
defined as
\begin{equation}
\label{ModFuCal}
h_{\textup{mod}}(\mathscr{F}_{\!_{\scriptstyle{\mathbb{R}^{+}}}})%
\stackrel{\textup{\tiny def}}{=}
U^{-1}\mathcal{M}_{h(F)}U.
\end{equation}
The correspondence \(h(\zeta)\to{}h_{\textup{mod}}(\mathscr{F}_{\!_{\scriptstyle{\mathbb{R}^{+}}}})\)
is said to be the model-based function calculus for the operator
\(\mathscr{F}_{\!_{\scriptstyle{\mathbb{R}^{+}}}}\).
\end{definition}
\begin{theorem}{\ }%
\label{PrMofFuCa}
\begin{enumerate}
\item[\textup{1.}]
The mapping \(h\to{}h_{\textup{mod}}(\mathscr{F}_{\!_{\scriptstyle{\mathbb{R}^{+}}}})\)
is an algebraic homomorphism of the Banach algebra \(\mathfrak{B}_{_{{\scriptstyle\mathscr{F}}_{_{{\mathbb{R}}^{+}}}}}\)\!\!\! %
of all \(\mathscr{F}_{\!_{\scriptstyle{\mathbb{R}^{+}}}}\)-admissible functions
into the Banach algebra of all bounded operators in \(L^2(\mathbb{R^{+}})\).
\item[\textup{2.}] The mapping \(h\to{}h_{\textup{mod}}(\mathscr{F}_{\!_{\scriptstyle{\mathbb{R}^{+}}}})\)
    is a homeomorphism of these Banach algebras.  Moreover the two sided estimate
    \begin{equation}
    \label{TwSiEsti}
    \frac{1}{2}\,\|h\|_{{}_{\scriptstyle\mathscr{F}_{_{\mathbb{R}_{+}}}}}\!\!\!\!\leq
\|h_{\textup{mod}}(\mathscr{F}_{\!_{\scriptstyle{\mathbb{R}^{+}}}})\|
\leq\|h\|_{{}_{\scriptstyle\mathscr{F}_{_{\mathbb{R}_{+}}}}}\,
    \end{equation}
    holds.
\item[\textup{3.}]  Let \(\{h_n\}_{1\leq{}n<\infty}\) be a sequence  of \(\mathscr{F}_{\!_{\scriptstyle{\mathbb{R}^{+}}}}\)\!-\,admissible functions.
    We assume that:
     \begin{enumerate}
     \item[\textup{a)}.]
     The \(\mathscr{F}_{\!_{\scriptstyle{\mathbb{R}^{+}}}}\)\!-\,norms \(\|h_n\|_{{}_{\scriptstyle\mathscr{F}_{_{\mathbb{R}_{+}}}}}\) of these functions are uniformly bounded:
     \begin{equation}
     \label{unibou}
     \sup_n\|h_n\|_{{}_{\scriptstyle\mathscr{F}_{_{\mathbb{R}_{+}}}}}<\infty\,,
     \end{equation}
     \item[\textup{b)}.]
     For almost every \(\zeta\) from the interval \(\textup{\large\(\sigma\)}(\mathscr{F}_{\!_{\scriptstyle{\mathbb{R}^{+}}}})=
\textup{\large\(\sigma\)}(\mathcal{M}_{_{\scriptstyle F}})\), there exists the limit
\begin{equation}
\label{PWiLi}
h(\zeta)=\lim_{n\to\infty}h_n(\zeta)\,.
\end{equation}
     \end{enumerate}
     Then the function \(h(\zeta)\) is \(\mathscr{F}_{\!_{\scriptstyle{\mathbb{R}^{+}}}}\)\!-\,admissible, and
     \begin{equation}
\label{StrLi}
h_{\textup{mod}}(\mathscr{F}_{\!_{\scriptstyle{\mathbb{R}^{+}}}})=
\lim_{n\to\infty}(h_n)_{\textup{mod}}(\mathscr{F}_{\!_{\scriptstyle{\mathbb{R}^{+}}}})\,,
\end{equation}
where the limit is a strong limit of a sequence of operators.
\item[\textup{4.}] The model functional calculus is an extension of the holomorphic
functional calculus: if the function \(h\) is holomorphic on the spectrum
\(\textup{\large\(\sigma\)}(\mathscr{F}_{\!_{\scriptstyle{\mathbb{R}^{+}}}})\),
then
\begin{equation}
\label{Comphomo}
h_{\textup{mod}}(\mathscr{F}_{\!_{\scriptstyle{\mathbb{R}^{+}}}})=
h_{\textup{hol}}(\mathscr{F}_{\!_{\scriptstyle{\mathbb{R}^{+}}}})\,.
\end{equation}
\end{enumerate}
\end{theorem}
\begin{proof}{\ }
The statement \textsf{1} of the theorem is a consequence of the properties
of the mapping \(h(\zeta)\to{}h(F(\mu))\) (Lemma \ref{MaFuCal}, statements 1\,-\,4),
of the properties of the mapping \(G\to{}\mathcal{M}_{G}\)
(Lemma \ref{Homom}, Statements 1\,-\,3.), and of the equality \eqref{UnEq}.

The Statement \textsf{2} of the theorem is a  consequence of the estimate \eqref{TwSiEsh}
and of the equality \eqref{NoP}, applied to \(G(\mu)=h(F(\mu))\).

The statement \textsf{3} of the theorem is a consequence of Lemma \ref{StrConL}
and of the following simple fact.\emph{Given \(\mu\in[0,\infty)\), if the sequences
of numbers \(h_n(\zeta_{+}(\mu))\), \(h_n(\zeta_{-}(\mu))\) converge to a numbers
\(h(\zeta_{+}(\mu)),\,h(\zeta_{-}(\mu))\) respectively, then the
sequence of matrices \(h_n(F(\mu))\) converges to the matrix \(h(F(\mu))\).}

Now we turn to Statement \textsf{4}.
Let \(h\) be a function holomorphic on the spectrum \(\textup{\large\(\sigma\)}(\mathscr{F}_{\!_{\scriptstyle{\mathbb{R}^{+}}}})\).
We choose a contour \(\gamma\) which encloses the interval \(\textup{\large\(\sigma\)}(\mathscr{F}_{\!_{\scriptstyle{\mathbb{R}^{+}}}})\)
and is contained in the domain of holomorphy of the function \(h\). By definition,
\begin{equation*}
h_{\textup{hol}}(\mathscr{F}_{\!_{\scriptstyle{\mathbb{R}^{+}}}})=\frac{1}{2\pi{}i}%
\int\limits_{\gamma}\!\!h(z)(z\,\mathscr{I}-\mathscr{F}_{\!_{\scriptstyle{\mathbb{R}^{+}}}})^{-1}dz\,.
\end{equation*}
Using \eqref{RelRes} and \eqref{GlRes}, we rewrite this formula as
\begin{equation*}
h_{\textup{hol}}(\mathscr{F}_{\!_{\scriptstyle{\mathbb{R}^{+}}}})=U^{-1}\Big(\frac{1}{2\pi{}i}%
\int\limits_{\gamma}\!\!h(z)\mathcal{M}_{_{\scriptstyle(zI- F)^{-1}}}dz\,\Big)\,U
\end{equation*}
The operator \(\mathcal{M}_{_{\scriptstyle(zI- F)^{-1}}}\) is the multiplication operator
on the matrix \((zI- F(\mu))^{-1}\) acting in the space \(\mathscr{K}\) of functions
\emph{of the variable \(\mu\).} The variable \(z\) is a parameter. Therefore we can permute
the multiplication with respect to \(\mu\) and the integration with respect to \(z\):
\begin{equation*}
\frac{1}{2\pi{}i}%
\int\limits_{\gamma}\!\!h(z)\mathcal{M}_{_{\scriptstyle(zI- F)^{-1}}}dz=\mathcal{M}_{_{\scriptstyle G}}\,,
\end{equation*}
where
\begin{equation*}
G(\mu)=\frac{1}{2\pi{}i}%
\int\limits_{\gamma}\!\!h(z)(zI-F(\mu))^{-1}dz\,.
\end{equation*}
Substituting the expression \eqref{ElFrDe} for the resolvent \((zI-F(\mu))^{-1}\) in the last formula, we obtain that
\begin{equation*}
G(\mu)=h(\zeta_{+}(\mu))E_{+}(\mu)+h(\zeta_{-}(\mu))E_{-}(\mu)=h(F(\mu))\,.
\end{equation*}
So
\begin{equation*}
\frac{1}{2\pi{}i}%
\int\limits_{\gamma}\!\!h(z)\mathcal{M}_{_{\scriptstyle(zI- F)^{-1}}}dz=\mathcal{M}_{_{\scriptstyle h(F)}}\,
\end{equation*}
and
\begin{equation*}
h_{\textup{hol}}(\mathscr{F}_{\!_{\scriptstyle{\mathbb{R}^{+}}}})=U^{-1}\mathcal{M}_{_{\scriptstyle h(F)}}U=
h_{\textup{mod}}(\mathscr{F}_{\!_{\scriptstyle{\mathbb{R}^{+}}}})\,.
\end{equation*}
\end{proof}

\section{Spectral projectors of the operator
\mathversion{bold}%
\(\mathcal{M}_{_{\scriptstyle F}}\).
\mathversion{normal}
\label{SpPrMO}}
 Explicit expressions for the spectral projectors \(E_{+}(\mu),\,E_{-}(\mu)\)  of the matrix \(F(\mu)\)
  is done by \eqref{SpPr}.
\begin{lemma}
\label{EsMaPrLe}
\begin{subequations}
\label{EsMaPro}%
 The norms of the spectral projectors
\(E_{+}(\mu),\,E_{-}(\mu)\) of the matrix \(F(\mu)\) are:
\begin{equation}
\label{EsMaPro1} \|E_{+}(\mu)\|=\cosh{}\frac{\pi\mu}{2},\quad
\|E_{-}(\mu)\|=\cosh{}\frac{\pi\mu}{2}\,,\quad 0<\mu<\infty\,.
\end{equation}
In terms of the eigenvalue \(\zeta_{\pm}(\mu)\) of the matrix
\(F(\mu)\),
\begin{align}
\|E_{+}(\mu)\|&=\frac{1}{2|\zeta_{+}(\mu)|}\,\sqrt{1+2|\zeta_{+}(\mu)|^2},\notag\\[-2.0ex]
\\[-2.0ex]
\|E_{-}(\mu)\|&=\frac{1}{2|\zeta_{-}(\mu)|}\sqrt{1+2|\zeta_{-}(\mu)|^2}\,.\notag
\end{align}
\end{subequations}
\end{lemma}
\begin{proof}
Calculation of the norm of the matrices
\(E_{+}(\mu),\,E_{-}(\mu)\)  can be reduced to calculation of the
norms of the appropriate self-adjoint matrices:
\[\|E_{+}(\mu)\|^2=\|E_{+}^{\ast}(\mu)\,E_{+}(\mu)\|,\quad
\|E_{-}(\mu)\|^2=\|E_{-}^{\ast}(\mu)\,E_{-}(\mu)\|
\]
In its turn, calculation of the norm of a selfadjoin matrix can be
reduced to calculation of its maximal eigenvalues. The
characteristic equation for the matrix
\(\|E_{+}^{\ast}(\mu)\,E_{+}(\mu)\|\) is:
\[p^2-(\textup{trace\,}E_{+}^{\ast}(\mu)\,E_{+}(\mu))p+\det{}E_{+}^{\ast}(\mu)\,E_{+}(\mu)=0\,.\]
By direct computation,
\(\det{}E_{+}^{\ast}(\mu)\,E_{+}(\mu)=|\det{}E_{+}(\mu)|^2=0\),
(recall that \(E_{+}(\mu)\) is a matrix of rank one), and
\[\textup{trace\,}E_{+}^{\ast}(\mu)\,E_{+}(\mu)
=\frac{1}{4}\bigg(2+\frac{|f_{+-}(\mu)|^2}{|\zeta(\mu)|^2}+\frac{|f_{-+}(\mu)|^2}{|\zeta(\mu)|^2}\bigg)\,.\]
According to \eqref{MaEnt} and \eqref{AbGa},
\begin{equation}
\label{AuEq}
\frac{|f_{+-}(\mu)|^2}{|\zeta(\mu)|^2}=e^{\mu\pi},\,\quad
\frac{|f_{-+}(\mu)|^2}{|\zeta(\mu)|^2}=e^{-\mu\pi}\,.
\end{equation}
Thus,
\[\textup{trace\,}E_{+}^{\ast}(\mu)\,E_{+}(\mu)=\frac{1}{4}(2+e^{\mu\pi}+e^{-\mu\pi})=
\cosh^2\frac{\mu\pi}{2}\,.\] So, the roots of the above
characteristic equation are \(p=0\), and
\(p=\cosh^2\frac{\mu\pi}{2}\). Therefore,
\(\|E_{+}^{\ast}(\mu)\,E_{+}(\mu)\|=\cosh^2\frac{\mu\pi}{2}\), and
\(\|E_{+}(\mu)\|=\cosh{}\frac{\mu\pi}{2}\). Analogously,
\(\|E_{-}(\mu)\|=\cosh{}\frac{\mu\pi}{2}\).
\end{proof}
\begin{theorem} {\ }\\
\label{EivF}
\begin{enumerate}
\item[\textup{1.}]
The matrices \(E_{+}(\mu)\) and \(E_{-}(\mu)\), which  are of rank one, admit the factorizations
\begin{equation}
\label{FaMa}
E_{+}(\mu)={u}_{+}(\mu){v}_{+}^{\ast}(\mu),\ \
E_{-}(\mu)={u}_{-}(\mu){v}_{-}^{\ast}(\mu),
\end{equation}
where
\begin{subequations}
\begin{alignat}{2}
{u}_{+}(\mu)&=\frac{1}{\sqrt{2}}
\begin{bmatrix}
1\\[2.5ex]
\dfrac{f_{+-}(\mu)}{\zeta(\mu)}
\end{bmatrix}\,,&\quad
{u}_{-}(\mu)&=\frac{1}{\sqrt{2}}
\begin{bmatrix}
1\\[2.5ex]
-\dfrac{f_{+-}(\mu)}{\zeta(\mu)}
\end{bmatrix}\,,\\[2.0ex]
{v}_{+}(\mu)&=\frac{1}{\sqrt{2}}
\begin{bmatrix}
1\\[2.5ex]
\dfrac{\overline{f_{-+}(\mu)}}{\overline{\zeta(\mu)}}
\end{bmatrix}\,,&\quad
{v}_{-}(\mu)&=\frac{1}{\sqrt{2}}
\begin{bmatrix}
1\\[2.5ex]
-\dfrac{\overline{f_{-+}(\mu)}}{\overline{\zeta(\mu)}}
\end{bmatrix}\,.
\end{alignat}
\end{subequations}
\item[\textup{2.}] The vectors \(u_{+}(\mu),\,{u}_{-}(\mu)\) are eigenvectors of the matrix
\(F(\mu)\)\textup{:}
\begin{subequations}
\begin{equation}
F(\mu)u_{+}(\mu)=\zeta_{+}(\mu)u_{+}(\mu),\quad F(\mu)u_{-}(\mu)=\zeta_{-}(\mu)u_{-}(\mu)\,,
\end{equation}
 The vectors \(v_{+}(\mu),\,{v}_{-}(\mu)\) are eigenvectors of the matrix
\(F^{\ast}(\mu)\)\textup{:}
\begin{equation}
F^{\ast}(\mu)v_{+}(\mu)=\overline{\zeta_{+}(\mu)}v_{+}(\mu),\quad
 F^{\ast}(\mu)v_{-}(\mu)=\overline{\zeta_{-}(\mu)}v_{-}(\mu)\,.
\end{equation}
\end{subequations}
\end{enumerate}
\end{theorem}
\begin{proof}
The proof is performed by direct calculation. The equality \eqref{UsEq} is involved in this calculation.
\end{proof}

Let us calculate the "angle" between the eigenvectors
\(u_{+}(\mu)\) and  \(u_(\mu)\) of the matrix \(F(\mu)\).
We recall that if \(x\) and \(y\) are non-zero vectors of some Hilbert space
provided by the scalar product \(\langle\,.\,,\,.\,\rangle\), then the angle \(\theta(x,y)\) between \(x\) and \(y\)
is the unique \(\theta\in[0,\pi/2]\) such that
\begin{equation}
\label{Ang}
\cos^{\,2}\theta=
\frac{|\langle\,x\,,\,y\,
\rangle|^2}{\langle\,x\,,\,x\,\rangle\,\langle\,y\,,\,y\,\rangle}\,\cdot
\end{equation}
\begin{lemma}
\label{Abev}
The angles \(\theta(u_{+}(\mu),u_{-}(\mu))\) and
\(\theta(v_{+}(\mu),v_{-}(\mu))\)
 between eigenvectors of the matrices \(F(\mu)\) and \(F^{\ast}(\mu)\) respectively are:
\begin{equation}
\label{ABev}
\sin\theta(u_{+}(\mu),u_{-}(\mu))=\frac{1}{\cosh\frac{\mu\pi}{2}},\,\quad
\sin\theta(v_{+}(\mu),v_{-}(\mu))=\frac{1}{\cosh\frac{\mu\pi}{2}}\,.
\end{equation}
\end{lemma}
\begin{proof}
The proof is performed by direct calculation using the equalities \eqref{AuEq}.
\end{proof}
\begin{remark}
\label{ExAn}
The equalities \eqref{ABev} can be presented in terms of the eigenvalues
\(\zeta_{+}(\mu),\,\zeta_{-}(\mu)\) of the matrices \(F(\mu),\,F^{\ast}(\mu)\).
According to \eqref{EiVa},
\begin{equation}
\label{TrEq}
\cosh\frac{\pi\mu}{2}=\frac{\sqrt{1+2|\zeta(\mu)|^2}}{2|\zeta(\mu)|}\,.
\end{equation}
Thus
\begin{equation}
\label{AlEx}
\sin\theta(u_{+}(\mu),u_{-}(\mu))=\frac{2|\zeta(\mu)|}{\sqrt{1+2|\zeta(\mu)|^2}}\,.
\end{equation}
If \(\mathcal{C}\) is a contractive operator in a Hilbert space and \(e_1,\,e_2\)
are its eigenvectors corresponding to the eigenvalues \(\zeta_1,\,\zeta_2\), then the
angle \(\theta(e_1,\,e_2)\) between the eigenvectors \(e_1,\,e_2\) admits the estimate
from below:
\begin{equation}
\label{EAfb}
\sin\theta(e_1,\,e_2)\geq\left|\frac{\zeta_1-\zeta_2}{1-\zeta_1\overline{\zeta_2}}\right|\,.
\end{equation}
As applied to the contractive matrix \(F(\mu)\),the eigenvalues of which are
\(\zeta_{+}(\mu)=\zeta(\mu),\,\zeta_{-}(\mu)=-\zeta(\mu)\), the estimate \eqref{EAfb}
turns to the inequality
\begin{equation}
\label{AlExAP}
\sin\theta(u_{+}(\mu),u_{-}(\mu))\geq\frac{2|\zeta(\mu)|}{1+|\zeta(\mu)|^2}\,.
\end{equation}
Comparing \eqref{AlEx} and \eqref{AlExAP}, we see that for small \(\zeta(\mu)\)
 the difference between the actual value of \(\sin\theta(u_{+}(\mu),u_{-}(\mu))\)
 and its estimate from below is very small:
 \begin{multline}
 \label{Comp}
 0<\sin\theta(u_{+}(\mu),u_{-}(\mu))-\frac{2|\zeta(\mu)|}{1+|\zeta(\mu)|^2}=\\
 =\frac{2|\zeta(\mu)|}{\sqrt{1+2|\zeta(\mu)|^2}}-\frac{2|\zeta(\mu)|}{1+|\zeta(\mu)|^2}<
 |\zeta(\mu)|^{5}.
 \end{multline}
\end{remark}

\section{Relation between resolvent-based\\ and model-based functional calculi\\ for the operator
\mathversion{bold}%
\(\mathscr{F}_{\!_{\scriptstyle\boldsymbol{\mathbb{R}^{+}}}}\).
\mathversion{normal}
\label{Relb}}

\begin{theorem}
\label{ExFCalM}%
 Let \(h(\zeta)\) be an \(\mathscr{F}_{_{{\scriptstyle\mathbb{R}}^{+}}}\)\,-\,admissible
 function:
 \(h(\zeta)\in\mathfrak{B}_{_{{\scriptstyle\mathscr{F}}_{_{{\mathbb{R}}^{+}}}}}\).\\
Then the operator \(\mathcal{M}_{_{\scriptstyle h(F)}}\) is representable as a strong limit\,%
\footnote{
In particular, the strong limit in the right hand side of \eqref{FAdFOfOp} exists.
} %
of
the family of integrals over the spectrum \(\textup{\large\(\sigma\)}( \mathcal{M}_F)%
\)\textup{:}
 \begin{align}
 \label{FAdFOfOpM}
 \mathcal{M}_{_{\scriptstyle h(F)}}=\\
 =\lim_{\varepsilon\to+0}\,\,&\frac{1}{2\pi{}i}
 \int\limits_{{\textstyle\sigma}(\mathcal{M}_F)}\hspace{-2.0ex}
 h(s)\,\Big(
 R_{_{{\scriptstyle \mathcal{M}}_F}}\!\big(s-\varepsilon{}ie^{i\pi/4}\big)-
 R_{_{{\scriptstyle \mathcal{M}}_F}}\!\big(s+\varepsilon{}ie^{i\pi/4}\big)\Big)\,ds\,,
 \notag
 \end{align}
 where \(R_{_{{\scriptstyle \mathcal{M}}_F}}(z)=
 (z\mathscr{I}- \mathcal{M}_F)^{-1}\) is the resolvent
 of the operator
 \( \mathcal{M}_F\),
 the integral is taken along the interval
 \(\textup{\large\(\sigma\)}( \mathcal{M}_F)=\Big[-\frac{1}{\sqrt{2}}\,e^{i\pi/4},\,
\frac{1}{\sqrt{2}}\,e^{i\pi/4}\Big]\), from the point
\(-\frac{1}{\sqrt{2}}\,e^{i\pi/4}\) to the point
\(\frac{1}{\sqrt{2}}\,e^{i\pi/4}\).
\end{theorem}
\begin{proof}
Since \(R_{_{{\scriptstyle \mathcal{M}}_F}}(z)=\mathcal{M}_{(zI-F)^{-1}}\), the equality holds
\begin{equation*}
 R_{_{{\scriptstyle \mathcal{M}}_F}}\!\big(s-\varepsilon{}ie^{i\pi/4}\big)-
 R_{_{{\scriptstyle \mathcal{M}}_F}}\!\big(s+\varepsilon{}ie^{i\pi/4}\big)=
 \mathcal{M}_{_{\scriptstyle G}}\,,
\end{equation*}
where
\begin{multline}
\label{DiRe}
G(\mu;\,s,\varepsilon)=
\begin{bmatrix}
g_{++}(\mu;\,s,\varepsilon)&g_{+-}(\mu;\,s,\varepsilon)\\[1.0ex]
g_{-+}(\mu;\,s,\varepsilon)&g_{--}(\mu;\,s,\varepsilon)
\end{bmatrix}=
\\[1.5ex]
=\Big(\big(s-\varepsilon{}ie^{i\pi/4}\big)I-F(\mu)\Big)^{-1}-
\Big(\big(s+\varepsilon{}ie^{i\pi/4}\big)I-F(\mu)\Big)^{-1}\,.
\end{multline}
(\(\mathcal{M}_{_{\scriptstyle G}}\) is the multiplication operator on the matrix function \(G\)
of the variable \(\mu\). The variables \(s,\,\varepsilon\) are parameters.)
According to the rule \eqref{WoEx}  of calculating of a function of the matrix~\(F(\mu)\),
\begin{gather}
\label{DeReRes}
\big(zI-F(\mu)\big)^{-1}=\hspace{0.70\linewidth}
\\[1.0ex]
=\begin{bmatrix}
\frac{(z-\zeta_{+}(\mu))^{-1}+(z-\zeta_{-}(\mu))^{-1}}{2}\phantom{f_{-+}(\mu)\ }&%
\frac{(z-\zeta_{+}(\mu))^{-1}-(z-\zeta_{-}(\mu))^{-1}}{2\zeta(\mu)}f_{+-}(\mu)\ \\[2.5ex]
\frac{(z-\zeta_{+}(\mu))^{-1}-(z-\zeta_{-}(\mu))^{-1}}{2\zeta(\mu)}f_{-+}(\mu)&
\frac{(z-\zeta_{+}(\mu))^{-1}+(z-\zeta_{-}(\mu))^{-1}}{2}\phantom{f_{-+}(\mu)}
\end{bmatrix}
\notag
\end{gather}
Substituting the expressions \(z=s-\varepsilon{}ie^{i\pi/4}\) and \(z=s+\varepsilon{}ie^{i\pi/4}\)
into \eqref{DeReRes}, we obtain from \eqref{DiRe}
\begin{multline*}
g_{++}(\mu;\,s,\varepsilon)=g_{--}(\mu;\,s,\varepsilon)=\\
= \frac{1}{2}\Big(s-\varepsilon{}ie^{i\pi/4}-\zeta_{+}(\mu))^{-1}+(s-\varepsilon{}ie^{i\pi/4}-\zeta_{-}(\mu))^{-1}-\\
\phantom{\frac{1}{2}}-(s+\varepsilon{}ie^{i\pi/4}-\zeta_{+}(\mu))^{-1}-(s+\varepsilon{}ie^{i\pi/4}-\zeta_{-}(\mu))^{-1}\Big)\,,
\end{multline*}
\begin{multline*}
g_{+-}(\mu;\,s,\varepsilon)=\\
= \frac{1}{2\zeta(\mu)}\Big(s-\varepsilon{}ie^{i\pi/4}-\zeta_{+}(\mu))^{-1}-(s-\varepsilon{}ie^{i\pi/4}-\zeta_{-}(\mu))^{-1}-
\phantom{f_{\pm}(\mu)}\\
\phantom{\frac{1}{2}}-(s+\varepsilon{}ie^{i\pi/4}-\zeta_{+}(\mu))^{-1}+(s+\varepsilon{}ie^{i\pi/4}-\zeta_{-}(\mu))^{-1}\Big)
f_{+ -}(\mu)\,,
\end{multline*}
\begin{multline*}
g_{-+}(\mu;\,s,\varepsilon)=\\
= \frac{1}{2\zeta(\mu)}\Big(s-\varepsilon{}ie^{i\pi/4}-\zeta_{+}(\mu))^{-1}-(s-\varepsilon{}ie^{i\pi/4}-\zeta_{-}(\mu))^{-1}-
\phantom{f_{\pm}(\mu)}\\
\phantom{\frac{1}{2}}-(s+\varepsilon{}ie^{i\pi/4}-\zeta_{+}(\mu))^{-1}+(s+\varepsilon{}ie^{i\pi/4}-\zeta_{-}(\mu))^{-1}\Big)
f_{-+}(\mu)\,.
\end{multline*}
Thus the entries of the matrix in the right hand side of \eqref{FAdFOfOp} take the form
\begin{math}
\frac{1}{2\pi{}i}\int\limits_{{\textstyle\sigma}(\mathcal{M}_F)}\hspace{-2.0ex}
 h(s)g(\mu;\,s,\varepsilon)\,ds\,,
\end{math}
where \(g(\mu;\,s,\varepsilon)\) is one of the four
 functions \(g_{++}(\mu;\,s,\varepsilon)\),\,\(g_{+-}(\mu;\,s,\varepsilon)\),\,\(g_{-+}(\mu;\,s,\varepsilon)\)\,,
 \(g_{--}(\mu;\,s,\varepsilon)\).\\[2.0ex]
  Studying the limiting behavior of the integrals
 \begin{math}
\frac{1}{2\pi{}i}\int\limits_{{\textstyle\sigma}(\mathcal{M}_F)}\hspace{-2.0ex}
 h(s)g(\mu;\,s,\varepsilon)\,ds\,
\end{math}
as \(\varepsilon\to+0\),  we can ignore the factors \(f_{+-}(\mu)\),  \(f_{-+}(\mu)\) in the expressions
for \(g_{+-}(\mu;\,s,\varepsilon)\),\, \(g_{-+}(\mu;\,s,\varepsilon)\). These factors depend neither on \(s\),
nor on \(\varepsilon\). So we can curry these factors out the integrals. We can also ignore the dependence
of the values \(\zeta_{+}(\mu), \zeta_{+}(\mu)\) on the variable \(\mu\). We only have to assume that
\(\zeta_{+}=-\zeta_{+}=\zeta\), where \(\zeta\) is an arbitrary point of the spectrum
\(\textup{\large\(\sigma\)}( \mathcal{M}_F)\). So, we consider the integrals
\begin{subequations}
\label{SIi}
\begin{align}
\label{SIip}
I_{p}(h;\,\zeta,\varepsilon)&=\frac{1}{2\pi{}i}\int\limits_{{\textstyle\sigma}(\mathcal{M}_F)}\hspace{-2.0ex}
 h(s)p(\zeta;\,s,\varepsilon)\,ds,\\
 \label{SIiq}
I_{q}(h;\,\zeta,\varepsilon)&=\frac{1}{2\pi{}i}\int\limits_{{\textstyle\sigma}(\mathcal{M}_F)}\hspace{-2.0ex}
 h(s)q(\zeta;\,s,\varepsilon)\,ds\,,
\end{align}
\end{subequations}
where
\begin{multline*}
p(\zeta;\,s,\varepsilon)
= \frac{1}{2}\Big(s-\varepsilon{}ie^{i\pi/4}-\zeta)^{-1}+(s-\varepsilon{}ie^{i\pi/4}+\zeta)^{-1}-\hfill\\
\hfill-(s+\varepsilon{}ie^{i\pi/4}-\zeta)^{-1}-(s+\varepsilon{}ie^{i\pi/4}+\zeta)^{-1}\Big)\,,
\end{multline*}
\vspace{-5.0ex}
\begin{multline*}
q(\zeta;\,s,\varepsilon)
= \frac{1}{2\zeta}\Big(s-\varepsilon{}ie^{i\pi/4}-\zeta)^{-1}-(s-\varepsilon{}ie^{i\pi/4}+\zeta)^{-1}-
\hfill\\
\hfill-(s+\varepsilon{}ie^{i\pi/4}-\zeta)^{-1}+(s+\varepsilon{}ie^{i\pi/4}+\zeta)^{-1}\Big)\,.
\end{multline*}
It is clear that
\begin{subequations}
\begin{align}
g_{++}(\mu;\,s,\varepsilon)&=p(\zeta(\mu);\,s,\varepsilon)\,,\\
g_{--}(\mu;\,s,\varepsilon)&=p(\zeta(\mu);\,s,\varepsilon)\,,\\
g_{+-}(\mu;\,s,\varepsilon)&=q(\zeta(\mu);\,s,\varepsilon)f_{+-}(\mu)\,,\\
g_{-+}(\mu;\,s,\varepsilon)&=q(\zeta(\mu);\,s,\varepsilon)f_{-+}(\mu)\,.
\end{align}
\end{subequations}
According to Lemma \ref{StrConL}, the equality \ref{FAdFOfOp} will be probed if we prove that for almost every \(\mu\in[0,\,\infty)\)
\begin{equation}
\label{TLiR}
h(F(\mu))=\lim_{\varepsilon\to+0}
\begin{bmatrix}
I_{p}(h;\,\zeta(\mu),\varepsilon)\phantom{f_{-+}(\mu)\ \ }&I_{q}(h;\,\zeta(\mu),\varepsilon)f_{+-}(\mu)\\[1.5ex]
I_{q}(h;\,\zeta(\mu),\varepsilon)f_{-+}(\mu)\ \ &I_{p}(h;\,\zeta(\mu),\varepsilon)\phantom{f_{+-}(\mu)}
\end{bmatrix}\,,
\end{equation}
and moreover the matrix functions in the right hand side of \eqref{TLiR} are bounded for \(\mu\in[0,\,\infty)\)
uniformly with respect to \(\varepsilon>0\):
\begin{multline}
\label{TLiRu}
\sup_{\mu\in[0,\infty)}\,\left\|\,
\begin{bmatrix}
I_{p}(h;\,\zeta(\mu),\varepsilon)\phantom{f_{-+}(\mu)\ \ }&I_{q}(h;\,\zeta(\mu),\varepsilon)f_{+-}(\mu)\\[1.5ex]
I_{q}(h;\,\zeta(\mu),\varepsilon)f_{-+}(\mu)\ \ &I_{p}(h;\,\zeta(\mu),\varepsilon)\phantom{f_{+-}(\mu)}
\end{bmatrix}\,\right\|\leq{}C<\infty\,\\
\textup{for every}\ \ \varepsilon>0\,,
\end{multline}
where \(C<\infty\) is a value which does not depend on \(\varepsilon>0\).

The relations \eqref{TLiR}, \eqref{TLiRu} are consequences of the expression \eqref{WoEx} for \(h(F(\mu))\)
end the following Lemma:
\begin{lemma}
\label{Jptl}
Let \(h(s)\) be a summable  function on the interval \(\textup{\large\(\sigma\)}( \mathcal{M}_F)=
\Big[-\frac{1}{\sqrt{2}}\,e^{i\pi/4},\,
\frac{1}{\sqrt{2}}\,e^{i\pi/4}\Big]\):
\begin{math}
\int\limits_{\sigma( \mathcal{M}_F)}|h(s)|\,|ds|<\infty.
\end{math}
\begin{enumerate}
\item[\textup{1.}]
If \(h\) satisfies the condition
\begin{equation}
C_p(h)\stackrel{\textup{\tiny def}}{=}\underset{s\in\sigma(\mathcal{M}_F)}{\textup{ess\,sup}}\,
\frac{|h(s)+h(-s)|}{2}<\infty,
\end{equation}
then the integrals \(I_{p}(h;\,\zeta,\varepsilon)\) are bounded for \(\zeta\in\textup{\large\(\sigma\)}( \mathcal{M}_F)\)
uniformly with respect to  \(\varepsilon>0\):
\begin{equation}
\label{UBop}
|I_{p}(h;\,\zeta,\varepsilon)|\leq{}C_p(h),\quad \forall\,\zeta\in\textup{\large\(\sigma\)}( \mathcal{M}_F),
\ \ \forall\,\varepsilon>0\,,
\end{equation}
and for almost every \(\zeta\in\textup{\large\(\sigma\)}( \mathcal{M}_F)\) the limiting relation
\begin{equation}
\label{LRIp}
\lim_{\varepsilon\to+0}I_{p}(h;\,\zeta,\varepsilon)=\frac{h(\zeta)+h(-\zeta)}{2}
\end{equation}
holds.
\item[\textup{2.}]
If \(h\) satisfies the condition
\begin{equation}
C_q(h)\stackrel{\textup{\tiny def}}{=}\underset{s\in\sigma(\mathcal{M}_F)}{\textup{ess\,sup}}\,
\frac{|h(s)-h(-s)|}{2|s|}<\infty,
\end{equation}
then the integrals \(I_{q}(h;\,\zeta,\varepsilon)\) are bounded for \(\zeta\in\textup{\large\(\sigma\)}( \mathcal{M}_F)\)
uniformly with respect to  \(\varepsilon>0\):
\begin{equation}
\label{UBoq}
|I_{q}(h;\,\zeta,\varepsilon)|\leq{}4C_q(h),\quad \forall\,\zeta\in\textup{\large\(\sigma\)}( \mathcal{M}_F),
\ \ \forall\,\varepsilon>0\,,
\end{equation}
and for almost every \(\zeta\in\textup{\large\(\sigma\)}( \mathcal{M}_F)\) the limiting relation
\begin{equation}
\label{LiResd}
\lim_{\varepsilon\to+0}I_{q}(h;\,\zeta,\varepsilon)=\frac{h(\zeta)-h(-\zeta)}{2\zeta}
\end{equation}
holds.
\end{enumerate}
\end{lemma}
The reference to Lemma \ref{Jptl} finalizes the proof of Theorem \ref{ExFCal}.
\end{proof}
\begin{proof}[Proof of Lemma \ref{Jptl}] The singular integrals \eqref{SIi}, which appear in Lemma, are cognate
to the Poisson integral. However in \eqref{SIi} the integration is performed over the slopping interval
\(\scriptstyle\Big[-\frac{1}{\sqrt{2}}\,e^{i\pi/4},\,
\frac{1}{\sqrt{2}}\,e^{i\pi/4}\Big]\) and an evaluation point \(\zeta\) also belongs to this interval.
To pass to the more conventional setting, where the interval of integration is real, and the evaluation point is real,
we change variables. We put
\begin{subequations}
\label{CVa}
\begin{gather}
\label{CVav}
s=e^{i\pi/4}\rho,\quad\zeta=e^{i\pi/4}r,\qquad -1/\sqrt{2}\leq{}r,\,\rho\leq1/\sqrt{2},\\
\label{CVaf}
h(e^{i\pi/4}\rho)=H(\rho)\,.
\end{gather}
\end{subequations}
Then the integrals \eqref{SIip} takes the form
\begin{subequations}
\label{SIiNV}
\begin{align}
\label{SIipJ}
J_{P}(H;\,r,\varepsilon)&=\int\limits_{-1/\sqrt{2}}^{1/\sqrt{2}}
 H(\rho)\,\frac{P(r,\,\rho;\,\varepsilon)+P(r,\,-\rho;\,\varepsilon)}{2}\,d\rho,\\
 \intertext{and the integral \eqref{SIiq} takes the form}
 \label{SIiqJ}
J_{Q}(H;\,r,\varepsilon)&=\int\limits_{-1/\sqrt{2}}^{1/\sqrt{2}}
 H(\rho)\,\frac{P(r,\,\rho;\,\varepsilon)-P(r,\,-\rho;\,\varepsilon)}{2re^{i\pi/4}}\,d\rho\,,
\end{align}
where \(P(r,\,\rho;\,\varepsilon)\) is the Poisson kernel:
\end{subequations}
\begin{equation}
\label{PoKe}
P(r,\,\rho;\,\varepsilon)=\frac{1}{\pi}\cdot\frac{\varepsilon}{(r-\rho)^2+{\varepsilon}^{2}}\,,
\quad -\infty<r,\,\rho<\infty,\ \ \varepsilon>0.
\end{equation}
Splitting the integral in the right hand side of \eqref{SIip} into the sum of two integrals end changing
\(\rho\to-\rho\) in the second of these integrals, we transform the expression for \(J_{P}(H;\,r,\varepsilon)\) to
the form
\begin{equation}
\label{SIipT}
J_{P}(H;\,r,\varepsilon)=\int\limits_{-1/\sqrt{2}}^{1/\sqrt{2}}
 \frac{H(\rho)+H(-\rho)}{2}\,P(r,\,\rho;\,\varepsilon)\,d\rho
 \end{equation}
 Changing \(\rho\to-\rho\) in the integral in \eqref{SIiq}, we transform the  expression for \(J_{Q}(H;\,r,\varepsilon)\) to
the form
\begin{equation}
\label{SIiqT}
J_{Q}(H;\,r,\varepsilon)=\int\limits_{0}^{1/\sqrt{2}}
 \frac{H(\rho)-H(-\rho)}{2\rho{}e^{i\pi/4}}\,Q(r,\,\rho;\,\varepsilon)\,d\rho\,,
 \end{equation}
 where
 \begin{equation*}
Q(r,\,\rho;\,\varepsilon)= \frac{\rho}{r}(P(r,\,\rho;\,\varepsilon)-
 P(r,\,-\rho;\,\varepsilon))\,,
 \end{equation*}
 or, in more detail,
 \begin{multline}
 \label{KerQ}
\hfill  Q(r,\,\rho;\,\varepsilon)=P(|r|,\rho;\,\varepsilon)\,\frac{4\rho^{2}}{(|r|+\rho)^2+\varepsilon^2}\,,
\hfill \hfill \\
 -\infty<r<\infty,\,-\infty<\rho<\infty,\ \varepsilon>0\,.
 \end{multline}
 Since \(\dfrac{4\rho^{2}}{(|r|+\rho)^2+\varepsilon^2}\leq4\) for \(\rho>0\), the inequality
 \begin{equation}
 \label{Four}
0< Q(r,\,\rho;\,\varepsilon)\leq4\,P(|r|,\rho;\,\varepsilon)
 \end{equation}
 holds for \(-\infty<r<\infty,\,0<\rho<\infty,\ \varepsilon>0\).
 Since \(\int\limits_{-\infty}^{\infty}P(r,\rho;\,\varepsilon)d\rho=1\)
 for every \(r\in(-\infty,\infty)\) and \(\varepsilon>0\), it follows from
 \eqref{SIipT} that
 \begin{equation}
 \label{UBopt}
 |J_{P}(H;\,r,\varepsilon)|\leq
 \underset{\rho\in[-1/\sqrt{2},1/\sqrt{2}]}{\textup{ess\,sup}}\,
\frac{|H(\rho)+H(-\rho)|}{2},\quad -\infty<r<\infty,\,\varepsilon>0\,.
 \end{equation}
 From \eqref{SIiqT} and \eqref{Four} it follows that
 \begin{equation}
 \label{UBoqt}
 |J_{Q}(H;\,r,\varepsilon)|\leq
 4\!\!\!\!\!\underset{\rho\in[-1/\sqrt{2},1/\sqrt{2}]}{\textup{ess\,sup}}\,
\frac{|H(\rho)-H(-\rho)|}{2|\rho|},\quad -\infty<r<\infty,\,\varepsilon>0\,.
 \end{equation}
 Inequality \eqref{UBopt} is the same that the inequality \eqref{UBop},
 inequality \eqref{UBoqt} is the same that the inequality \eqref{UBoq}.

 Let \(f:\,\mathbb{R}\to\mathbb{R}\) be a summable function: \(\int\limits_{\mathbb{R}}|f(\rho)|\,d\rho<\infty\). We recall
 that \emph{the point \(r\in\mathbb{R}\) is said to be the Lebesgue point for the function \(f\)} if
 \begin{equation}
 \label{DLePo}
 \lim_{h\to+0}\frac{1}{h}\int\limits_{r-h}^{r+h}|f(\rho)-f(r)|\,d\rho=0\,.
 \end{equation}
 The following fact is one of the main facts of the Lebesgue integration theory:\\
 \emph{Given a summable function \(f:\,\mathbb{R}\to\mathbb{R}\), then almost every point \(r\in\mathbb{R}\) is
 a Lebesgue point for the function \(f\).}

 We also use the following fundamental fact.\\
 \begin{nonumtheorem}\,\textsf{\textup{(On limiting behavior of the Poisson integral)}}. \\
 Let \(f:\,\mathbb{R}\to\mathbb{R}\) be a summable function, and \(P\) be a Poisson kernel, \eqref{PoKe}.
 Assume that \(r,\,r\in\mathbb{R}\), is a Lebesgue point for \(f\).

 Then
 \begin{equation}
 \label{BbPI}
 \lim_{\varepsilon\to+0}\int\limits_{-\infty}^{\infty}P(r,\rho,\,\varepsilon)\,f(\rho)\,d\rho=f(r)\,
 \end{equation}
 In particular, the equality \eqref{BbPI} holds for almost every \(r\in\mathbb{R}\).
 \end{nonumtheorem}

 We apply this theorem to study the limiting behavior of the integral \(J_{P}(H;\,r,\varepsilon)\), \eqref{SIipT}, as
 \(\varepsilon\to+0\). This is an integral of the form \eqref{BbPI}, where
 \begin{equation*}
 f(\rho)=1/2(H(\rho)+H(-\rho)),\ |\rho|\leq1/\sqrt{2};\ \ f(\rho)=0,\ \ |\rho|>1/\sqrt{2}\,.
 \end{equation*}

 According to the above mentioned theorem,
 \begin{multline}
 \label{LBJ1}
 \lim_{\varepsilon\to+0}J_{P}(H;\,r,\varepsilon)=\frac{H(r)+H(r)}{2}\ \
 \\ \textup{for almost every}\ \
 r\in[-1/\sqrt{2}\,,\,1/\sqrt{2}]\,.
 \end{multline}
 The relation \eqref{LBJ1} is the same that \eqref{LRIp}.

 The integral \eqref{SIiqT} is not a Poisson integral, however the study of the  integral \eqref{SIiqT}
 can by reduced to the study of some Poisson integral. From \eqref{KerQ} it follows that
 \begin{equation}
 \label{RtoP}
 Q(r,\rho;\,\varepsilon)=P(|r|,\rho;\,\varepsilon)+T(r,\rho;\,\varepsilon),
 \end{equation}
where the kernel \(T(r,\rho;\,\varepsilon)\) is of the form
\begin{equation*}
  T(r,\rho;\,\varepsilon)=P(|r|,\rho;\,\varepsilon)\,\frac{(|r|-\rho)(3|r|+\rho)-
 \varepsilon^{2}}{(|r|+\rho)^{2}+\varepsilon^{2}}
 \end{equation*}
 It is clear that
 \begin{equation}
 \label{KerT}
 \big|T(r,\rho;\,\varepsilon)\big|\leq{}3P(|r|,\rho;\,\varepsilon)\frac{\big||r|-\rho\big|}{|r|+\rho}+
 P(|r|,\rho;\,\varepsilon)\frac{\varepsilon^{2}}{r^{2}}
 \end{equation}
 From \eqref{SIiqT}, \eqref{RtoP} and \eqref{KerT} it follows that
 \begin{equation}
 \label{SpIQ}
 J_{Q}(H;\,r,\varepsilon)=\int\limits_{0}^{1/\sqrt{2}}
 \frac{H(\rho)-H(-\rho)}{2\rho{}e^{i\pi/4}}\,P(|r|,\,\rho;\,\varepsilon)\,d\rho+J_{T}(H;\,r,\varepsilon)\,,
 \end{equation}
 where
 \begin{equation*}
 J_{T}(H;\,r,\varepsilon)=\int\limits_{0}^{1/\sqrt{2}}
 \frac{H(\rho)-H(-\rho)}{2\rho{}e^{i\pi/4}}\,T(|r|,\,\rho;\,\varepsilon)\,d\rho\,,
 \end{equation*}
 and according to \eqref{KerT}
 \begin{multline*}
 \big|J_{T}(H;\,r,\varepsilon)\big|\leq
 \underset{\rho\in[0,1/\sqrt{2}]}{\textup{ess\,sup}}\,
\frac{|H(\rho)-H(-\rho)|}{2|\rho|}\times\\
\times3\bigg(\int\limits_{-\infty}^{\infty}\!\!P(|r|,\,\rho;\,\varepsilon)\,\frac{\big||r|-\rho\big|}{|r|+\rho}\,d\rho\,+
\frac{\varepsilon^{2}}{r^{2}}\bigg)\,.
\end{multline*}
For each fixed \(r\not=0\), the function \(\varphi(\rho)=\frac{||r|-\rho|}{|r|+\rho}\) is a continuous\
bounded function of \(\rho\) for \(\rho\in[-\infty,\infty]\)
 vanishing at the point \(\rho~=~r\). According to the elementary property of the Poisson integral,
 \begin{equation*}
 \lim_{\varepsilon\to+0}
 \int\limits_{-\infty}^{\infty}\!\!P(|r|,\,\rho;\,\varepsilon)
 \frac{\big||r|-\rho\big|}{|r|+\rho}\,d\rho=0
 \end{equation*}
 for every \(r\not=0\). Thus,
 \begin{equation}
 \label{Remtz}
 \lim_{\varepsilon\to+0}J_{T}(H;\,r,\varepsilon)=0\ \ \textup{for every}\ r\not=0\,.
 \end{equation}
 According to Theorem on limiting behavior of Poisson integral,
 \begin{multline}
 \label{Bas}
 \lim_{\varepsilon\to+0}\int\limits_{0}^{1/\sqrt{2}}
 \frac{H(\rho)-H(-\rho)}{2\rho{}e^{i\pi/4}}\,P(|r|,\,\rho;\,\varepsilon)\,d\rho=
 \frac{H(r)-H(-r)}{2r{}e^{i\pi/4}}\\[-1.0ex]
 \textup{for almost every}\ r\in[-1/\sqrt{2}\,,\,1/\sqrt{2}]\,.
 \end{multline}
 From \eqref{SpIQ}, \eqref{Remtz} and \eqref{Bas} it follows that
 \begin{multline}
 \label{FiLiRe}
 \lim_{\varepsilon\to+0}J_{Q}(H;\,r,\varepsilon)=\frac{H(r)-H(-r)}{2r{}e^{i\pi/4}}\\[0.0ex]
 \textup{for almost every}\ r\in[-1/\sqrt{2}\,,\,1/\sqrt{2}]\,.
 \end{multline}
 The relation \eqref{FiLiRe} is the same that the relation \eqref{LiResd}. Lemma \ref{Jptl} is proved.
 \end{proof}
 \begin{remark} \label{pkq}\hspace{3.0ex}\\[-3.0ex]
 \begin{enumerate}
\item[{\textup{1}.}]
 If the function \(h\) is continuous on the interval \([-e^{i\pi/4}\,,\,e^{i\pi/4}]\) and vanishes at the endpoints
  \(\zeta=\pm{}e^{i\pi/4}/\sqrt{2}\) of this interval, then the limit in \eqref{LRIp} is uniform with respect to \(\zeta\).
This follows from the elementary properties of the Poisson integral.
\item[{\textup{2}.}]
Let us assume \emph{moreover} that the function \((h(\zeta)-h(-\zeta))/\zeta\) also is continuous on the interval
\([-e^{i\pi/4}\,,\,e^{i\pi/4}]\). The integral \eqref{SIiqT} is not a Poisson integral. However the kernel
\(Q(r,\rho;\,\varepsilon)\), like the Poisson kernel \(P(r,\rho;\,\varepsilon)\),  possesses the properties of an approximative identity. In particular,
\begin{equation}
\label{ExCa}
\int\limits_{0}^{\infty}Q(r,\rho;\,\varepsilon)\,d\rho=1\ \ \textup{for every } r\in(-\infty\,,\,\infty),\, \varepsilon>0\,.
\end{equation}
 Therefore the limit in \eqref{LiResd} is uniform with respect to \(\zeta\).
 \end{enumerate}
 \end{remark}
  In view of this remark, \textup{Theorem \ref{ExFCalM}} can be supplemented.
 \begin{remark}
 \label{ComTe}  Assume that the function \(h\) possesses the properties:
 \begin{enumerate}
\item[{\textup{1}.}]
The function \(h(\zeta)\) is continuous on the interval
\(\Big[\frac{-e^{i\pi/4}}{\sqrt{2}}\,,\frac{\,e^{i\pi/4}}{\sqrt{2}}\Big]\).
\item[{\textup{2}.}] The function \(\frac{h(\zeta)-h(-\zeta)}{\zeta}\) is continuous at the point \(\zeta=0\).
\item[{\textup{3}.}] \(h\Big(\pm\frac{e^{i\pi/4}}{\sqrt{2}}\Big)=0\)\,.
\end{enumerate}
Then the limit in \eqref{FAdFOfOpM}     also exists in the uniform operator topology.
\end{remark}

\vspace*{5.0ex}
\noindent
\begin{minipage}[h]{0.45\linewidth}
Victor Katsnelson\\[0.2ex]
Department of Mathematics\\
The Weizmann Institute\\
Rehovot, 76100, Israel\\[0.1ex]
e-mail:\\
{\small\texttt{victor.katsnelson@weizmann.ac.il}}
\end{minipage}\\

\vspace*{3.0ex} \noindent
\begin{minipage}[h]{0.45\linewidth}
Bobik Katzenbeisser\\[0.2ex]
Tsar Baalei Haim, 76100, Israel\\[0.1ex]
e-mail:\\
\texttt{bobik.katzenbeisser@gmail.com}
\end{minipage}

\begin{thebibliography}{MoMo1}
\bibitem[Car]{Car} \textsc{Carleman,\,T.} \textit{Sur les {\'e}quations
integrales singuli{\`e}res a noy{\`a}u r{\'e}el et
symm{\'e}trique}. (Uppsala Universitets {\AA}rksskrift 1923.
Matematik och Naturvetenskap, 3). Uppsala, Sweden: Almqvist \&
Wiksell, 1923. 228 S.
\bibitem[Dir]{Dir} \textsc{Dirac,\,P.A.M.} \textit{The Principles of Quantum Mechanics.}
  Clarendon Press. Oxford. First Edition\,-\, 1930. Forth Edition\,-\,1958.
\bibitem[DuSch]{DuSch} \mbox{\textsc{Dunford,\,N.,
Schwartz,\,J.T.}\ \textit{Linear Operators. Part~I.}}
\textit{General Theory.} \mbox{(Pure and
Applied Mathematics, Vol. 7).} \mbox{Interscience Publishers, %
New\,York\,--\,London,\,1958,} xiv+858pp.
\bibitem[GoKr]{GoKr} \textsc{Gohberg,\,I.C.,\,\,Kre\u{i}n,\,M.G.}
\textit{Introduction to the Theory of Linear Nonselfadjoint
Operators.} AMS, Providence, RI, 1969.
\bibitem[KaMa1]{KaMa1} \textsc{Katsnelson,\,V,\,Machluf,\,R.}
\textit{The truncated Fourier operator.}{I}. ArXiV:0901.2555.
\bibitem[KaMa2]{KaMa2} \textsc{Katsnelson,\,V,\,Machluf,\,R.}
\textit{The truncated Fourier operator.}{II}. ArXiV:0901.2709.
\bibitem[KaMa3]{KaMa3} \textsc{Katsnelson,\,V,\,Machluf,\,R.}
\textit{The truncated Fourier operator.}{IV}. ArXiV:0902.0568.
\bibitem[Mag]{Mag} \textsc{Magnus,\, W.}\textit{On the spectrum of Hilbert matrix.}
American Math. Journ., \textbf{72}:4, (1950), pp.\,699\,-\,704.
\bibitem[Mau]{Mau} \textsc{Mautner,\, F.I.} \textit{On eigenfunction expansions.}
 Proc. Nat. Acad. Sci. U. S. A. \textbf{39}, (1953). 49--53.\\
 Russuan translation:\\
\textsc{Маутнер, Ф.И.} \textit{О разложениях по собственным
функциям.} Успехи Мат. Наук, \textbf{10}:4 (1955), 127\,--\,132.
\bibitem[Pov1]{Pov1} \textsc{Повзнер,\,А.Я.} \textit{О разложении по
собственным функциям уравнения Шредингера.} Доклады АН СССР,
{\textbf79} (1951) , 193--196. (In Russian). \\ {}
 [\textsc{Povzner,
A.Ya.} \textit{On the differentiation of the spectral function of
the Schrцdinger equation.}  Doklady Akad. Nauk SSSR (N.S.)
{\textbf79}, (1951). 193--196.]
\bibitem[Pov2]{Pov2} \textsc{Повзнер,\,А.Я.} \textit{О разложении
произвольных функций по собственным функциям оператора
\(-\Delta{}u+cu\).}
Матем. Сборник, \textbf{32}:1 (\textbf{74}), (1953), 109--156.\\
English transl.:\\
 \textsc{Povzner,\,A.Ya.} \textit{The expansion of arbitrary
 functions in terms of eigenfunctions of the operator \(-\Delta{}u+cu.\)}
 Amer.\,Math. \,Soc.\,Transl.\, Ser.{\textbf2}, Vol.\,\textbf60\,,
 (1967), 1\,--\,49.
\bibitem[Ros]{Ros} \textsc{Rosenblum,\,M.}\textit{On the Hilbert matrix, II.}
\textit{Proc.\,Amer.\,Math.\,Soc.}, \textbf{9}:4 (1958),
pp.\,581\,-\,585.
\bibitem[Sl3]{Sl23} \textsc{Slepian,\,D.} \textit{On bandwidth.}
Proc. IEEE  \textbf{64}:3 (1976), 292--300.
\bibitem[Sl4]{Sl4} \textsc{Slepian,\,D.} \textit{Some comments on Fourier analysis, uncertainty
and modelling.} SIAM Review, \textbf{25}:3, 1983, 379\,-\,393.
\bibitem[SzNFo]{SzNFo} \textsc{Sz.-Nagy,\,B.} and \textsc{C.\,Foias.}
\textsl{Analyse Harmonique des Op\'erateurs de l'espace  de
Hilbert} (French). Masson and Acad\'emiae Kiado, 1967.
 English
transl.:\\ %
 \textsl{Harmonic Analysis of Operators in Hilbert
Space}. North Holland, Amsterdam 1970.
\bibitem[Wey1]{Wey1} \textsc{Weyl,\,H.} \textit{Singul\"are
Integralgleichungen mit besoderer Ber\"ucksichtigung des
Fourierischen Integraltheorems.} Dissertation, G\"ottingen (1908).
Reprinted in \cite{Wey2}, pp.\,1\,-\,87.
\bibitem[Wey2]{Wey2} \textit{Hermann Weyl Gesammelte
Abhandlungen. Band I.} \\ Springer-Verlag.
Berlin\(\boldsymbol\cdot\)Heidelberg\(\boldsymbol\cdot\)New\,York,
1968. 698 pp.
\bibitem[Wien]{Wien} \textsc{Wiener,\,N.} \textit{The Fourier Integral and certain of its
Applications}. Cambridge University Press. Cambridge, 1933.
\end{thebibliography}
\end{document}